\documentclass [11pt]{amsart}
\usepackage {amsmath, amssymb, amscd, mathrsfs, url, verbatim}
\usepackage[text={6.5in,9in},centering,letterpaper,dvips]{geometry}
\usepackage{color,multirow,latexsym,graphicx,comment,subfig}
\usepackage[all]{xy}

\newtheorem {theorem}{Theorem}
\newtheorem {lemma}[theorem]{Lemma}
\newtheorem {proposition}[theorem]{Proposition}

\newtheorem {definition}[theorem]{Definition}

\theoremstyle{remark}

\newtheorem {remark}[theorem]{Remark}

\numberwithin{equation}{section}
\numberwithin{theorem}{section}

\newcommand{\comments}[1]{}

\newcommand{\Spinc}{$\mathrm{Spin}^c\;$}

\DeclareMathOperator{\CPo}{\mathbb{C}P^1}
\DeclareMathOperator{\CP2}{\mathbb{C}P^2}
\DeclareMathOperator{\barCP2}{\overline{\mathbb{C}P^2}}
\DeclareMathOperator{\T}{\mathcal{T}}
\DeclareMathOperator{\St}{\mathcal{S}}
\DeclareMathOperator{\Z}{\mathbb{Z}}

\specialcomment{cagri}{%
  \begingroup\color{blue}  Cagri \; $\blacktriangleright$ \;}{%
  \endgroup}
\specialcomment{laura}{%
  \begingroup\color{red} Laura \; $\blacktriangleright$ \;}{%
  \endgroup}
\subjclass[2010]{57R17, 57R57, 32Q28, 32Q60}
\title{Surgery Along Star Shaped Plumbings and Exotic Smooth Structures on $4$-Manifolds}
\author{\c{C}a\u gr\i\; Karakurt}
\author{Laura \; Starkston}
\address{Department of Mathematics, Bogazici University, TR-34342 Bebek, Istanbul, TURKEY}
\email{cagri.karakurt@boun.edu.tr}
\address{Department of Mathematics,
The University of Texas at Austin 
2515 Speedway, RLM 8.100 Stop C1200,
Austin, TX 78712-1202, USA}%
\email{lstarkston@math.utexas.edu}
\date{}

\excludecomment{cagri}
\excludecomment{laura}

\begin{document}
\maketitle

\begin{abstract}
We define a new $4$-dimensional symplectic cut and paste operation which is analogous to Fintushel and Stern's rational blow-down. We use this operation to produce multiple constructions of symplectic smoothly exotic complex projective space blown-up eight times, seven times, and six times. We also show how this operation can be used in conjunction with knot surgery to construct an infinite family of minimal exotic smooth structures on the complex projective space blown-up seven times.
\end{abstract}

\section{Introduction}


In \cite{FS}, Fintushel and Stern introduced a cut and paste operation for $4$-manifolds called rational blow-down. They used it to compute the Donaldson polynomial of the logarithmic transforms of the elliptic surfaces. Since then, rational blow-down operation has proven to be very useful in $4$-dimensional topology. It was a useful constructional tool in the exotic copies of blown-up complex projective plane \cite{P, SS, PSS, FS2, Mic}. It could be used to construct symplectic manifolds \cite{Sym1, Sym2, GS1}, and in the presence of a certain Lefschetz fibration structure one can re-interpret it as a monodromy substitution \cite{EG, EMV}. 

The purpose of the present paper is to define a new cut and paste operation, called \emph{star surgery}, which is a strong generalization of Fintushel-Stern's rational blow-down. Just like rational blow-down, our operation reduces $b_2^-$ of the manifold to which it is applied. Moreover it can be performed symplectically and can be seen as a monodromy substitution.  

The original rational blow-down operation, and its generalizations \cite{P1, SSW, BS}, amount to removing the neighborhood of a union of spheres which intersect according to a particular plumbing tree and re-gluing a rational ball which has the same boundary as this neighborhood. Our star surgery operation is similarly defined. First identify (symplectic) spheres which intersect according to a star-shaped graph with a negativity condition on the central vertex. The star surgery operation cuts out a neighborhood of these spheres and replaces it by an alternate symplectic filling of the induced contact boundary. It is shown in \cite{Sta2} that these alternate fillings always have smaller Euler characteristic than the neighborhood of spheres and are negative definite. Unlike the rational blow-down we do not require the alternate filling to be a rational homology ball. This greatly generalizes the set of configurations of spheres which we can consider for these operations.


By reinterpreting this operation as a monodromy substitution, one can show that some star surgeries are obtained by a sequence of rational blow-downs. The spheres to rationally blow-down after the first step in the sequence are not all visible in the original configuration, and would be difficult to find. The star surgery bypasses the need to find these spheres by performing a single symplectic cut-and-paste operation that performs the entire sequence of rational blow-downs simultaneously. 

However, there are other star surgeries which are inequivalent to any sequence of symplectic rational blow-downs. An example of such a star surgery was proven in \cite{Sta2}. It is expected that many of these star surgery operations cannot be obtained from sequences of rational blow-downs. This contrasts with the operations one would obtain by replacing linear plumbings of spheres by alternate fillings, which were shown to all be equivalent to sequences of rational blow-downs in \cite{BO}.

Using star surgery, we construct many examples of exotic 4-manifolds. These constructions involve two steps. First we must find a configuration of symplectic spheres inside a well understood 4-manifold. In our examples we do this by looking at blow-ups of elliptic fibrations $E(1)$ using varying types of singular fibers to find symplectic spheres with the required intersection data. We explicitly construct many elliptic fibrations by blowing up various Lefschetz pencils on $\CP2$. Then we apply the star surgery operation which replaces this neighborhood of spheres with the smallest symplectic filling of the induced contact boundary Seifert fibered space. By keeping track of the homology classes of all of the spheres in the elliptic fibration, we are able to compute the small perturbation Seiberg-Witten invariants of the resulting manifold.

Using this technique, we construct a minimal symplectic $4$-manifold which is an exotic copy of  $\CP2 \# 8 \barCP2 $.
\begin{theorem}\label{theo:main}
There is a minimal symplectic $4$-manifold $X$ which is homeomorphic but not diffeomorphic to $\CP2 \# 8 \barCP2 $ which is obtained by a star surgery. The symplectic Kodaira dimension of $X$ is $2$.
\end{theorem}

We push our techniques further using different examples of star surgery operations. These examples yield exotic (potentially non-minimal) symplectic copies of $\CP2 \# k \barCP2 $, for $k=6,7$. 
\begin{theorem}\label{theo:main3}
There are constructions of symplectic exotic copies of $\CP2 \# 7 \barCP2 $ and $\CP2 \# 6 \barCP2$ obtained by performing star surgery operations on blow-ups of $E(1)$.
\end{theorem}

Other star surgery operations, including the star surgery which is known to be inequivalent to any sequence of rational blow-downs, can be used for similar constructions. In particular, we show that this star surgery can be used to construct an exotic $\CP2\# 8\barCP2$, and related star surgeries can be used to improve these constructions to manifolds with $b_2^-=6,7$.

While the star surgery operations are inspired by symplectic topology, they can also be used smoothly in the absence of a symplectic structure. By using star surgery after performing Fintushel and Stern's knot surgery in a double node neighborhood \cite{FS2} (which destroys the symplectic structure), we prove the following result.

\begin{theorem}\label{theo:main2}
For every $n\geq2$ there exist smooth minimal mutually non-diffeomorphic $4$-manifolds  $Y_n$ which are all homeomorphic to $\CP2 \# 7\barCP2$. These manifolds are obtained by a star surgery.
\end{theorem}

Examples of  (minimal)  exotic copies of $\CP2 \# k \barCP2 $, for $k=6,7,8$ have previously been constructed using the rational blow-down technique. The exotic structure was detected by calculating the effect of the rational blow-down on the Seiberg-Witten invariants \cite{P, SS, FS2, Mic}. The effect of star surgery on Seiberg-Witten invariants is similar to the effect of rational blow-down \cite{Mic}. The main reason is that the boundary of the star shaped configuration  is an $L$-space. In other words the Monopole Floer homology of the boundary of the neighborhood of these configurations of spheres is the simplest group it could be.

Finding exotic copies of $\CP2 \# k \barCP2 $ for small $k$ is a problem which has been studied for many years. In the 1980s, gauge theoretic techniques were used to distinguish Dolgachev surfaces from $\CP2 \# 9 \barCP2$ \cite{D, FM} and the Barlow surface from $\CP2\# 8\barCP2$ \cite{Ko}. Significant progress was made using the rational blow-down to construct an exotic $\CP2 \#  k \barCP2$ for $k\geq 5$, \cite{P, SS, PSS, FS2, Mic}. Later this was improved to $k\geq 2$ using different techniques \cite{ AP1, AP2, A1, ABP, BK, FS5, FS6}.  Because these star surgery operations greatly increase the possible configurations of surfaces which can be cut out and replaced, we hope that more star surgery constructions will be found and can be used to improve this bound or exhibit other new phenomena in smooth 4-manifold topology.

The organization is as follows: In section \ref{s:starsurgerydefinitions} we define our star surgery operation, and describe the explicit examples which we will use in constructions of exotic 4-manifolds. In section \ref{s:fillingproperties}, we determine properties of these star surgery operations by computing algebraic topological invariants of the fillings. In section \ref{s:ellipticfibrations}, we construct three explicit elliptic fibrations which we will use to embed configurations of symplectic spheres to perform star surgery on. Theorems \ref{theo:main} and \ref{theo:main3} are proven in section \ref{s:exoticconstructions}, where we use star surgery to construct manifolds and compute their homeomorphism invariants, Kodaira dimension, and Seiberg-Witten invariants. Finally, theorem \ref{theo:main2} is proven in section \ref{s:smoothexample}, by using knot surgery and star surgery together.

\section*{Acknowledgments}
We would like to thank Kouichi Yasui and Tian-Jun Li for helpful e-mail correspondences.  In the course of this work, the first author  was supported by  the National Science Foundation FRG Grant DMS-1065178 and a TUBITAK grant BIDEB 2232. The second author was supported by a National Science Foundation Graduate Research Fellowship under Grant No. DGE-1110007.

\section{Star surgery} \label{s:starsurgerydefinitions}
{

\subsection{Description}
Rational blow-downs of plumbings of spheres were shown to be symplectic operations by Symington \cite{Sym1,Sym2} proving that both the plumbing of spheres and the rational homology ball, support symplectic structures with convex boundary inducing the same contact structures. One may ask more generally, what can replace a neighborhood of spheres in this symplectic cut-and-paste manner. This question is reduced to understanding symplectic fillings of certain contact structures by the following result of Gay and Stipsicz.

\begin{theorem}[Theorem 1.2 of \cite{GS2}]
If $C=C_1\cup \cdots \cup C_n\subset (X,\omega)$ is a collection of symplectic surfaces in a symplectic 4-manifold $(X,\omega)$ intersecting each other $\omega$-orthogonally according to the negative definite plumbing graph $\Gamma$ and $\nu C\subset X$ is an open set containing $C$, then $C$ admits an $\omega$-convex neighborhood $U_C\subset \nu C\subset (X,\omega)$.
\end{theorem}

Note that the $\omega$-orthogonal condition can be achieved by any configuration of spheres which intersect positively and transversely by an isotopy through symplectic spheres.

 The contact structures induced on the boundaries can be understood through an open book decomposition by results of Gay and Mark (under the additional assumption that the plumbing graph contains no bad vertices i.e. $w_j+e_j<0$ where $w_j$ is the weight of a vertex $v_j$ and $e_j$ is the number of edges emanating from $v_j$). Let $C$ be the union of symplectic surfaces intersecting $\omega$-orthogonally according to such a graph. Form a surface $\Sigma$ from the plumbing graph as follows. Start with the surfaces corresponding to each vertex $v_j$  and connect sum on $|w_j+e_j|$ disks. Then connect sum the resulting surfaces according to the edges of the graph. Take one simple closed curve around each connect sum neck, and denote these curves by $c_1,\dots, c_k$. 
 
 \begin{theorem}[Theorem 1.1 of \cite{GM}] \label{thm:GM}
 Any neighborhood of $C$ contains a neighborhood of $C$ with strongly convex boundary that admits a symplectic Lefschetz fibration having regular fibers $\Sigma$ and exactly one singular fiber. The vanishing cycles are $c_1,\dots, c_k$ and $C$ is the union of the closed components of the singular fiber. The induced contact structure on the boundary is supported by the induced open book $(\Sigma,\tau)$, where $\tau$ is a composition of positive Dehn twists around the curves $c_1,\dots, c_k$.
 \end{theorem}
 
Note that since the curves $c_1,\dots, c_k$ are disjoint from each other, the order of the Dehn twists does not matter in defining $\tau$.

\begin{remark}
It was shown by Park and Stipsicz \cite{PS} that this contact structure is in fact the canonical contact structure on the boundary Seifert fibered space (given by the complex tangencies on the link of the corresponding normal surface singularity). Their result holds more generally for the boundary contact structure of any convex negative-definite plumbing of surfaces.
\end{remark}

In the case that the symplectic spheres intersect according to a star-shaped graph, additional tools are available to search for alternate convex fillings of the same convex boundary. Classifications of such fillings were studied in \cite{Sta}. While most plumbings do not share the same convex boundary with a rational homology ball, many share convex boundary with a symplectic filling of significantly smaller Euler characteristic. We will call the operation of cutting out the neighborhood of spheres which intersect according to a star-shaped graph, and replacing it with an alternative convex symplectic filling of strictly smaller Euler characteristic \emph{star surgery}.

It is not clear that any contactomorphism of the boundary extends over the alternate convex filling, so the star surgery will depend on an identification of the convex contact boundary of the filling with the concave contact boundary of the complement of the star-shaped plumbing. In our cases, this identification will be made using the open book decomposition defined by Theorem \ref{thm:GM} and an equivalent open book decomposition on the boundary of the alternate convex filling.

In each case, the neighborhood of spheres will be replaced by an alternate symplectic filling supported by a Lefschetz fibration. The fibers of this Lefschetz fibration will agree with the fibers of the Lefschetz fibration constructed by Gay and Mark on the plumbing neighborhood of spheres. However the vanishing cycles will differ. We will show that the induced contact structures on the boundary agree, by showing that the open book decompositions have equal monodromy. In order to do this, we will require knowledge of relations in the mapping class group of planar surfaces.

\subsection{Conventions on mapping class elements and handle diagrams for Lefschetz fibrations}
{
A Lefschetz fibration naturally induces an open book decomposition on the boundary where the fibers of the open book are the same as the fibers of the Lefschetz fibration, and the monodromy is given by a product of positive (right-handed) Dehn twists about the vanishing cycles. Since mapping class groups of surfaces are non-abelian, the order of the vanishing cycles generally matters. Conventions in the literature vary, but we will use a fixed set of conventions throughout this paper which are consistent with each other, that we describe here.

Suppose $c_1,\cdots, c_n$ are simple closed curves on the fiber. Denote by $D_{c_i}$ a positive Dehn twist around $c_i$. The product $D_{c_1}D_{c_2}\cdots D_{c_n}$ means first Dehn twist along $c_1$, then $c_2$, and so on until finally along $c_n$, meaning we are using group notation as opposed to functional notation. When the fiber is a disk with holes, we can place the holes along a circle concentric with the bounday of the disk. Labelling the holes $\{1,\cdots, m\}$ counterclockwise, we use the notation $D_{i_1,\cdots, i_k}$ for $i_1,\cdots, i_k\in \{1,\cdots , m\}$ to indicate a positive Dehn twist about a curve which convexly contains the holes $i_1,\cdots, i_k$.

Any factorization of the monodromy of an open book decomposition into a product of positive Dehn twists corresponds to a Lefschetz fibration. When the fibers are disks with holes, we have the natural handlebody decomposition for this Lefschetz fibration where the holes are represented by dotted circles forming a trivial braid corresponding to 1-handles and the vanishing cycles correspond to 2-handles. We view the holed-disk fibers as orthogonal to the dotted circles, oriented so that the outward normal points \emph{downward} (i.e. turn the holed-disk upside-down). Then the monodromy factorization $D_{c_1}\cdots D_{c_n}$ corresponds to the Lefschetz fibration where the vanishing cycles appear as curves, each lying in a disk transverse to the trivial braid of dotted circles such that $c_1$ is at the top of the diagram and $c_n$ at the bottom. 

To draw the handlebody, we will isotope the holes on the disk so that they all lie on the bottom half of the disk along a circle concentric to the boundary. Then using the upside-down disk convention, the holes, ordered counterclockwise on the downward pointing disk, correspond to dotted trivial braid components labeled \emph{left to right}. A curve which convexly encloses holes $i_1,\cdots, i_k$ will appear in this diagram as a circle, half of which passes in front of all of the dotted circles, and the other half passes behind the dotted circles corresponding to $i_1,\cdots, i_k$, but in front of all the other circles. An example, using the top to bottom convention where the outward normal to the disk points downward, is in figure \ref{fig:LFconv}.

\begin{figure}
	\centering
	\includegraphics[scale=.7]{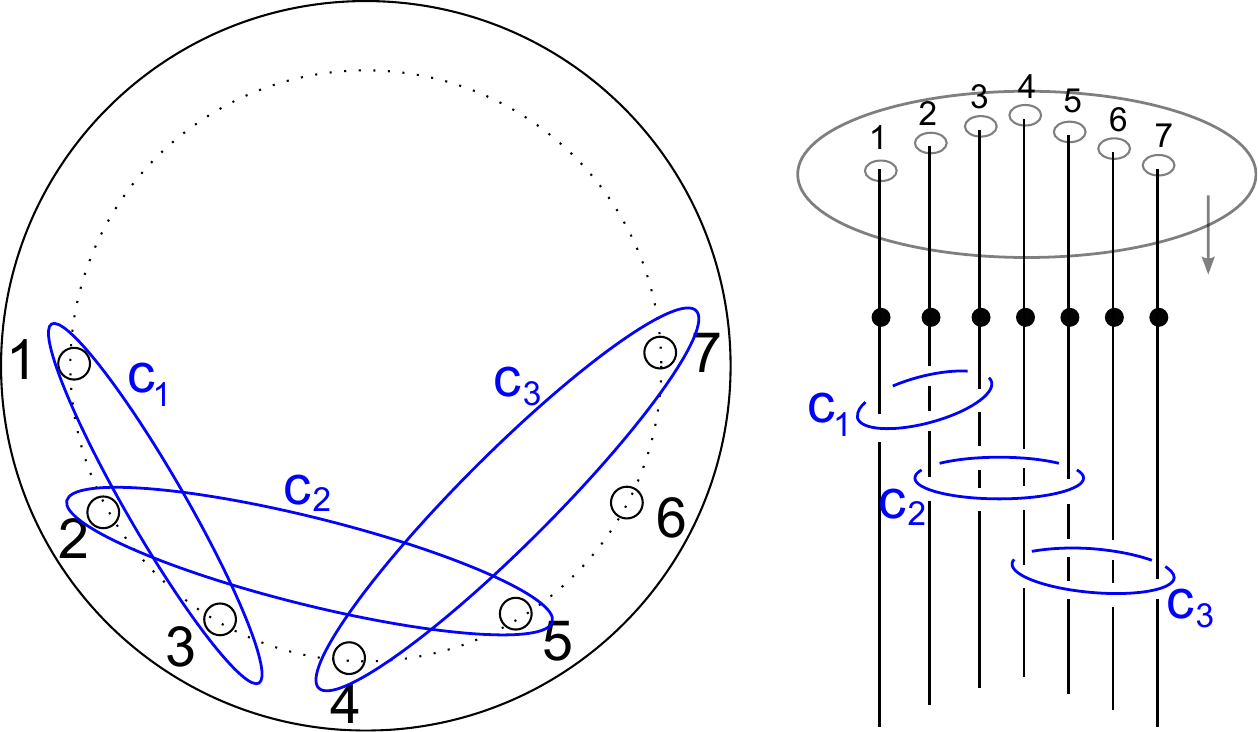}
	\caption{The Lefschetz fibration corresponding to the monodromy factorization $D_{1,3}D_{2,5}D_{4,7}=D_{c_1}D_{c_2}D_{c_3}$. The tops and bottoms of the dotted segments are identified to form a closed trivial braid.}
	\label{fig:LFconv}
\end{figure}

The mapping class group on a disk with holes is generated by Dehn twists. Dehn twists about disjoint curves commute. If we place the holes on a circle concentric with the boundary, we can order them counter-clockwise. Suppose $A$, $B$, and $C$ are collections of holes such that all the holes of $A$ precede all the holes of $B$ which preced all the holes of $C$ going around the circle counterclockwise. Then the \emph{lantern relation} states
\begin{equation}\label{e:lantern}
D_{A\cup B\cup C}D_AD_BD_C=D_{A\cup B}D_{A\cup C}D_{B\cup C}.
\end{equation}
The Dehn twists on the right-hand side can be cyclically permuted.

By combining a sequence of lantern relations one obtains the \emph{daisy relation} which is given as follows. Suppose $B_0,B_1,\cdots , B_m$  are disjoint subsets of the $k$ holes on the disk labelled counter-clockwise ($m\geq 2$). Then 
\begin{equation}\label{e:daisy}
D_{B_0\cup B_1\cup \cdots \cup B_m} D_{B_0}^{m-1}D_{B_1}\cdots D_{B_m} = D_{B_0\cup B_1}D_{B_0\cup B_2}\cdots D_{B_0\cup B_m}D_{B_1\cup \cdots \cup B_m}.
\end{equation}
This daisy relation was shown to correspond to Fintushel and Stern's rational blow-down operations in \cite{EMV}.

We will use one more combination of lantern moves, corresponding to one of Park's generalized rational blow-downs which starts with a linear plumbing with weights $(-2,-5,-3)$. It was first worked out in \cite{EMV} that the relation is given as follows. Consider a disk whose holes are grouped into sets $A,B,C,D,E$ labelled counter-clockwise.
\begin{equation}\label{e:park}
D_{A\cup B\cup C \cup D \cup E}D_{A\cup B}D_{A}D_BD_C^2D_DD_E=D_{A\cup C}D_{B\cup C}D_{A\cup B\cup D}D_{A\cup B\cup E}D_{C\cup D\cup E}.
\end{equation}
This equality can be shown by performing one daisy relation \eqref{e:daisy} where $B_0=A\cup B$ introducing some negative Dehn twists, followed by a lantern relation \eqref{e:lantern}.
}

\subsection{The family of star surgeries $(\St_i, \T_i)$}
 A particularly nice family of star-shaped surgeries is given by symplectically replacing a neighborhood of a configurations of spheres, $\mathcal{S}_i$ by its smallest filling $\T_i$ (a specific filling of minimal Euler characteristic). The configurations $\mathcal{S}_i$ are made up of symplectic spheres which intersect according to star-shaped graphs with $i+2$ arms. Each arm contains $i-1$ spheres of square $-2$, and the central vertex is a sphere of square $-i-3$ (Figure \ref{fig:PlumbingSn}). Note that $\mathcal{S}_1$ is just a $-4$ sphere, and the replacement $\mathcal{T}_1$ is the rational blow-down of this $-4$ sphere. However for $i>1$, the graphs are star-shaped but not linear, and the replacement fillings $\mathcal{T}_i$ are not obtained by a rational blow-down of a subgraph of the spheres shown in the original configuration. Handlebody diagrams for the fillings $\mathcal{T}_i$ for $i=1,2,3$ are shown in Figure \ref{fig:ST123}. In general, a handlebody diagram for $\mathcal{T}_i$ (see Figure \ref{fig:SC1}) has $i+2$ 1-handles represented by dotted circles, and $\frac{(i+1)(i+2)}{2}$ 2-handles, one passing through each distinct pair of 1-handles. The corresponding monodromy factorization is
 $$(D_{1,2}D_{1,3}D_{1,4}\cdots D_{1,i+2})(D_{2,3}D_{2,4}\cdots D_{2,i+2})\cdots (D_{i,i+1}D_{i,i+2})(D_{i+1,i+2})$$

\begin{figure}
\centering
\subfloat[Plumbing graph for $\mathcal{S}_i$]
{\includegraphics[scale=.3]{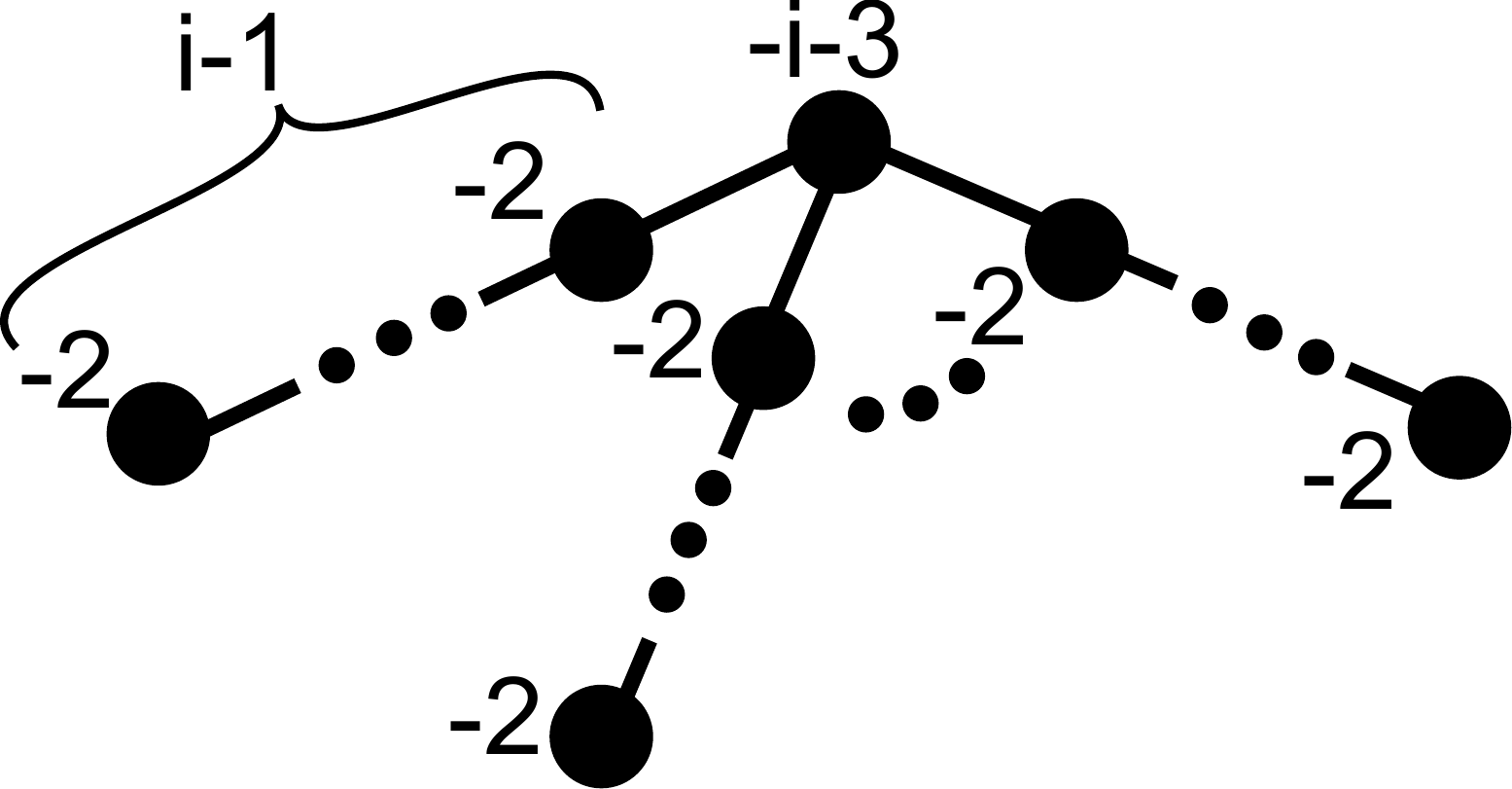}
\label{fig:PlumbingSn}
}
\subfloat[Fibers and vanishing cycles for $\mathcal{S}_i$]
{\includegraphics[scale=.3]{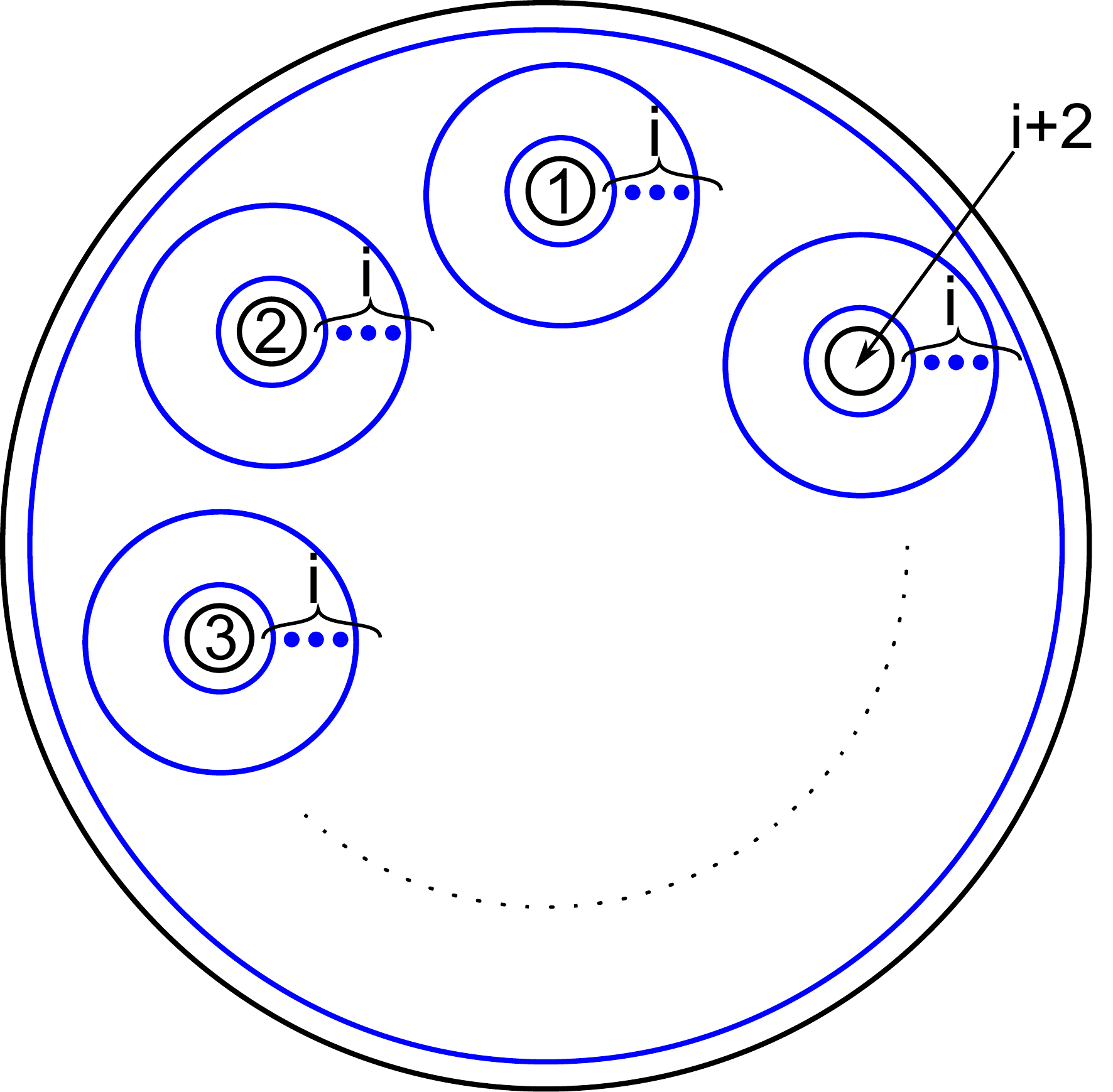}
\label{fig:SnLF}
}
\caption{$\St_i$}
\label{fig:Sn}
\end{figure}

\begin{figure}
\centering
\subfloat[$\mathcal{S}_1$ and $\mathcal{T}_1$]
{\includegraphics[scale=1]{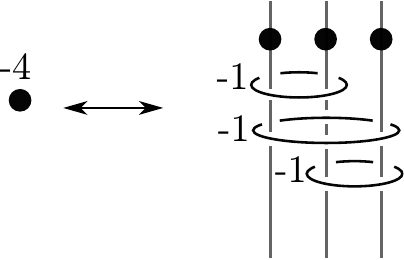}
}\hspace{.5cm}
\subfloat[$\mathcal{S}_2$ and $\mathcal{T}_2$]
{\includegraphics[scale=1]{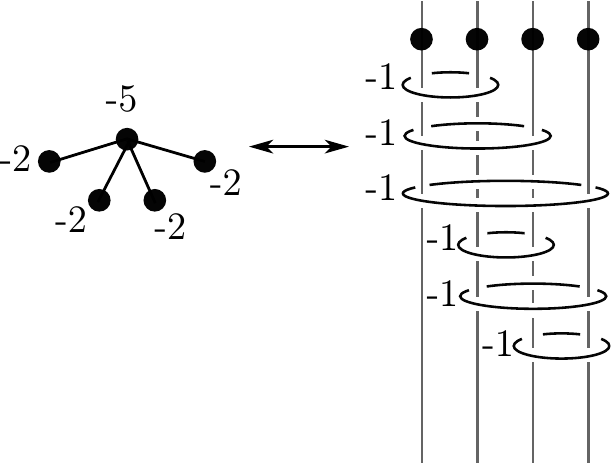}
}\\
\subfloat[$\mathcal{S}_3$ and $\mathcal{T}_3$]
{\includegraphics[scale=1]{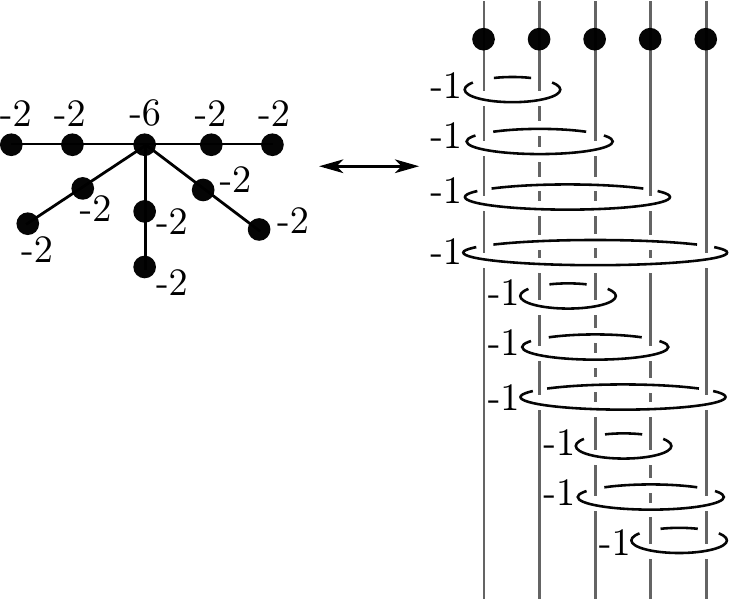}}
\caption{The first three configurations of spheres and their replacement symplectic fillings in the family $(\mathcal{S}_i,\mathcal{T}_i)$.}
\label{fig:ST123}
\end{figure}

We can replace $\St_i$ by $\T_i$ symplectically due to the following proposition.

\begin{proposition} \label{p:starsymplectic}
The contact structure induced on the convex boundary of the plumbing of spheres $\St_i$ is the same as the contact structure induced on the convex boundary of $\T_i$ (whose symplectic structure is determined by the supporting Lefschetz fibration described above).
\end{proposition}

\begin{proof}
For the plumbing of spheres $\St_i$, Gay and Mark's construction implies that the fibers of the supporting Lefschetz fibration are $i+2$ holed disks, and the vanishing cycles consist of a single curve parallel to the outer boundary component, and $i$ boundary parallel curves around each of the holes (see Figure \ref{fig:SnLF}). 

In order to show that the contact structures induced on the boundaries of $\St_i$ and $\T_i$ agree, we will show that the open book decompositions induced on the boundary of the corresponding Lefschetz fibrations are the same. Since the pages are the same it suffices to show the monodromies are equal, which amounts to the following relation.

\begin{equation}\label{e:genlantern}
D_{1,2,\cdots , i+1,i+2}D_1^{i}\cdots D_{i+2}^{i}=(D_{1,2}D_{1,3}D_{1,4}\cdots D_{1,i+2})(D_{2,3}D_{2,4}\cdots D_{2,i+2})\cdots (D_{i,i+1}D_{i,i+2})(D_{i+1,i+2})
\end{equation}
This relation is sometimes referred to in the literature as the \emph{generalized lantern relation}.

To see these are equal in the mapping class group, proceed by induction on $i$. When $i=1$ this is the standard lantern relation. By a relabelled version of the $i-1$ case, we can use the inductive hypothesis to say that the right hand side is equal to
$$(D_{1,2}D_{1,3}D_{1,4}\cdots D_{1,i+2})D_{2,3,\cdots, i+1,i+2}D_2^{i-1}D_3^{i-1}\cdots D_{i+1}^{i-1}D_{i+2}^{i-1}$$
Applying a daisy relation to this then gives the left hand side. 
\end{proof}

\begin{remark}
For the negative definite star plumbings we consider, the induced contact structure is supported by a planar open book. By work of Wendl, \cite{W}, any other convex filling is supported by a planar Lefschetz fibration inducing the same open book decomposition on the boundary. Therefore the two Lefschetz fibrations correspond to positive factorizations of the same monodromy. Equivalent elements in a planar mapping class group are always related by some sequence of lantern relations and commutation, but this sequence may pass through factorizations involving negative Dehn twists. If one can obtain one positive factorization from another through a sequence of relations so that at each stage we remain in a positive factorization then the overall symplectic operation is broken down into a sequence of other symplectic operations.

For example, this proof of the generalized lantern relation shows that it can be obtained by performing a sequence of daisy and lantern relations, so that after each relation, we still have a positive factorization. Endo, Mark, and Van Horn-Morris \cite{EMV} showed that daisy relations correspond to Fintushel-Stern rational blow-downs. Therefore these particular star surgeries are equivalent to sequences of rational blow-downs. However, it is not easy to see the existence of all the configurations which are rationally blown-down at each stage, so in applications it would be difficult to find all of these rational blow-downs to perform. Instead, we can just perform the sequence all at once with a single star surgery.

A question one can ask is whether all such star surgery operations arise as sequences of rational blow-downs. Surprisingly, this was shown in the linear case in \cite{BO}, but it was suspected that star surgery was more general. After the appearence of the first draft of this paper, the second author proved that a certain example of a star surgery cannot be realized as any sequence of symplectic rational blow-downs \cite[Theorem 5.2]{Sta2}.
\end{remark}


\subsection{The star surgery $(\mathcal{Q},\mathcal{R})$}
The following star surgery is realted to the $\St_2$, $\T_2$ star surgery but improves it in the sense that it reduces the Euler characteristic by a larger amount. Let $\mathcal{Q}$ denote the configuration of spheres indicated on the left hand side of Figure~\ref{fig:plumbalt}. Let $\xi_{can}$ be the canonical contact structure $\partial \mathcal{Q}$. 

\begin{figure}[h]
\centering
\includegraphics[scale=1]{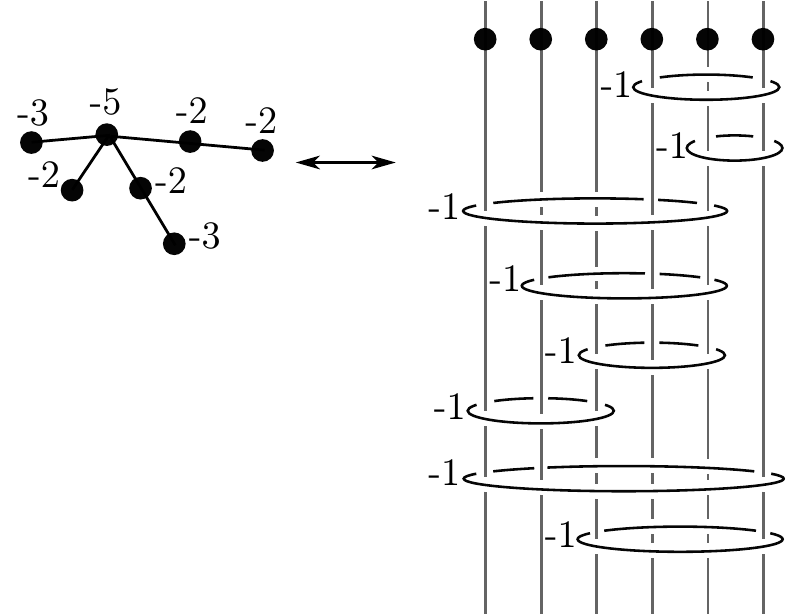}
\caption{A star-shaped plumbing graph for $\mathcal{Q}$, and an alternate symplectic filling, $\mathcal{R}$.}
\label{fig:plumbalt}
\end{figure}

\begin{proposition} \label{pro:exchangeST}
There exists a symplectic manifold $\mathcal{R}$ of Euler characteristic $3$ with convex boundary, such that the induced contact manifold on the boundary is contactomorphic to $(\partial \mathcal{Q},\xi_{can})$.
\end{proposition}

\begin{proof}
Let $\mathcal{R}$ be the 4-manifold given by the handlebody diagram on the right of figure \ref{fig:plumbalt}. This particular diagram makes apparent a Lefschetz fibration structure on $\mathcal{R}$. The fibers of the Lefschetz fibration are six holed disks, and the base is a disk. The fibers near the boundary are disks perpendicular to the dotted circles which give the holes. The vanishing cycles are given by the $-1$ framed 2-handles. We can verify that the induced open book decomposition on the boundary agrees with the one that is induced on the boundary of the symplectic plumbing given by \cite{GM}.

We choose a standard $6$-holed disk, such that the holes are centered at the vertices of a regular hexagon on the disk. Label the holes with numbers $1,\cdots, 6$ going around counter-clockwise. The construction of Gay and Mark \cite{GM} indicates that the open book induced on the boundary of the plumbing $\mathcal{Q}$ has pages which are 6-holed disks, with monodromy given by positive Dehn twists about disjoint curves enclosing holes as follows:
$$D_{123456}D_{12}D_1D_2D_3^2D_{45}^2D_4D_5D_6^3$$
The monodromy induced by the Lefschetz fibration on $\mathcal{R}$ (reading the vanishing cycles from top to bottom) is:
$$D_{46}D_{56}D_{145}D_{245}D_{345}D_{123}D_{126}D_{36}$$
Commuting when needed and then performing a Park  relation \eqref{e:park} on the plumbing monodromy where $A=\{4\}$, $B=\{5\}$, $C=\{6\}$, $D=\{1,2\}$ and $E=\{3\}$, we get an intermediate factorization:
$$D_{46}D_{56}D_{1245}D_{345}D_{1236}D_{45}D_6D_1D_2D_3$$
Note that this corresponds to a symplectic filling obtained from the original plumbing by rationally blowing down the configuration that comes from $u_{4,1}$, $u_0$ and the symplectic resolution of the union of $u_{3,1}$ with $u_{3,2}$. Continuing, by commuting terms and performing a lantern relation \eqref{e:lantern} where $A=\{45\}$, $B=\{1\}$, and $C=\{2\}$ we obtain the factorization
$$D_{46}D_{56}D_{145}D_{245}D_{12}D_{345}D_{1236}D_6D_3$$
Note this corresponds to rationally blowing down a $-4$ sphere which was not visible until after the first rational blowdown. We commute terms and perform one more lantern relation \eqref{e:lantern} (corresponding to blowing down another $-4$ sphere) where $A=\{12\}$, $B=\{3\}$, and $C=\{6\}$ to obtain the factorization corresponding to the Lefschetz fibration on $T$.
\begin{equation}D_{46}D_{56}D_{145}D_{245}D_{345}D_{123}D_{126}D_{36} \end{equation}
\end{proof}

\begin{remark} \label{rem:Qbd}
The proof makes it clear that $\mathcal{R}$ is obtained from $\mathcal{Q}$ through a sequence of rational blowdowns: one Park rational blowdown of a $(-2,-5,-3)$ configuration, followed by two consecutive rational blowdowns of $-4$ spheres.
\end{remark}

\subsection{The star surgery $(\mathcal{U},\mathcal{V})$}

Let $\mathcal{U}$ denote the plumbing according to the graph in figure \ref{fig:plumbalt2}. Let $\xi$ denote the (canonical) contact structure induced on the convex boundary of the plumbing.

\begin{figure}
\centering
\includegraphics[scale=1]{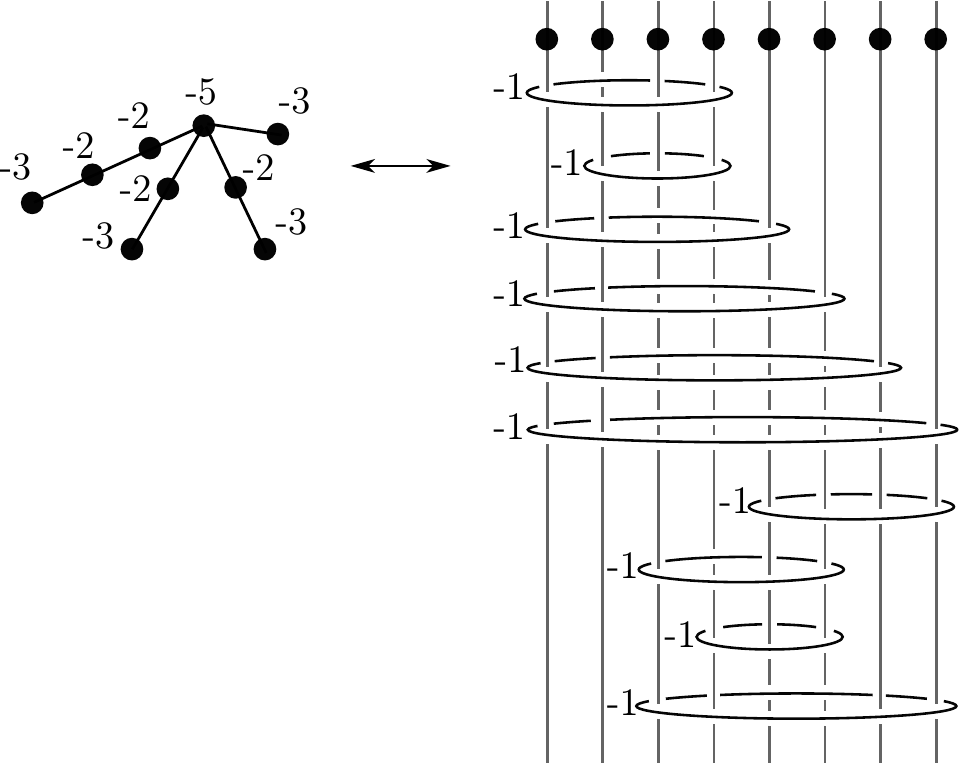}
\caption{A star-shaped plumbing graph for $\mathcal{U}$, and an alternate symplectic filling, $\mathcal{V}$.}
\label{fig:plumbalt2}
\end{figure}

\begin{proposition} \label{pro:exchangeUV}
There exists a symplectic manifold $\mathcal{V}$ of Euler characteristic $3$ with convex boundary, such that the induced contact manifold on the boundary is contactomorphic to $(\partial \mathcal{U},\xi_{can})$.
\end{proposition}

\begin{proof}
Gay and Mark's construction gives an open book on the boundary of the Lefschetz fibration for the plumbing $\mathcal{U}$ whose pages are $8$-holed disks and whose monodromy is given as follows:
$$D_{12345678}D_{12}^3D_1D_2D_{34}^2D_3D_4D_{56}^2D_5D_6D_{78}D_7D_8$$

Using the monodromy equivalence corresponding to the Park $(-2,-5,-3)$ rational blowdown, where $A=\{1\}$, $B=\{2\}$, $C=\{3,4\}$, $D=\{5,6\}$, $E=\{7,8\}$, we get the following monodromy.
$$D_{134}D_{234}D_{1256}D_{1278}D_{345678}D_{12}^2D_3D_4D_{56}D_5D_6D_7D_8$$
After commuting Dehn twists about disjoint curves, we can perform two lantern relations. One where $A=\{1,2\}$, $B=\{5\}$, and $C=\{6\}$ and the other where $A=\{1,2\}$, $B=\{7\}$, and $C=\{8\}$, which results in the following factorization.
$$D_{134}D_{234}D_{125}D_{126}D_{56}D_{127}D_{128}D_{78}D_{345678}D_3D_4D_{56}$$
After commuting $D_{56}$ and $D_{78}$ towards the end, we can use a daisy relation with $B_0=\{5,6\}$, $B_1=\{7,8\}$, $B_2=\{3\}$, and $B_3=\{4\}$. The resulting monodromy corresponds to that of the Lefschetz fibration for $\mathcal{V}$.
\begin{equation}D_{134}D_{234}D_{125}D_{126}D_{127}D_{128}D_{5678}D_{356}D_{456}D_{3478}\end{equation}
\end{proof}

\begin{remark}
Note that this proof shows that this operation is also obtained as a sequence of rational blowdowns.
\end{remark}

\subsection{The star surgeries $(\mathcal{K},\mathcal{L})$, $(\mathcal{M},\mathcal{N})$ and $(\mathcal{O},\mathcal{P})$}

In \cite{Sta2}, it was shown that a configuration of symplectic spheres intersecting according to a graph as in figure \ref{fig:notqbd} can be replaced by a symplectic filling of Euler characteristic two, whose Lefschetz fibration handlebody is shown in figure \ref{fig:notqbd}.

\begin{figure}
	\centering
	\includegraphics[scale=1]{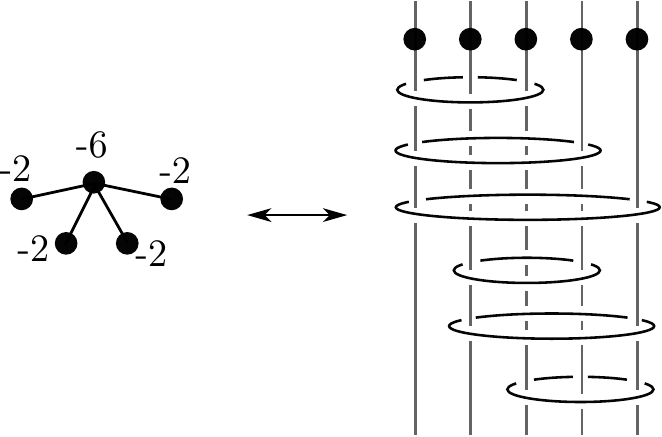}
	\caption{The $\mathcal{K}$ plumbing and alternate filling $\mathcal{L}$, providing a star surgery operation which is inequivalent to a sequence of rational blow-downs \cite{Sta2}.}
	\label{fig:notqbd}
\end{figure}

The corresponding monodromy substitution for this $\mathcal{K},\mathcal{L}$ star surgery is
\begin{equation}\label{eqn:notqbdsub}
D_1^2D_2^2D_3D_4^2D_5^2D_{12345}=D_{123}D_{14}D_{15}D_{24}D_{25}D_{345}
\end{equation}
The equivalence of these elements was shown directly in \cite{Sta2} and it was also shown that this substitution is not equivalent to any sequence of symplectic rational blow-downs/ups. From this star surgery, we can generate two other useful star surgeries, such that the alternate fillings also have Euler characteristic two. The plumbing and filling diagrams are shown in figures \ref{fig:notqbdspr2} and \ref{fig:notqbdspr4}.

\begin{figure}
\centering
	\includegraphics[scale=1]{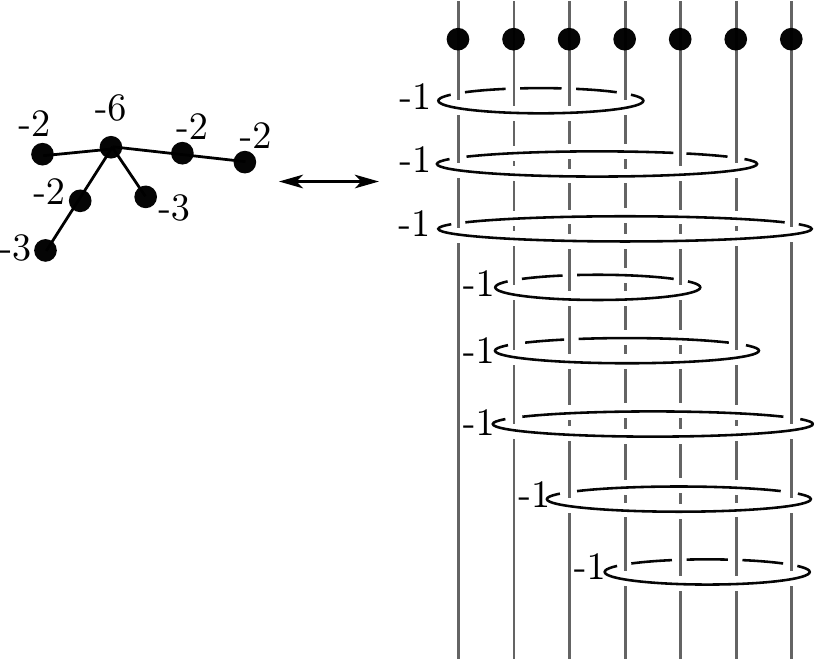}
	\caption{The $\mathcal{M}$, $\mathcal{N}$ star surgery. The dotted circles are labeled left to right as $1, 2_a, 2_b, 3, 4_a, 4_b, 5$}
	\label{fig:notqbdspr2}
\end{figure}

\begin{figure}
\centering
	\includegraphics[scale=1]{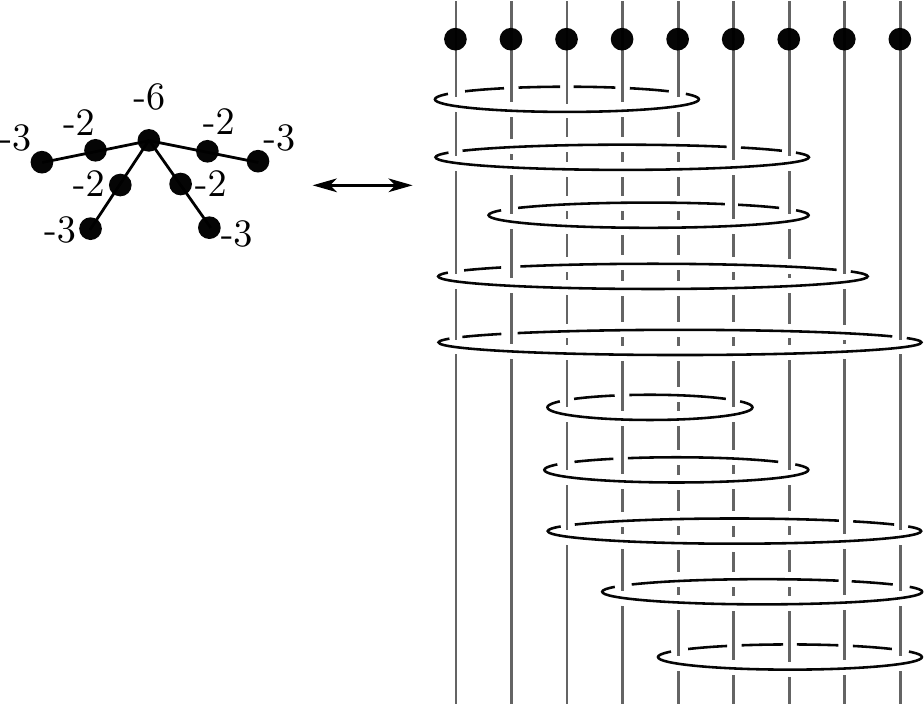}
	\caption{The $\mathcal{O}$, $\mathcal{P}$ star surgery. The dotted circles are labeled left to right as $1_a, 1_b, 2_a, 2_b, 3, 4_a, 4_b, 5_a, 5_b$.}
	\label{fig:notqbdspr4}
\end{figure}

These monodromy substitutions are obtained from the $\mathcal{K}$, $\mathcal{L}$ substitution (equation \ref{eqn:notqbdsub}) together with an additional lantern relation applied according to Lemma 2.1 from \cite{EMV}, where some of the original holes split into two holes (e.g. hole $2$ splits into holes $2_a$ and $2_b$). When applying the $\mathcal{K}$, $\mathcal{L}$ monodromy substitution, treat the split holes as a single joined hole, but when applying the lantern relations treat them as separate holes. Performing two splittings applying Lemma 2.1 of \cite{EMV} each time, we get a monodromy substitution corresponding to the $\mathcal{M},\mathcal{N}$ star surgery of figure \ref{fig:notqbdspr2}.
$$D_1^2D_{2_a2_b}^2D_{2_a}D_{2_b}D_3D_{4_a4_b}D_{4_a}D_{4_b}D_5^3D_{12_a2_b34_a4_b5}=D_{12_a2_b3}D_{14_a4_b}D_{15}D_{2_a2_b4_a}D_{2_a2_b4_b}D_{2_a5}D_{2_b5}D_{34_a4_b5}$$
Here, the holes are labeled $1,2_a,2_b,3,4_a,4_b,5$ counter-clockwise around the disk. Performing two more splits, we get the monodromy substitution corresponding to the $\mathcal{O},\mathcal{P}$ star surgery, as shown in figure \ref{fig:notqbdspr4} on the disk with holes labeled $1_a,1_b,2_a,2_b,3,4_a,4_b,5_a,5_b$ counter-clockwise.
\begin{align*}
&D_{1_a1_b}^2D_{1_a}D_{1_b}D_{2_a2_b}^2D_{2_a}D_{2_b}D_3D_{4_a4_b}^2D_{4_a}D_{4_b}D_{5_a5_b}^2D_{5_a}D_{5_b}D_{1_a1_b2_a2_b34_a4_b5_a5_b}\\
&=D_{1_a1_b2_a2_b3}D_{1_a4_a4_b}D_{1_b4_a4_b}D_{1_a1_b5_a}D_{1_a1_b5_b}D_{2_a2_b4_a}D_{2_a2_b4_b}D_{2_a5_a5_b}D_{2_b5_a5_b}D_{34_a4_b5_a5_b}
\end{align*}
Note that the proof that these factorizations are equivalent involves applying the substitution from equation \ref{eqn:notqbdsub} and then lantern relations, but the first step will introduce negative Dehn twists into the factorization, so there is not a simple way to understand these operations a sequence of known operations (though it would be more difficult to prove that they are not equivalent to sequences of known operations--see \cite{Sta, Sta2} for an idea of how this might be proven).

\section{Algebraic topology of the star surgery fillings}\label{s:fillingproperties}

In this section we will compute the fundamental group of the fillings $\T_i,\mathcal{R},\mathcal{V},\mathcal{L}$, and various other algebraic topology invariants which will be needed to understand the homeomorphism type, Kodaira dimension, and Seiberg-Witten invariants of the manifolds constructed by star surgeries using these fillings. The computations are reasonably straightforward given the handlebody descriptions of these manifolds and the Lefschetz fibration structure. The most thorough computations will be given for $\T_2$, as this will be our model example used to show how star surgery can be applied to create exotic manifolds whose Kodaira dimension and Seiberg-Witten invariants can be fully computed.

\subsection{Properties of the Fillings $\mathcal{T}_i$} \label{s:fil}
We will compute basic algebraic topological properties of $\T_i$, and specifically $\T_2$ since we will use the $\St_2$, $\T_2$ star surgery to construct an exotic copy of $\CP2\# 8 \barCP2$.

\begin{proposition}\label{p:fund}
The fillings $\T_i$ satisfy $\pi_1(\T_i)=\mathbb{Z}/2\mathbb{Z}$. The generator can be represented by the meridian of any of the -2 framed surgery curves on the ends of the arms in $\partial \St_i$.
\end{proposition}
\begin{proof}
A handlebody diagram for $\T_i$ is given as in Figure \ref{fig:SC1} by $(i+2)$ 1-handles corresponding to generators $\{y_1,\cdots, y_{i+2}\}$ of $\pi_1(\T_i)$, and $\frac{(i+1)(i+2)}{2}$ 2-handles corresponding to the relations $\{y_jy_k=1\}_{j\neq k\in \{1,\cdots , i+2\}}$. We can easily compute the fundamental group:
$$\pi_1(\T_i)= \langle y_1,\cdots , y_{i+2}: y_j=y_k^{-1}, j\neq k\rangle = \langle y_1: y_1^2=1\rangle.$$

Now we will track a curve representing the generator $y_1$ of $\pi_1(\T_i)$ through a sequence of equivalent surgery diagrams, to show that it restricts to a nontrivial element in $\partial \T_i=\partial S_i$ represented by the meridian of a $-2$-framed surgery curve on the end of the first arm in the standard handlebody diagram for $\St_i$. Note that $\pi_1$ is equivalently generated by any of the $y_j$ with the relation $y_j^2=1$, and that $y_j$ restricts in the same way to the meridian of the last curve in the $j^{th}$ arm.

Start with the handlebody diagram in Figure \ref{fig:SC1}. The generator of the fundamental group, $y_1$, is represented by the dashed red curve. By rotating the plane of projection about a vertical axis, observe this handlebody diagram is isotopic to that given by Figure \ref{fig:SC2} (the direction into the page in figure \ref{fig:SC1} corresponds to the right side of the page in figure \ref{fig:SC2}). Now, exchange the dotted circles for $0$-framed circles and treat the handlebody diagram as a surgery diagram for its boundary 3-manifold. After blowing-down all the $-1$ framed 2-handles, we obtain a surgery diagram as in Figure \ref{fig:SC3} containing $i+2$ unknotted $(i+1)$-framed circles, twisted together with a full positive twist. Blowing up negatively once at a common intersection point of their Seifert surfaces, and then $i$ times along each individual curve,the curves become untwisted and 0-framed again, so we switch the $0$-framings to dotted circles as in Figure \ref{fig:SC4}.  Perform handleslides by first sliding the dashed red curve over the top $-1$ framed 2-handle on that dotted circle, then sliding that 2-handle over the one below it, and so on until there is a chain of $-2$ curves linked to a single $-1$ curve. Do this for each arm (without the red reference curve on the other arms), and then finally slide the $-1$ framed handle that links all of the dotted circles over each of the remaining $-1$ framed 2-handles. Then we see that after cancelling 1,2-handle pairs, the dashed red curve appears as the meridian of the last $-2$ framed curve as in Figure \ref{fig:SC5}.

\begin{figure}
\centering
\captionsetup[subfigure]{width=2.5cm}
\subfloat[$\T_i$]{\includegraphics[scale=.8]{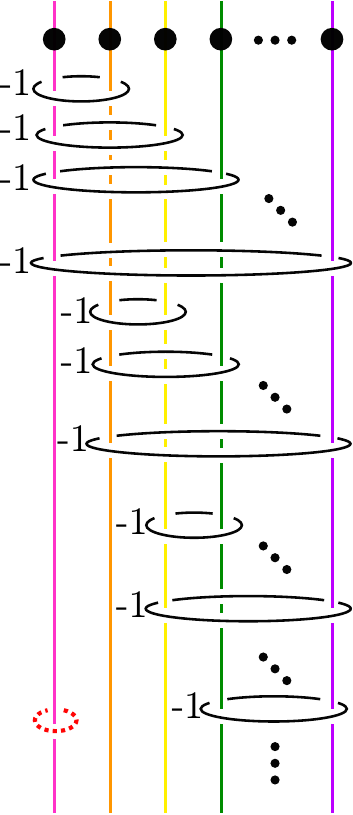} \label{fig:SC1}} \hspace{1cm}
\subfloat[Isotopic diagram in a different projection]{\includegraphics[scale=.5]{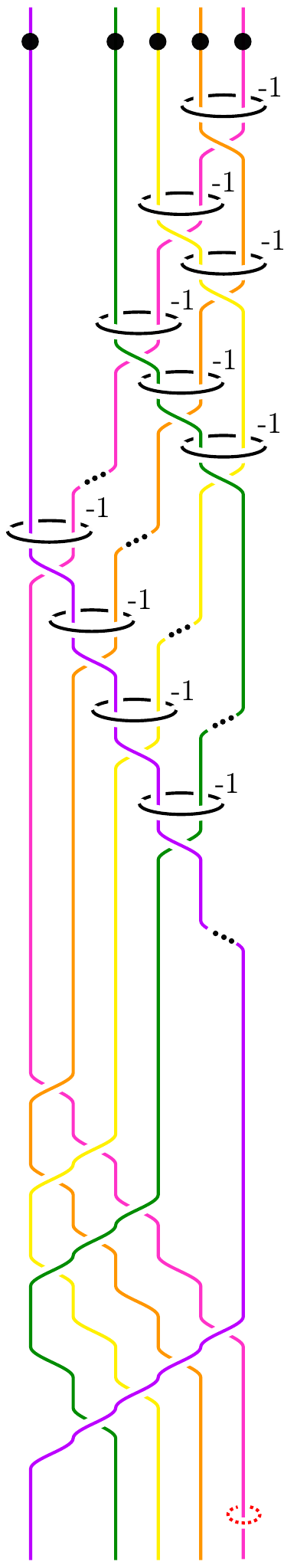}\label{fig:SC2}}\hspace{1cm}
\subfloat[Surgery diagram for the common boundary]{\includegraphics[scale=.8]{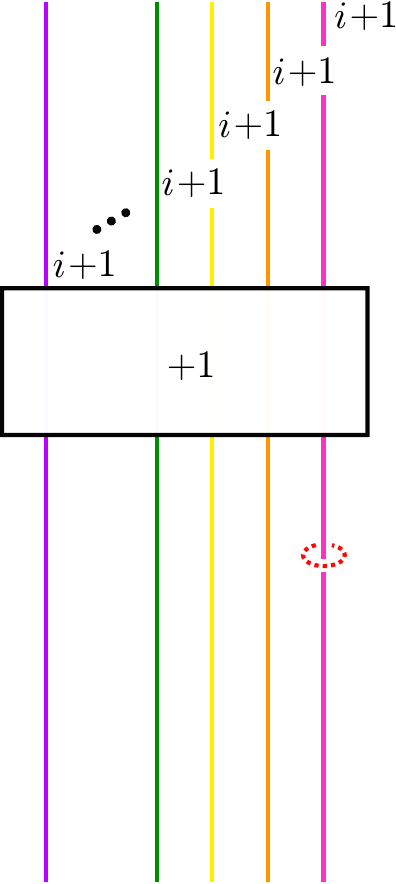}\label{fig:SC3}}\hspace{1cm}
\subfloat[$\St_i$ Lefschetz fibration]{\includegraphics[scale=.8]{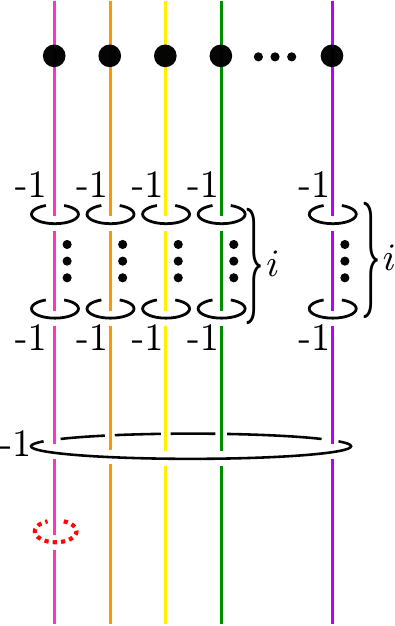}\label{fig:SC4}}\hspace{1cm}
\subfloat[$\St_i$ equivalent diagram showing the plumbing]{\includegraphics[scale=.6]{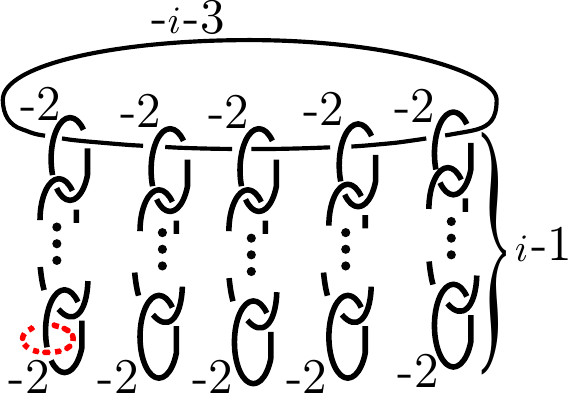}\label{fig:SC5}}
\caption{Relating diagrams for $\partial \T_i$ and $\partial \St_i$}
\label{fig:SC}
\end{figure}
\end{proof}

\begin{proposition}\label{p:hom}
The second homology of the fillings is given by $H_2(\T_2;\mathbb{Z})=\mathbb{Z}\oplus\mathbb{Z}$. The generators can be represented by tori. Moreover $\chi(\T_2)=3$ and $\sigma(\T_2)=-2$.
\end{proposition}
\begin{proof}
Using the handlebody diagram as in Figure \ref{fig:SC1} but ignoring the reference curve, we can compute the CW chain complex. Let $y_j$ denote the $j^{th}$ 1-handle and let $x_{jk}$ denote the 2-handle whose attaching circle passes through the $j^{th}$ and $k^{th}$ 1-handles. Then the relevant chain groups are
$C_2(\T_i)=\langle x_{jk}: j\neq k\in \{1,\cdots, i+2\}\rangle$, $C_1(\T_i)=\langle y_j: j\in \{1,\cdots, i+2\}\rangle$. The boundary map is determined by $\partial x_{jk}=y_j+y_k$. 

In particular, the Euler characteristic is
$$\chi(\T_i)=\frac{(i+1)(i+2)}{2}-(i+2)+1$$

When $i=2$, $\chi(\T_2)=6-4+1=3$ and the 2-cycles are generated freely by $x_{12}+x_{34}-x_{13}-x_{24}$ and $x_{14}+x_{23}-x_{12}-x_{34}$. Note that each $x_{jk}$ has square $-1$ so the intersection form with respect to the above basis is given by the matrix
$$\left[ \begin{array}{cc} -4 & 2\\2 & -4\end{array}\right]\sim_{\mathbb{R}} \left[ \begin{array}{cc} -4 & 0\\0 & -3\end{array}\right].$$
Therefore $\T_2$ is negative definite.

We can see that the homology class $x_{12}+x_{34}-x_{13}-x_{24}$ can be represented by tori by examining the handlebody diagram. Take the cores of the 2-handles $x_{12}$ and $x_{34}$, and the cores with opposite orientation for $x_{13}$ and $x_{14}$. We connect these up to a closed torus by adding in the twice punctured disks whose outer boundary coincides with the attaching circle for $x_{ij}$ which does not intersect any of the dotted circles, and tubing together the holes with a tube encircling the dotted circle so the orientations match up as in Figure \ref{fig:HomologyTori}. One can see this surface is indeed a torus directly or check by calculating its Euler characteristic is $4\chi(D^2)+4\chi(\text{2 holed }D^2)+4\chi(\text{tube})=4-4+0=0$. A similar surface represents the other generator $x_{14}+x_{23}-x_{12}-x_{34}$.

\begin{figure}
\includegraphics[scale=.5]{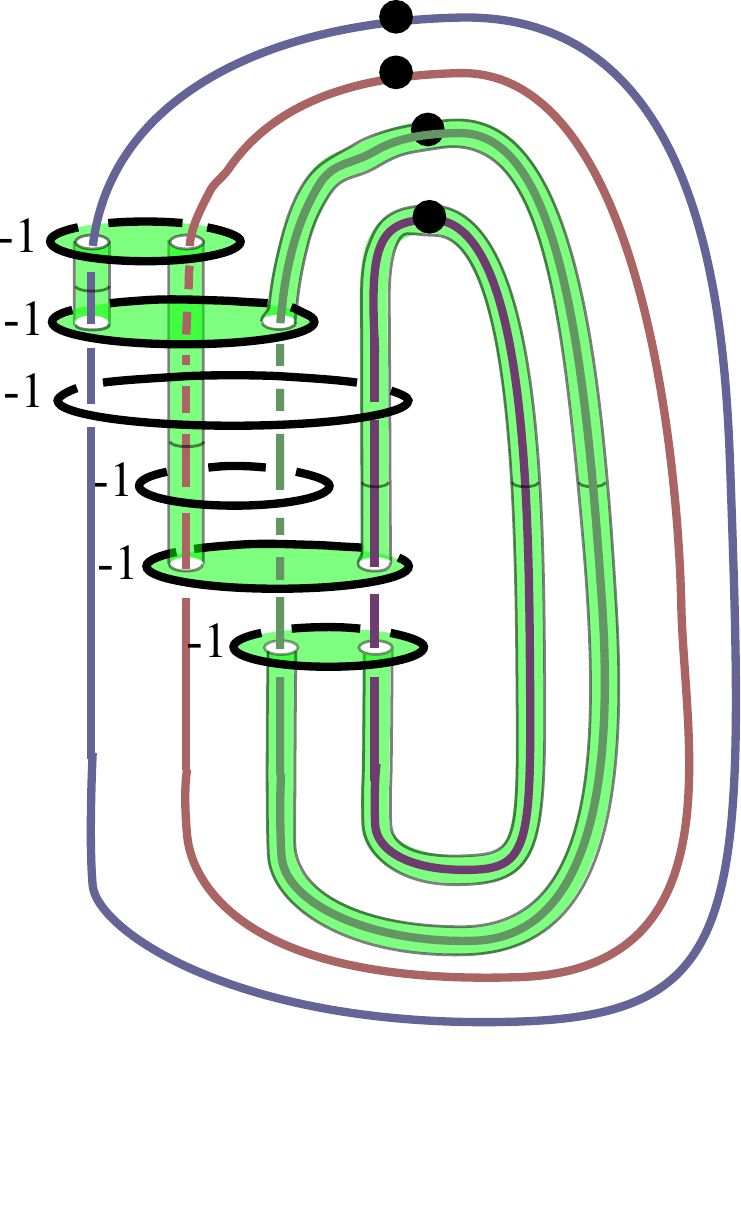}
\caption{A torus representing a generator of the homology of $\T_2$.}
\label{fig:HomologyTori}
\end{figure}
\end{proof}

\begin{proposition}\label{p:can}
The canonical class of the filling vanishes: $K|_{\T_2}=0$.
\end{proposition}

\begin{proof}
We will use the obstruction theoretic interpretation of the first Chern class. See [\cite{Gom}, Proposition 2.3] and  [\cite{EO}, Section 3.1] for  related discussions. Figure \ref{fig:SC1} suggests that each $\T_i$ admits a positive allowable Lefschetz fibration whose fibers are disks with $i+2$ punctures, which are indicated by the circles with dots, and $\frac{(i+2)(i+1)}{2}$ vanishing cycles which are indicated by $-1$ framed $2$-handles. This Lefschetz fibration defines an almost complex structure which is compatible with the symplectic structure of $\T_i$. We will show that $c_1(T\T_i,J)=0$ for $i=2$.

Let $x_{jk}$ be a $2$-cell which is a generator of the chain complex $C_2(\T_i)$ as indicated in the proof of Proposition \ref{p:hom}. The attaching curve $\tilde{x}_{jk}$ of each $2$-cell $x_{jk}$ can be put on a  fiber $F_{jk}$ of the Lefschetz fibration.  Drawing the regular fibers on planes induces  trivializations of their tangent bundles which in turn induce a trivialization of $(T\T_i,J)$ over 1-skeleton. Now $c_1(T\T_i,J)(x_{jk})$ is the obstruction to extending this trivialization over the $2$-cell $x_{jk}$.  This obstruction is precisely the winding number of $\tilde{x}_{jk}$ measured with respect to the trivialization of $F_{jk}$. Each $\tilde{x}_{jk}$ is an embedded planar curve, so its winding number is one. For $i=2$ the proof of Proposition \ref{p:hom} tells us that the generators of $H_2(\T_2,\mathbb{Z})$ are  $x_{12}+x_{34}-x_{13}-x_{24}$ and  $x_{14}+x_{23}-x_{12}-x_{34}$. Hence  $c_1(T\T_i,J)$ evaluates as zero on both of these generators.
\end{proof}

We will need to understand $H^2(\T_2)$ and $H^2(\partial \T_2)$, as well as the restriction map between them. By Poincare duality, $H^2(\partial \T_2)\cong H_1(\partial \T_2)$. We can compute the first homology of the boundary of the filling explicitly from the surgery diagram obtained from the handlebody diagram of $\T_2$ by switching the dotted circles to $0$-framed circles as in Figure \ref{fig:H1boundary}.

\begin{figure}
\includegraphics[scale=.5]{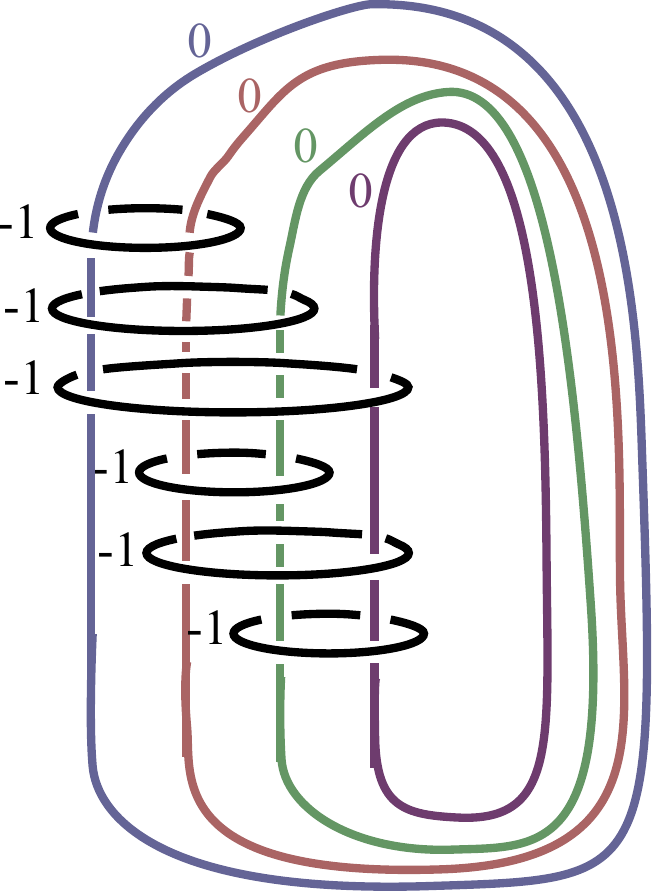}
\caption{A surgery diagram for $\partial \T_2$.}
\label{fig:H1boundary}
\end{figure}

\begin{proposition}\label{p:H1boundary}
Let $\lambda_1,\lambda_2,\lambda_3,\lambda_4$ denote the meridians of the four $0$-framed curves (from left to right) in the surgery diagram in Figure \ref{fig:H1boundary}. Then 
$$H_1(\partial \T_2;\Z) = \langle \lambda_1, \lambda_2-\lambda_1,\lambda_3-\lambda_1| 2(\lambda_2-\lambda_1)=2(\lambda_3-\lambda_1)=12\lambda_1=0\rangle\cong \Z/12\oplus \Z/2\oplus \Z/2$$
\end{proposition}

\begin{proof}
Let $\mu_{ij}$ denote the meridian of the $-1$-framed curve which links the ith and jth $0$-framed curves ($i\neq j\in \{1,2,3,4\}$).

A presentation for $H_1(\partial \T_2)$ has generators $(\lambda_1,\lambda_2,\lambda_3,\lambda_4, \mu_{12},\mu_{13}, \mu_{14}, \mu_{23}, \mu_{24}, \mu_{34})$ with relations given by the linking matrix:
\begin{eqnarray*}
-\mu_{ij}+\lambda_i+\lambda_j&=&0\\
\mu_{12}+\mu_{13}+\mu_{14}&=&0\\\
\mu_{12}+\mu_{23}+\mu_{24}&=&0\\
\mu_{13}+\mu_{23}+\mu_{34}&=&0\\
\mu_{14}+\mu_{24}+\mu_{34}&=&0\\
\end{eqnarray*}


Eliminate the $\mu_{ij}$ using the first relations and then solve for $\lambda_4$. The resulting presentation for $H_1(\partial \T_2;\Z)$ is
$$\langle \lambda_1,\lambda_2,\lambda_3| 2(\lambda_2-\lambda_1)=2(\lambda_3-\lambda_1)=-8\lambda_1-2\lambda_2-2\lambda_3=0\rangle.$$
Equivalently
$$\langle \lambda_1,\lambda_2-\lambda_1,\lambda_3-\lambda_1| 2(\lambda_3-\lambda_1)=2(\lambda_2-\lambda_1)=12\lambda_1=0\rangle$$
\end{proof}

\begin{proposition} \label{p:cohomT2}
$$H^2(\T_2)\cong \Z\oplus \Z \oplus \Z/2$$
\end{proposition}

\begin{proof}
The $d$-cells of $\T^2$ are generators for the $d^{th}$ CW homology chain complex. Their duals freely generate $C^d(\T_2)$. Let $y_1,y_2,y_3,y_4$ denote the 1-handles in the diagram for $\T_2$ (from left to right), and let $x_{ij}$ denote the 2-handle whose attaching circle passes over $y_i$ and $y_j$. Let $\phi^i$ denote the dual of $y_i$ and $\psi^{ij}$ denote the dual of $x_{ij}$ so $C^1(\T_2)=\langle \phi^1,\phi^2,\phi^3,\phi^4\rangle$ and $C^2(\T_2)=\langle \psi^{12},\psi^{13},\psi^{14},\psi^{23},\psi^{24},\psi^{34}\rangle$ and $C^k(\T^2)=0$ for $k>2$. Then
$$(\delta \phi^i)(x_{jk})=\phi^i(\partial x_{jk})=\phi^i(y_j+y_k)=\delta_j^i+\delta_k^i = \left(\sum_{\ell \neq i}\psi^{i\ell}\right)(x_{jk}).$$
Therefore $\delta \phi^i = \sum_{\ell\neq i}\psi^{i\ell}$. (Note we identify $\psi^{ij}=\psi^{ji}$.)

We conclude
$$H^2(\T_2) = \langle \psi^{12},\psi^{13},\psi^{14},\psi^{23},\psi^{24},\psi^{34}|\psi^{12}+\psi^{13}+\psi^{14}=\psi^{12}+\psi^{23}+\psi^{24}= \psi^{13}+\psi^{23}+\psi^{34}=\psi^{14}+\psi^{24}+\psi^{34}=0\rangle$$

By eliminating variables we obtain
$$H^2(\T_2) = \langle \psi^{12},\psi^{13},\psi^{12}+\psi^{13}+\psi^{23}| 2(\psi^{12}+\psi^{13}+\psi^{23})=0\rangle\cong \Z\oplus \Z\oplus \Z/2$$
\end{proof}

\begin{proposition}  \label{p:restriction}
The image of the restriction map $i: H^2(\T_2)\to H^2(\partial T_2)$ has index 2 (therefore order 24).
\end{proposition}
\begin{proof}
The restriction map $i: H^2(\T_2)\to H^2(\partial \T_2)$ composed with Poincare duality yields a map $\rho=PD\circ i: H^2(\T_2)\to H_1(\partial \T_2).$ On generators $\psi^{ij}$ we have $\rho(\psi^{ij})=\mu_{ij}$ where $\mu_{ij}$ is the meridian of the surgery curve corresponding to the attaching circle for $x_{ij}$. In $H_1(\partial \T_2)$ we had the relation $\mu_{ij}=\lambda_i+\lambda_j$ so 
$$\rho(\psi^{12}) = (\lambda_2-\lambda_1)+2\lambda_1 \qquad \rho(\psi^{13})=(\lambda_3-\lambda_1)+2\lambda_1$$
$$\rho(\psi^{12}+\psi^{13}+\psi^{23})=2(\lambda_2-\lambda_1)+2(\lambda_3-\lambda_1)+6\lambda_1 =6\lambda_1\in H_1(\partial \T_2)$$
Therefore the image of $\rho$ (which equals the image of $i$) is generated by $(\lambda_2-\lambda_1,\lambda_3-\lambda_1,2\lambda_1)$ which has index 2 in $H_1(\partial \T_2)$.
\end{proof}

\subsection{Properties of $\mathcal{R}$}

\begin{lemma} The filling $\mathcal{R}$ is simply connected. \end{lemma}
\begin{proof}
Using the handlebody decomposition, a presentation for $\pi_1(\mathcal{R})$ is
$$\langle y_1,\cdots, y_6 | y_3y_6=y_4y_6=y_5y_6=y_1y_4y_5=y_2y_4y_5=y_3y_4y_5=y_1y_2y_3=y_1y_2y_6=1\rangle$$
We can eliminate $y_1,y_2,y_3,y_4,y_5$ with the first five relations resulting in the following presentation.
$$\langle y_6 | y_6^{-3}=y_6^3=y_6^5=1\rangle$$
Since $3$ and $5$ are relatively prime, this is a presentation for the trivial group.
\end{proof}

\begin{lemma} \label{l:homT}
The intersection form on $H_2(\mathcal{R};\Z)$ is negative definite. In fact, there exists homology classes $\alpha$ and $\beta$ which freely  generate $H_2(\mathcal{R};\Z)\cong \Z^2$, such that the intersection form with respect to the basis $\langle \alpha, \beta\rangle $ for $H_2(\mathcal{R};\Z)$ is given by 
$$\left[\begin{array}{cc} -10&-23\\ -23&-79\end{array} \right].$$
\end{lemma}

\begin{proof}
Labeling the 2-handles from top to bottom as $x_1,\cdots, x_8$, the cycles generating homology from the CW complex given by the handlebody decomposition are $\alpha=x_1+x_2-x_3-x_4+x_5+x_6-2x_8$ and  $\beta=x_1+x_2-3x_3-3x_4+5x_5+3x_7-5x_8$. Since $x_i^2=-1$ and the $x_i$ are orthogonal with respect to the intersection pairing, the intersection form can be computed directly.
\end{proof}

\subsection{Properties of $\mathcal{V}$}

\begin{lemma} 
The filling $\mathcal{V}$ is simply connected. 
\end{lemma}

\begin{proof}
Using the handlebody decomposition, a presentation for $\pi_1(\mathcal{V})$ is
\begin{align*}\langle y_1,\cdots, y_8|&y_1y_3y_4=y_2y_3y_4=y_1y_2y_5=y_1y_2y_6=y_1y_2y_7\\&=y_1y_2y_8=y_5y_6y_7y_8=y_3y_5y_6=y_4y_5y_6=y_3y_4y_7y_8=1\rangle\end{align*}
Using the first six relations to eliminate $y_1,y_2,y_5,y_6,y_7,y_8$ this simplifies to
$$\langle y_3,y_4|(y_3y_4)^8=y_3(y_3y_4)^4=y_4(y_3y_4)^4=(y_3y_4)^5=1\rangle$$
Since $5$ and $8$ are relatively prime we get that $y_3y_4=1$, and thus $y_3=y_4=1$ so the group is trivial.
\end{proof}

\begin{lemma} The intersection form on $H_2(\mathcal{V};\Z)$ is negative definite. In fact $H_2(\mathcal{V};\Z)\cong \Z^2$ and the intersection form with respect to a basis is $$\left[\begin{array}{cc} -30&5\\ 5&-49\end{array} \right].$$\end{lemma}

\begin{proof}
Labelling the 2-handles from top to bottom as $x_1,\cdots, x_{10}$, two cycles generating homology from the CW complex given by the handlebody decomposition are $-2x_1-2x_2+x_3+x_4-3x_7+x_8+x_9+3x_{10}$ and  $-3x_3-3x_4+3x_5+3x_6-x_7+2x_8+2x_9-2x_{10}$. Since $x_i^2=-1$ and the $x_i$ are orthogonal with respect to the intersection pairing, the intersection form can be computed directly.
\end{proof}

\subsection{Properties of $\mathcal{L}$}

We compute the fundamental group and second homology of the filling $\mathcal{L}$ of figure \ref{fig:notqbd}.

\begin{proposition}
$\pi_1(\mathcal{L})\cong \Z/4$, and the generator restricts to the boundary Seifert fibered space as a meridian of any of the $-2$ surgery curves in the plumbing diagram.
\end{proposition}

\begin{proof}
Using the handlebody decomposition of figure \ref{fig:notqbd}, we obtain a presentation for $\pi_1(\mathcal{L})$ generated by the five 1-handles with relations given by the 2-handles as follows

$$\langle y_1,y_2,y_3,y_4,y_5 | y_1y_2y_3=y_1y_4=y_1y_5=y_2y_4=y_2y_5=y_3y_4y_5=1\rangle$$
which simplifies to
$$\langle y_1,y_2,y_3 | y_1y_2y_3=y_2y_1^{-1}=y_3y_2^{-2}=1\rangle$$
which again simplifies to
$$\langle y_1 | y_1^4=1\rangle.$$
Note that the relations set $y_1,y_2,y_4^{-1}$ and $y_5^{-1}$ all equal. Each of these curves can be isotoped into the boundary of $\mathcal{L}$. By performing blow-downs, handle-slides, and handle cancellations, in a similar manner to proposition \ref{p:fund}, we see that these curves in $\partial \mathcal{L}$ are meridians of the $-2$ surgery curves in the diagram for $\partial \mathcal{K}$.
\end{proof}

\begin{proposition}
$H_2(\mathcal{L})\cong \Z$. It is generated by an element $x_{14}-x_{15}-x_{24}+x_{25}$ of square $-4$, represented by a torus obtained by gluing tubes to the cores of the specified 2-handles $x_{ij}$ which link the $i^{th}$ and $j^{th}$ dotted circles in figure \ref{fig:notqbd}.
\end{proposition}

The proof is a direct computation from the handlebody decomposition of figure \ref{fig:notqbd}.

\begin{proposition}
Using the symplectic structure induced by the Lefschetz fibration on $\mathcal{L}$, $c_1(\mathcal{L})=0$, therefore the canonical class is trivial.
\end{proposition}

\begin{proof}
This follows from the winding number interpretation of $c_1(\mathcal{L})$ for Lefschetz fibrations, and the fact that the generator of $H_2(\mathcal{L})$ passes over two Lefschetz 2-handles with $+1$ multiplicity and two 2-handles with $-1$ multiplicity.
\end{proof}

\section{Elliptic Fibrations}\label{s:ellipticfibrations}

In order to find symplectic embeddings of the plumbings into well-understood symplectic 4-manifolds, we will use many different elliptic fibrations exhibiting various types of singular fibers. These fibrations were classified by Persson in \cite{Per}, providing a full list of possible configurations of singular fibers. However, in order to keep track of homology classes of symplectic spheres, we need to explicitly construct these elliptic fibrations by blowing-up a special Lefschetz pencil in $\CP2$.

\begin{lemma}\label{l:fibration}
There is an elliptic fibration on $E(1)=\CP2 \# 9 \barCP2$ with one $I_3$ fiber, one $I^*_0$ fiber, three fishtail fibers, and a section. Labelling the components of the $I_0^*$ fiber as $S_1,\cdots, S_5$ where $S_5$ intersects the $E_9$ section, and $S_4$ intersects $S_1,S_2,S_3$ and $S_5$, and labelling the components of the $I_3$ fiber as $V_1,V_2,V_3$ where $V_1$ intersects the $E_9$ section, the homology classes are as follows.
$$\begin{array}{rlrl}
\left[S_1\right]&=h-e_1-e_2-e_3 & [V_1]&=h-e_1-e_8-e_9\\
\left[S_2\right]&=h-e_1-e_4-e_5 & [V_2]&=h-e_2-e_4-e_6\\
\left[S_1\right]&=h-e_1-e_6-e_7 & [V_3]&=h-e_3-e_5-e_7\\
\left[S_4\right]&=e_1-e_8&&\\
\left[S_5\right]&=e_8-e_9&&\\
\end{array}$$
\end{lemma}

\begin{proof}
Let $C_1$ be a reducible degree three algebraic curve in $\CP2$ defined by a homogeneous polynomial $p_1$ and made up of three complex projective lines, $L_1$, $L_2$, and $L_3$, which share a single common intersection point $p\in \CP2$. Let $C_2$ be a reducible degree three algebraic curve defined by a homogeneous polynomial $p_2$ and made up of three lines, $L_4$, $L_5$ and $L_6$, such that $L_4$ passes through $p$, and $L_5$ and $L_6$ each intersect all other $L_i$ generically in double points (see Figure \ref{fig:LP1}). Note that the homology class of $L_i$ for $i=1,\cdots, 6$ is the generator of $H_2(\CP2;\Z)$, $h$. Define a Lefschetz pencil on $\CP2$ by setting $C_{[t_1:t_2]}=\{t_1p_1+t_2p_2=0\}$. Note the base locus is the set of points where $C_1$ intersects $C_2$.

Blow-up at $p$, and let the exceptional sphere represent the homology class $e_1$. Then the proper transforms of $L_1,L_2,L_3$ and $L_4$ represent $h-e_1$ in homology, and $L_5$ and $L_6$ are unchanged (see Figure \ref{fig:LP2}). Therefore the proper transform $\widetilde{C_1}$ represents $3h-3e_1$ in homology and the proper transform $\widetilde{C_2}$ represents $3h-e_1$. We redefine $C_1$ to be the curve given by $\widetilde{C_1}$ together with the exceptional class $E_1$ with multiplicity two, and let $C_2=\widetilde{C_2}$ (see Figure \ref{fig:LP3}). Then the curves $C_1$ and $C_2$ represent the same class in homology so they define a new Lefschetz pencil as before.

Now blow-up at the six intersection points of $L_1,L_2,L_3$ with $L_5,L_6$, so that the exceptional classes are labeled $e_2,\cdots, e_7$ (see Figure \ref{fig:LP4}). Redefine $C_1$ and $C_2$ as the proper transforms, and note that these curves both represent the homology class $3h-e_1-e_2-\cdots - e_7$ and thus define a Lefschetz pencil.

There is still a non-empty base locus since the exceptional sphere $E_1$ is now part of $C_1$ which intersects the proper transform of $L_4$. We blow-up at this point to obtain a new exceptional sphere $E_8$ (see Figure \ref{fig:LP5}). The proper transform of $E_1$ represents $e_1-e_8$, so the proper transform of $C_1$ represents $3h-e_1-e_2-\cdots -e_7-2e_8$ whereas the proper transform of $C_2$ represents $3h-e_1-e_2-\cdots -e_8$. Redefining $C_1$ as its proper transform together with the exceptional sphere $E_8$ with multiplicity one, and $C_2$ as its proper transform, the two curves again define a Lefschetz pencil, but still intersect where $E_8$ meets $L_4$ (see Figure \ref{fig:LP6}). We blow-up one more time at this point, and the resulting proper transforms of $C_1$ and $C_2$ are homologous and do not intersect (Figure \ref{fig:LP7}).

The resulting elliptic fibration has singular fibers $C_1$ and $C_2$ which are of type $I_0^*$ and $I_3$ respectively, as well as other singular fibers, which we can perturb to generic fishtail fibers. An Euler characteristic computation implies that there are three fishtail fibers. A schematic for the singular fibers in this elliptic fibration together with the section $E_9$ is given by Figure \ref{fig:LP8}.

\begin{figure}
\centering
\subfloat[Starting configuration]{\includegraphics[scale=.5]{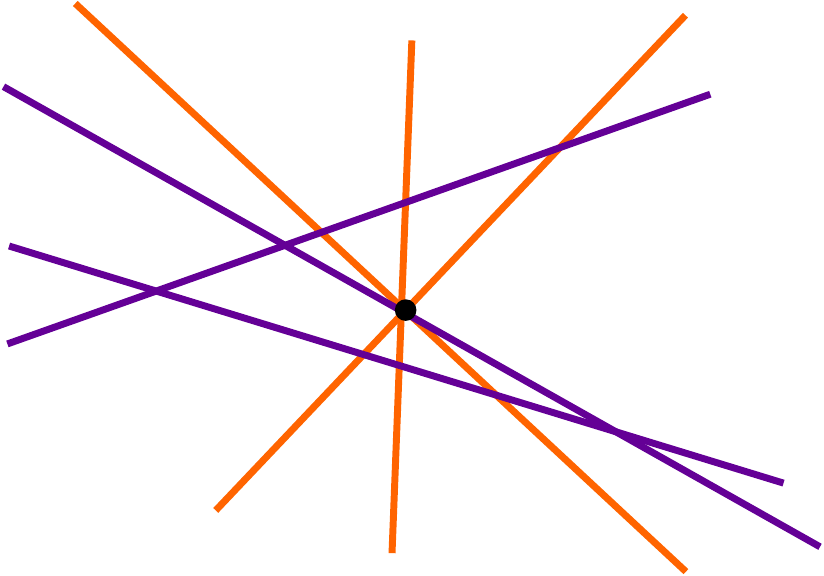} \label{fig:LP1}}
\subfloat[Blow-up $e_1$]{\includegraphics[scale=.5]{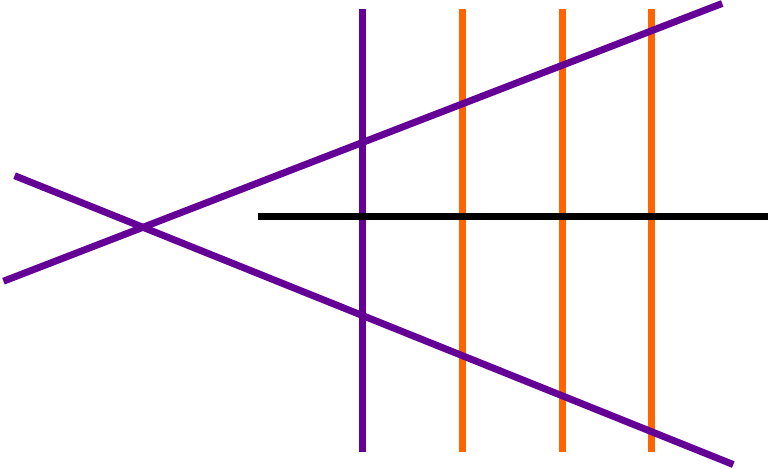} \label{fig:LP2}}
\subfloat[Reset $C_1$, $C_2$]{\includegraphics[scale=.5]{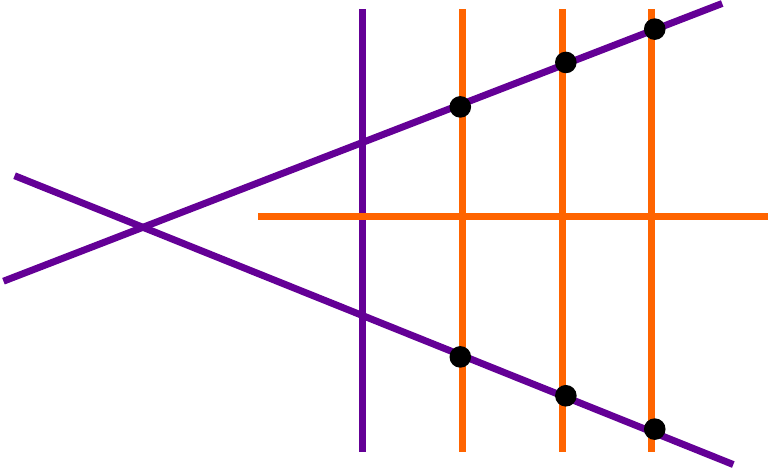} \label{fig:LP3}}
\subfloat[Blow-up $e_2,\cdots , e_7$]{\includegraphics[scale=.5]{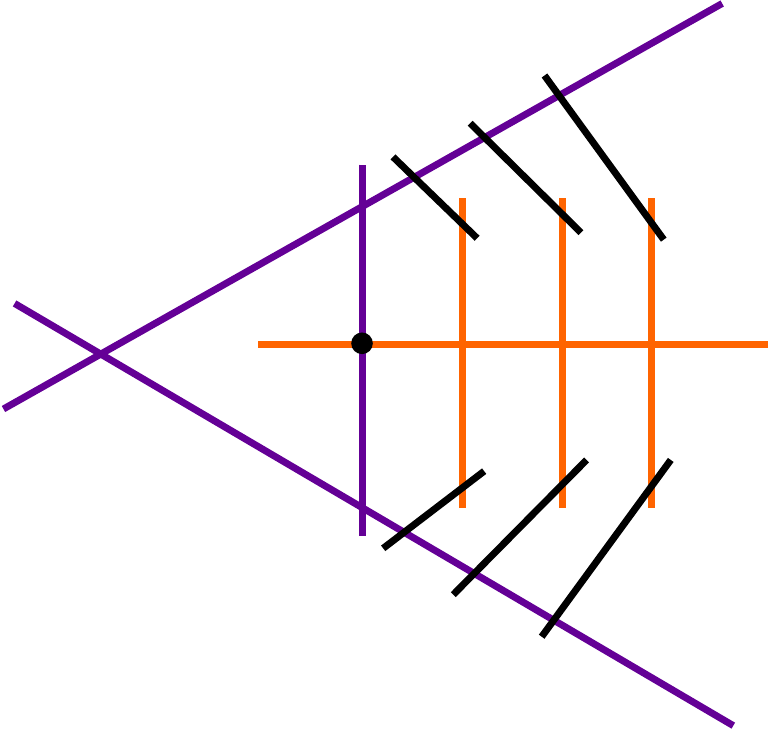} \label{fig:LP4}}\\
\subfloat[Blow-up $e_8$]{\includegraphics[scale=.5]{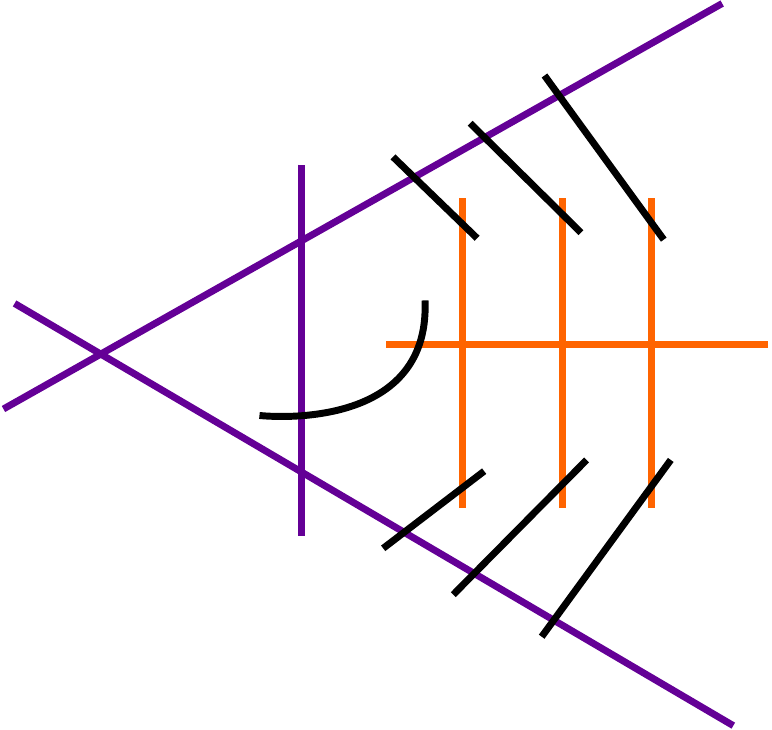} \label{fig:LP5}}
\subfloat[Reset $C_1$, $C_2$]{\includegraphics[scale=.5]{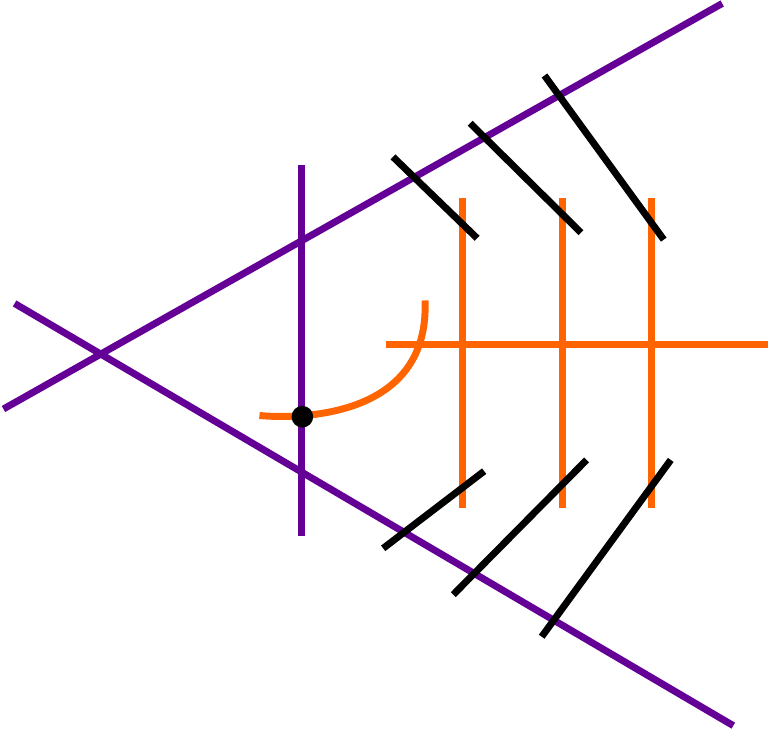} \label{fig:LP6}}
\subfloat[Blow-up $e_9$]{\includegraphics[scale=.5]{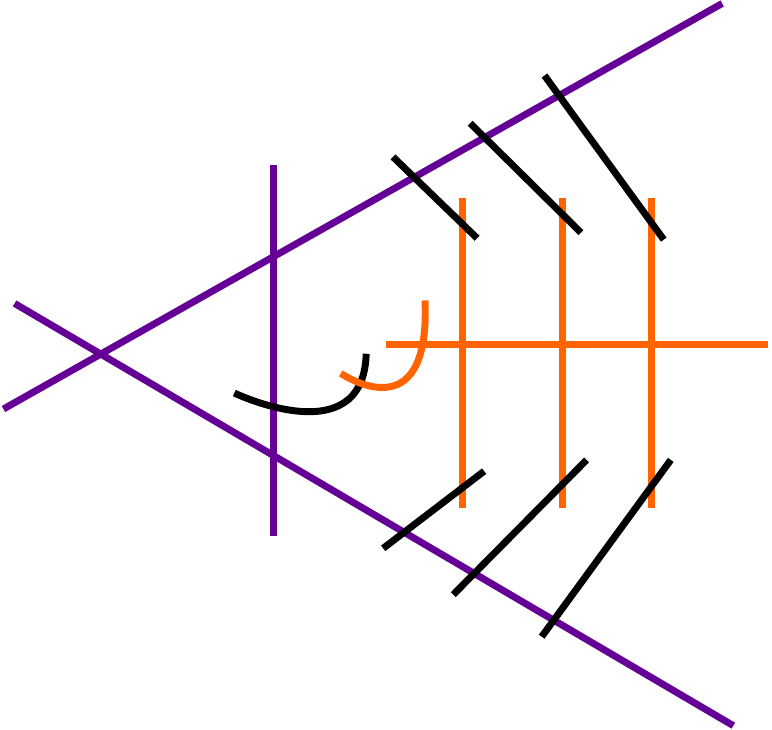} \label{fig:LP7}}
\subfloat[Singular fibers and section]{\includegraphics[scale=.4]{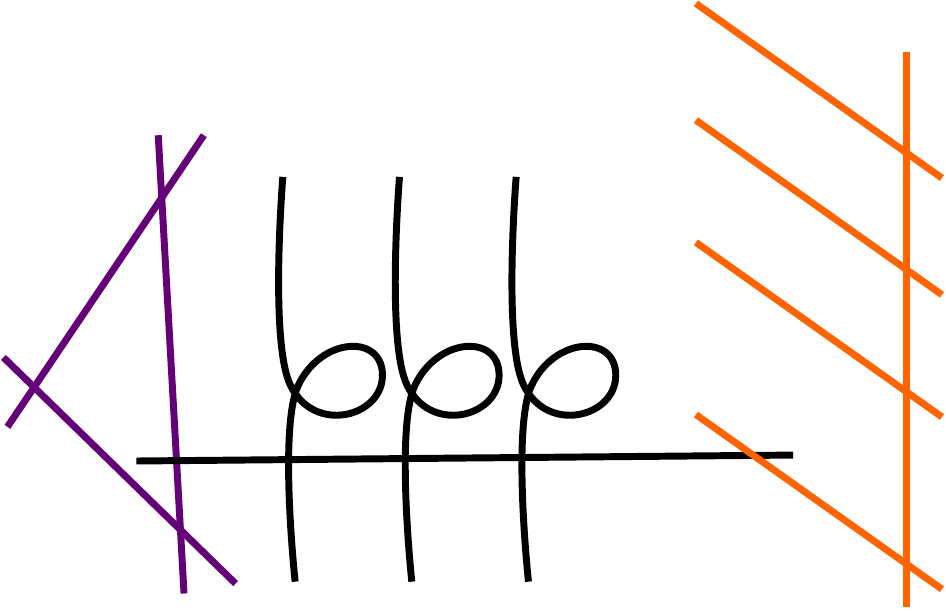} \label{fig:LP8}}
\caption{Blowing up the Lefschetz pencil to a fibration. The orange curves represent $C_1$ and the purple curves represent $C_2$.}
\label{fig:LP0}
\end{figure}
\end{proof}

Next, we will produce an elliptic fibration on $E(1)$ with two $I_2$ fibers and two $I_4$ fibers, and specify the homology classes of the spheres in the singular fibers.

\begin{figure}
	\includegraphics[width=0.60\textwidth]{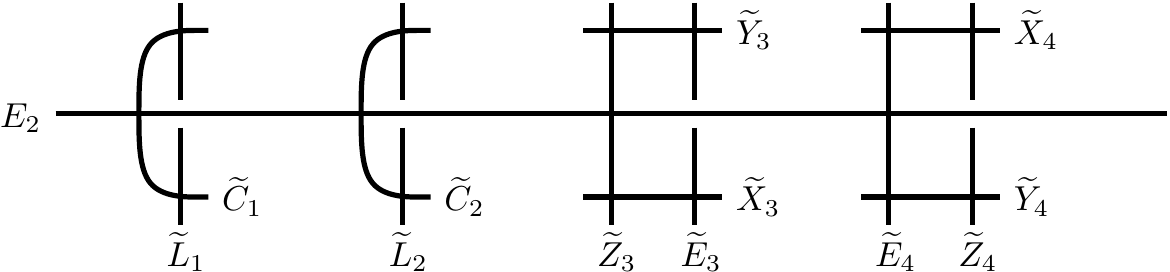}
	\caption{An elliptic fibration with two $I_2$ fibers and two $I_4$ fibers.}
	\label{fig:fibration2}
\end{figure}

\begin{lemma}\label{l:fibration2}
There exists an elliptic fibration on $E(1)\cong \CP2\# 9\barCP2$, with two $I_2$ fibers and two $I_4$ fibers such that the spheres in these singular fibers represent homology classes as follows.
\\
The first $I_2$ fiber is made up of two $-2$ spheres $\widetilde{C}_1$ and $\widetilde{L}_1$ such that \\
$\qquad \left[\widetilde{C}_1\right]=2h-e_1-e_2-e_3-e_4-e_7-e_8\qquad
\left[\widetilde{L}_1\right]= h-e_5-e_6-e_9$.\\
\\
The second $I_2$ fiber is made up of two $-2$ spheres $\widetilde{C}_2$ and $\widetilde{L}_2$ such that\\
$\qquad \left[\widetilde{C}_2\right]=2h-e_1-e_2-e_3-e_4-e_5-e_6\qquad
\left[\widetilde{L}_2\right]= h-e_7-e_8-e_9$.\\
\\
The first $I_4$ fiber is made up of four $-2$ spheres, $\widetilde{X}_3, \widetilde{Y}_3, \widetilde{Z}_3,\widetilde{E}_3$ with homology classes\\
$\left[\widetilde{X}_3\right]= h-e_3-e_6-e_8\qquad
\left[\widetilde{Y}_3\right]= h-e_3-e_5-e_7\qquad
\left[\widetilde{Z}_3\right]= h-e_1-e_2-e_9\qquad
\left[\widetilde{E}_3\right]=e_3-e_4$.\\
\\
The second $I_4$ fiber is made up of four $-2$ spheres, $\widetilde{X}_4, \widetilde{Y}_4, \widetilde{Z}_4,\widetilde{E}_4$ with homology classes\\
$\left[\widetilde{X}_4\right]= h-e_1-e_6-e_7\qquad
\left[\widetilde{Y}_4\right]= h-e_1-e_5-e_8\qquad
\left[\widetilde{Z}_4\right]= h-e_3-e_4-e_9\qquad
\left[\widetilde{E}_4\right]= e_1-e_2$.\\
Moreover this fibration admits a section whose homology class is $e_2$. See Figure \ref{fig:fibration2} for an illustration.
\end{lemma}
\begin{proof}
Consider the degree 3 homogeneous polynomials in three complex variables $p_1(x,y,z)=(yz+x^2)(y-z)$ and $p_2(x,y,z)=(yz+x^2)(y+z)$. Let $p_{(t_1,t_2)}(x,y,z)=t_1p_1(x,y,z)+t_2p_2(x,y,z)$. Then the degree three curves $T_{[t_1:t_2]} := \{p_{(t_1,t_2)}=0\}\subset \CP2$ are the fibers of a Lefschetz pencil on $\CP2$. Observe that $T_{[1:0]}=\{p_1=0\}$ and $T_{[0:1]}=\{p_2=0\}$ are reducible curves each made up of a linear part and an irreducible quadratic part. Let $C_1=\{yz-x^2=0\}$ denote the quadratic part of $T_{[1:0]}$ and $L_1=\{y-z=0\}$ denote the linear part. Similarly, let $C_2=\{yz+x^2=0\}$ and $L_2=\{y+z=0\}$ denote the quadratic and linear parts of $T_{[0:1]}$ respectively.

 We will keep track of two more fibers in this Lefschetz fibration:  $T_{[-1:1]}=\{(x+iz)(x-iz)y=0\}$ and $T_{[1:1]}=\{(y+ix)(y-ix)z=0\}$. These are reducible curves, each made up of three linear parts which we will label as $X_3=\{x+iz=0\}$, $Y_3 = \{x-iz=0\}$, $Z_3 =\{y=0\}$, $X_4 = \{x+iy=0\}$, $Y_4 = \{x-iy=0\}$, and $Z_4 = \{z=0\}$.
 
 Now we will blow-up at the intersection points of $T_{[1:0]}$ and $T_{[0:1]}$, and we would like to know how all of these curves intersect at those points and with what multiplicities, in order to determine their homology classes after blowing up the pencil. We summarize the relevant intersection data in Table \ref{tab:int}.

\begin{table} 
$$ \begin{array}{l|c|c|c|c|c|c|c|c|}
     & C_2 & L_2 & X_3 & Y_3 & Z_3 & X_4 & Y_4 & Z_4 \\\hline
C_1  & [0:0:1]_2, & [i:-1:1],  & [0:1:0],  & [0:1:0],  
& [0:0:1]_2  & [0:0:1],  & [0:0:1], & [0:1:0]_2\\
      & [0:1:0]_ 2 & [-i:-1:1] & [-i:-1:1] &  [i:-1:1]
&            & [i:-1:1]  & [-i:-1:1] &      
\\\hline
L_1  & [i:1:1]  & [1:0:0] & [-i:1:1] & [i:1:1] & [1:0:0] & 
[-i:1:1] & [i:1:1] & [1:0:0] \\
     & [-i:1:1] &         &          &         &         &
         &         &         
\\\hline
C_2  &   &   & [0:1:0]  & [0:1:0] & [0:0:1]_2 & [0:0:1]  & [0:0:1]
 & [0:1:0]_2\\ 
     &   &   & [-i:1:1] & [i:1:1] &           & [-i:1:1] & [i:1:1]
 &          
 \\\hline
 L_2 &   &   & [-i:-1:1] & [i:-1:1] & [1:0:0] & [i:-1:1] & [-i:-1:1] & [1:0:0] \\
 &&&&&&&&
 \\\hline
 X_3 &  &  &  & [0:1:0] & [-i:0:1] & [-i:1:1] & [-i:-1:1] & [0:1:0] \\
 &&&&&&&&
 \\\hline
 Y_3 &  &  &  &  & [i:0:1] & [i:-1:1] & [i:1:1] & [0:1:0] \\ 
 &&&&&&&&
 \\\hline
 Z_3 &  &  &  &  &  &  [0:0:1] & [0:0:1] & [1:0:0] \\
 &&&&&&&&
 \\\hline
 X_4 &  &  &  &  &  &  & [0:0:1] & [-i:1:0] \\
  &&&&&&&&
  \\\hline
 Y_4 &  &  &  &  &  &  &  & [i:1:0]\\
   &&&&&&&&
   \\\hline
 \end{array}$$
\caption{The intersection data of the curves we track through the pencil. Multiplicities greater than one are indicated by subscripts.}
\label{tab:int}
\end{table}

To obtain an elliptic fibration, we blow up at the intersection points of $C_1\cup L_1$ with $C_2\cup L_2$. By Table \ref{tab:int}, the relevant points are: 
\begin{eqnarray*}
\left[0:0:1\right]&=&C_1\cap C_2 \cap Z_3 \cap X_4 \cap Y_4\\
\left[0:1:0\right]&=&C_1\cap C_2 \cap X_3 \cap Y_3 \cap Z_4\\
\left[i:1:1\right]&=&L_1\cap C_2 \cap Y_3 \cap Y_4\\
\left[-i:1:1\right] &=& L_1\cap C_2 \cap X_3 \cap X_4\\
\left[i:-1:1\right] &=& C_1\cap L_2 \cap Y_3 \cap X_4\\
\left[-i:-1:1\right] &=& C_1 \cap L_2 \cap X_3 \cap Y_4\\
\left[1:0:0\right] &=& L_1\cap L_2 \cap Z_3 \cap Z_4\\
\end{eqnarray*}
Observe that $[0:0:1]$ and $[0:1:0]$ appear as intersection points of multiplicity two in $C_1,C_2,Z_3$ and $C_1,C_2,Z_4$ respectively, but all other intersections are transverse. We will need to blow up twice at the multiplicity two points, and once at each other point, to eliminate the base locus of the pencil.

We will denote the generator of $H_2(\CP2;\Z)$ by $h$. Note that the homology class represented by one of the listed curves is $h$ if the curve is linear, and $2h$ if the curve is quadratic.

Blow-up, introducing the exceptional sphere $E_1$ (with homology class $e_1$) at $[0:0:1]$. Then the proper transforms $\widetilde{X}_4$ and $\widetilde{Y}_4$ intersect $E_1$ at distinct points, and $\widetilde{C}_1$, $\widetilde{C}_2$, and $\widetilde{Z}_3$ all intersect at a common third point on $E_1$. Note that while $X_4\cup Y_4 \cup Z_4$ was a fiber of the Lefschetz pencil, the homology class $e_1$ appears with multiplicity two in its proper transform, while it appears with multiplicity one in the proper transforms of the other fibers. Therefore the Lefschetz pencil on the blown-up manifold now has a fiber  $\widetilde{X}_4\cup \widetilde{Y}_4\cup \widetilde{Z}_4 \cup E_1$ which includes the exceptional sphere with multiplicity one so that the fibers all represent the same homology class. 

Next, we blow up at $\widetilde{C}_1\cap \widetilde{C}_2\cap \widetilde{Z}_3\cap E_1$, introducing a new exceptional sphere $E_2$ which intersects the proper transforms $\widetilde{C}_1$, $\widetilde{C}_2$, $\widetilde{Z}_3$, and $\widetilde{E}_1$. At this point the homology classes of all proper transforms are as follows.
$[\widetilde{C}_1]=2h-e_1-e_2$, $[\widetilde{L}_1]=h$, $[\widetilde{C}_2]=2h-e_1-e_2$, $[\widetilde{L}_2]=h$, $[\widetilde{X}_3]=h$, $[\widetilde{Y}_3]=h$, $[\widetilde{Z}_3]=h-e_1-e_2$,$[\widetilde{X}_4]=h-e_1$, $[\widetilde{Y}_4]=h-e_1$, $[\widetilde{Z}_4]=h$, and $[\widetilde{E}_1]=e_1-e_2$. Note that $\widetilde{E}_1$ is included with multiplicity one in a fiber with $\widetilde{X}_4\cup \widetilde{Y}_4\cup \widetilde{Z}_4$ 

A similar situation occurs at $[0:1:0]$. We blow-up two new exceptional spheres represented homology classes $e_3$ and $e_4$ at this point. This time we must include $\widetilde{E}_3$ with multiplicity one in a fiber with $\widetilde{X}_3\cup \widetilde{Y}_3\cup \widetilde{Z}_3$. 

Finally we blow-up once at the points $[i:1:1], [-i:1:1], [i:-1:1], [-i:-1:1]$, and $[1:0:0]$, introducing exceptional homology classes $e_5,e_6,e_7,e_8$, and $e_9$ respectively. The homology classes of the proper transforms of the relevant curves in the four singular fibers are given as in the statement of the proposition.




\end{proof}

Finally, we construct an elliptic fibration with two $I_5$ fibers and two fishtails.

\begin{lemma}\label{l:fibration3}
There is an elliptic fibration on $E(1)=\CP2\#9\barCP2$ with two $I_5$ fibers two fishtails and a section. Both $I_5$ fibers are made up of five $-2$ spheres $C_1,\dots,C_5$, $D_1,\dots,D_5$ with homology classes
\begin{align*}
&[C_1]=e_1-e_6, &[D_1]&=e_4-e_9,\\
&[C_2]=h-e_1-e_4-e_9, &[D_2]&=h-e_3-e_4-e_5,\\
&[C_3]=h-e_2-e_5-e_8, &[D_3]&=e_5-e_8,\\
&[C_4]=e_2-e_7, &[D_4]&=h-e_1-e_5-e_6,\\
&[C_5]=h-e_1-e_2-e_3, &[D_5]&=h-e_2-e_4-e_7.
\end{align*}
\end{lemma}

\begin{proof}
Start with a configuration of six complex projective lines, intersecting as shown in figure \ref{fig:pencil1}. We can view these lines as two degenerate cubic curves indicated by distinct colors in figure \ref{fig:pencil1}, which generate a Lefschetz pencil on $\CP2$.
\begin{figure}[h]
	\includegraphics[width=.30\textwidth]{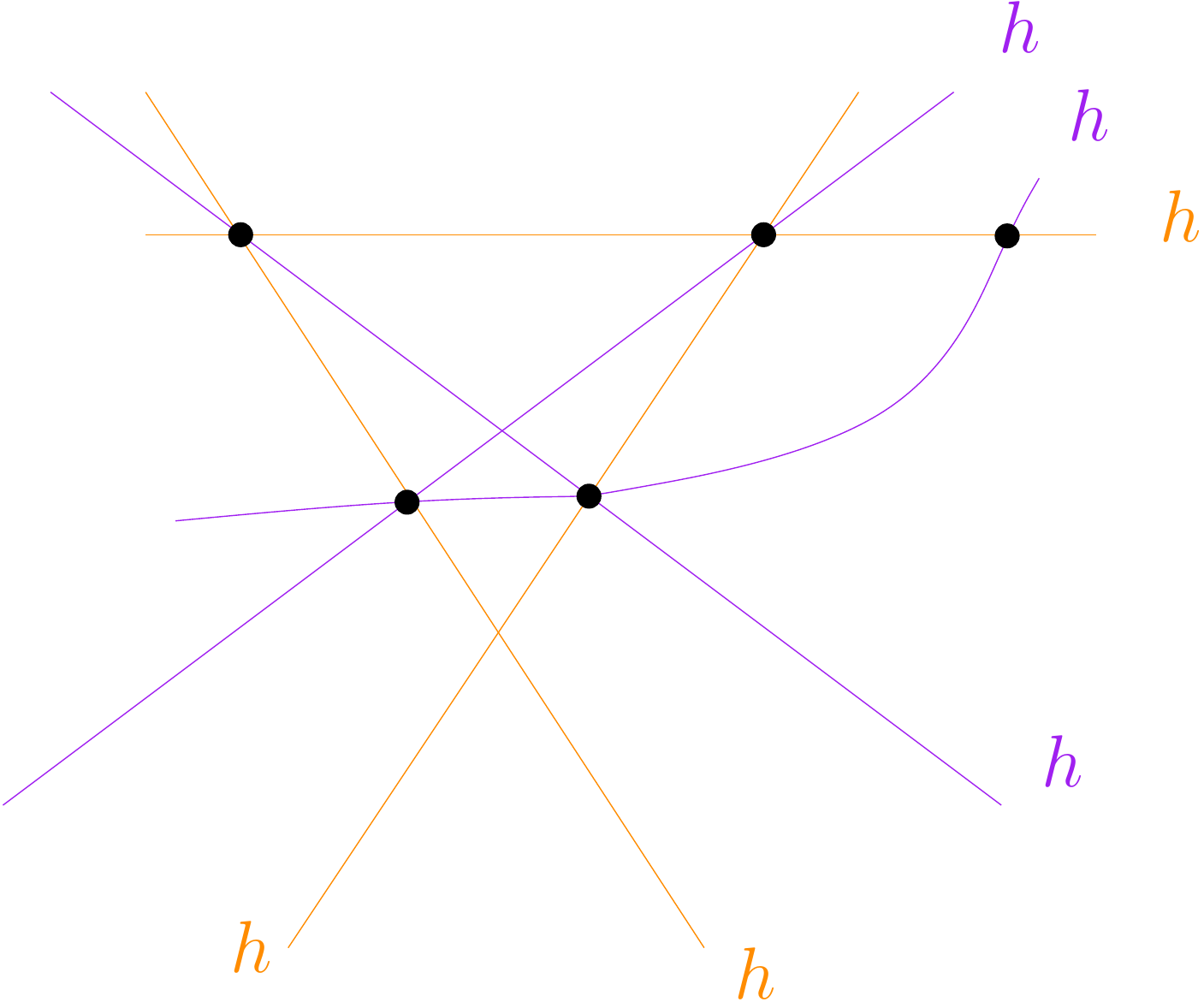}
	\caption{Initial configuration of lines}
	\label{fig:pencil1}
\end{figure}

Blowing up this pencil along the five intersection points between the orange and purple curves yields the configuration in figure \ref{fig:pencil2}. In order to keep the homology classes of the two curves the same so they continue to define a pencil, two of the exceptional spheres must be included with multiplicity one in the orange curve, and two must be included with multiplicity one in the purple curve. The resulting curves defining a pencil on $\CP2\#5\barCP2$ intersect in four distinct points, and blowing up at each of these points yields an elliptic fibration shown in figure \ref{fig:pencil3}. 
\begin{figure}[h]
	\includegraphics[width=0.40\textwidth]{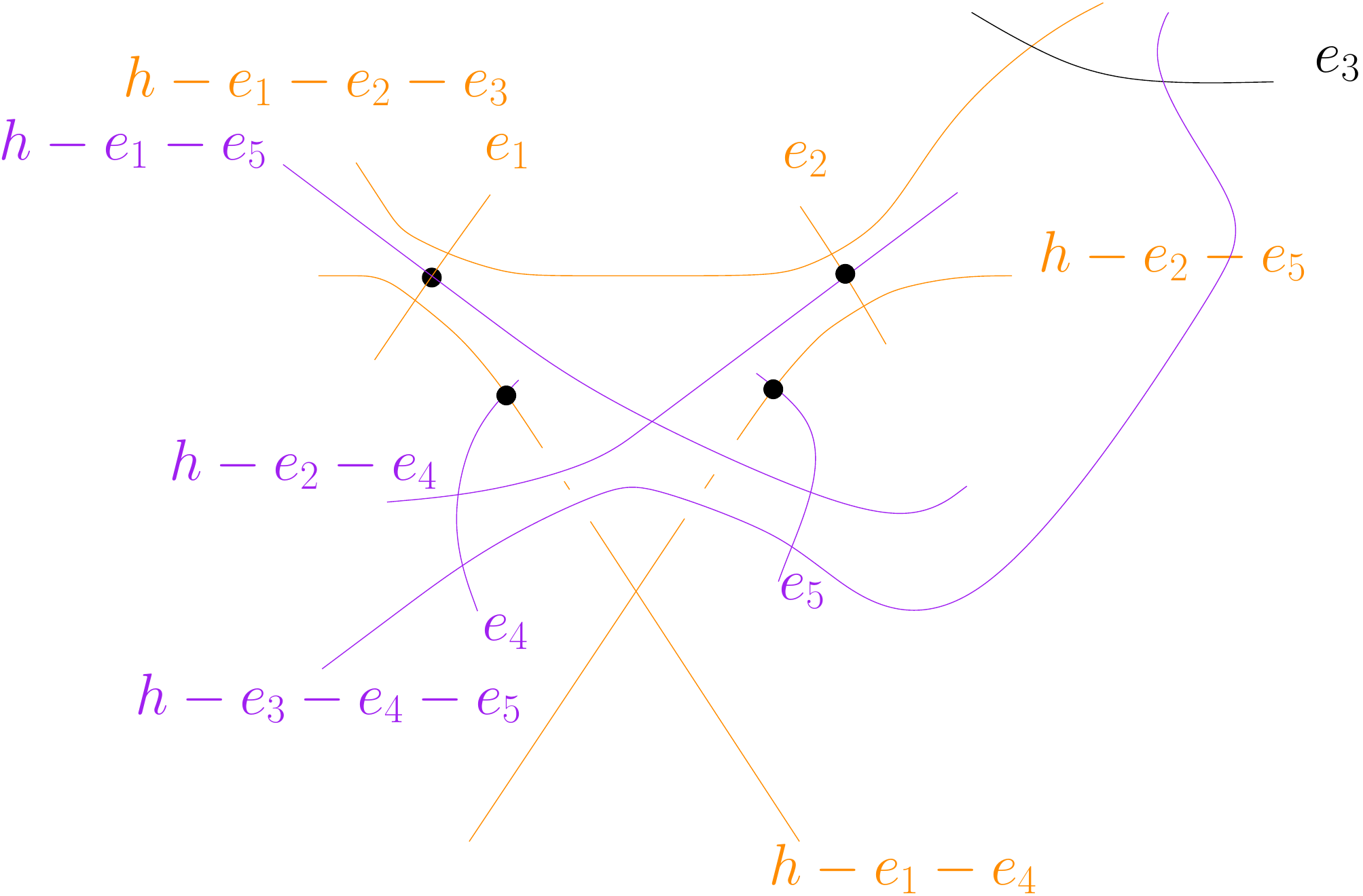}
	\caption{The after first five blow ups}
	\label{fig:pencil2}
\end{figure}

\begin{figure}[h]
	\includegraphics[width=0.40\textwidth]{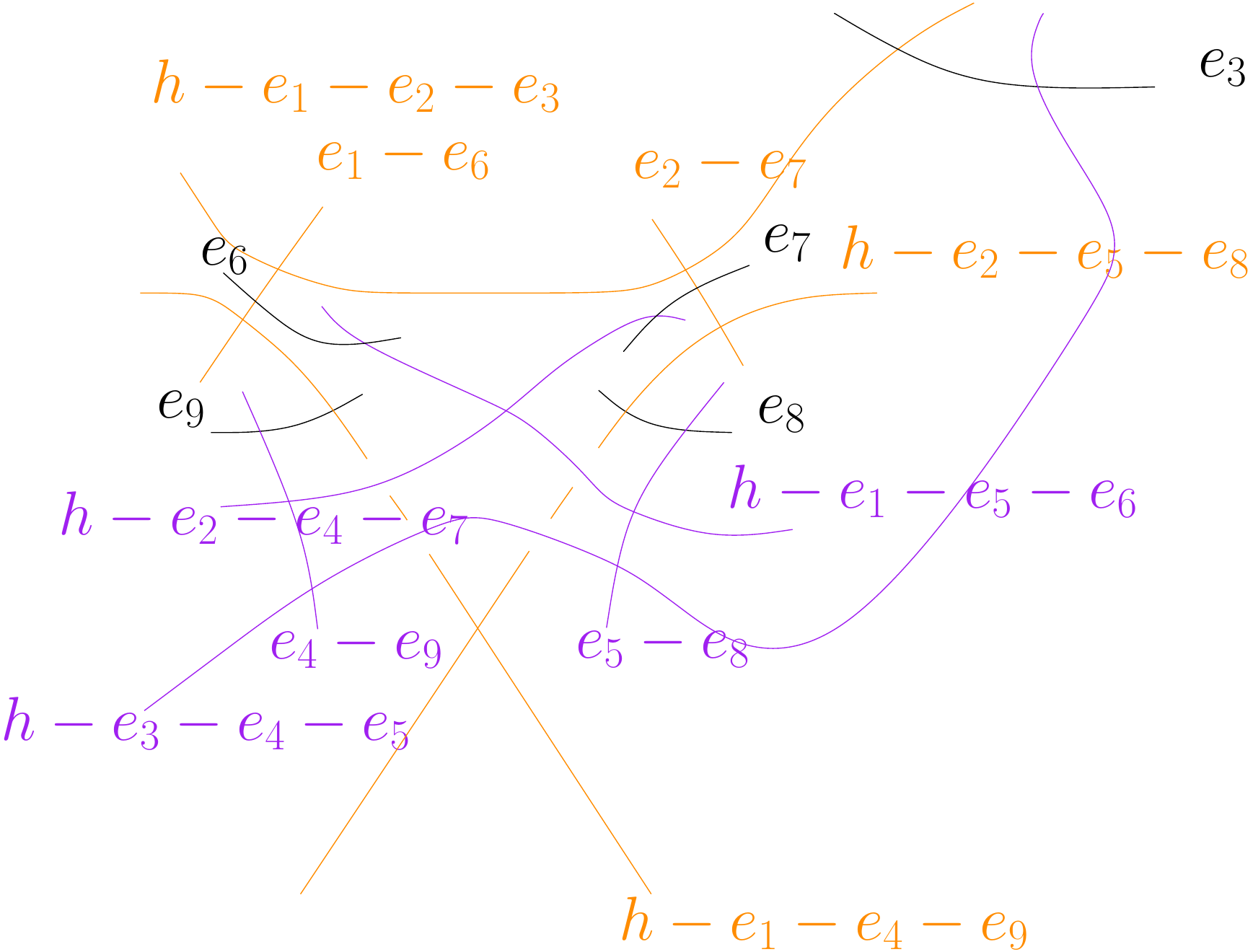}
	\caption{Elliptic fibration}
	\label{fig:pencil3}
\end{figure}

The proper transforms of the original curves defining the pencil become the $I_5$ fibers representing the specified homology classes. Any other singular fibers can be perturbed to be generic nodal (fishtail) singular fibers, and an Euler characteristic computation indicates there are two of these.

\end{proof}

\section{Constructions of Symplectic Exotic 4-manifolds} \label{s:exoticconstructions}

In this section we put together the information from the previous sections, to construct small exotic symplectic manifolds using star surgery. In the first construction, we will perform a full analysis of the smooth invariants of the resulting manifold. In the subsequent examples, we provide abridged computations that suffice to prove that the examples are exotic copies of $\CP2\# N\barCP2$.

\subsection{A Symplectic Exotic $\CP2 \# 8 \barCP2$}\label{s:8}

\subsubsection{The Construction} \label{s:con} Our first construction of an exotic $\CP2 \# 8 \barCP2$ uses the elliptic fibration given in Lemma \ref{l:fibration}, which has an $I_3$ fiber and $I_0^*$ fiber, three fishtail fibers, and a section, with homology classes specified in the proof of the lemma.

\begin{lemma}\label{l:classes}
The configuration $\St_2$ symplectically embeds into $\CP2 \# 11 \barCP2$ such that its vertices represent the following homology classes.
\begin{align*}
[u_0]&=2f+e_1-2e_{10}-2e_{11}\\
[u_1]&=h-e_1-e_2-e_3\\
[u_2]&=h-e_1-e_4-e_5\\
[u_3]&=h-e_1-e_6-e_7\\
[u_4]&=h-e_1-e_8-e_9\\
\end{align*}
where $f=3h-(e_1+\dots +e_9)$.
\end{lemma}

\begin{proof}
Consider the elliptic fibration constructed in Lemma \ref{l:fibration}. In this fibration we see the following symplectic spheres:
\begin{itemize}
\item The components $S_1,\dots,S_5$ of the $I^*_0$ fiber, with homology classes $[S_1]=h-e_1-e_2-e_3$, $[S_2]=h-e_1-e_4-e_5$, $[S_3]=h-e_1-e_6-e_7$, $[S_4]=e_1-e_8$ and $[S_5]=e_8-e_9$.
\item The component $V_1$ of the $I_3$ fiber with homology class $[V_1]=h-e_1-e_8-e_9$.
\item The exceptional sphere $E_9$ which is a section.
\end{itemize}
Take two fishtail fibers and blow-up their double points. The proper transforms $F_1$ and $F_2$ are now symplectic spheres in $\CP2\# 11\barCP2$ .Their cohomology classes are $[F_1]=f-2e_{10}$ and $[F_2]=f-2e_{11}$.

Take the union of spheres $u_0=F_1\cup F_2\cup   E_9 \cup S_4 \cup S_5 $. After symplectically smoothing its double points, $u_0$ gives a symplectic sphere whose homology class is $[u_0]=2f+e_1-2e_{10}-2e_{11}$. Let $u_1=S_1$, $u_2=S_2$, $u_3=S_3$, $u_4=V_1$. Then the union $\bigcup_{i=0}^4 u_i$ gives the required embedding of $\St_2$. 
\begin{figure}
\begin{center}
\includegraphics[scale=.5]{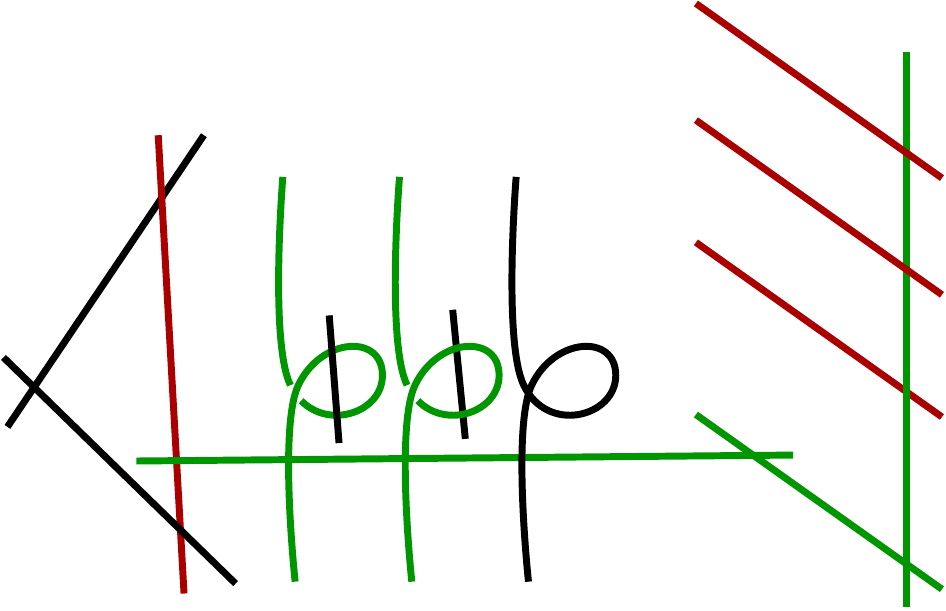}
\end{center}
\caption{A blow-up of the elliptic fibration of lemma \ref{l:fibration} shown in figure \ref{fig:LP0}. The symplectic resolution of the union of the green curves is the sphere $u_0$, and the four red curves are the spheres $u_1,\cdots , u_4$.}
\label{fig:Fibration}
\end{figure}
\end{proof}

Using this embedding of $\St_2$, perform star surgery on $\CP2\# 11\barCP2$ by cutting out this embedded $\St_2$ and replacing it with $\T_2$ resulting in a manifold
$$X= ((\CP2\#11\barCP2)\setminus \St_2)\cup_\partial \T_2.$$

\begin{lemma}\label{l:homeo}
The manifold $X$ is homeomorphic to $\CP2 \# 8 \barCP2$.
\end{lemma}

\begin{proof}
First we show the manifold $X$ is simply connected. By Proposition \ref{p:fund}, it suffices to show that the dual circle of one of the $-2$ spheres on $\partial \St_2$ bounds a disk in $\CP2 \# 11 \barCP2 \setminus \St_2$. By our construction the sphere $u_4$ is  a part of the $I_3$ fiber. Hence the meridian of $u_4$ bounds a disk $D$ in the complement of $\St_2$ which is contained in a sphere component of the $I_3$ fiber transversely intersecting $u_4$.   

Next we show that the intersection form of $X$ is isomorphic to the intersection form of $\CP2 \# 8 \barCP2$. We calculate the Euler characteristic and signature to get
\begin{align*}
\chi(X)&= \chi(\CP2 \# 11 \barCP2) - \chi (\St_2) + \chi(\T_2)\\
&=11,\\
\sigma (X) &= \sigma (\CP2 \# 11 \barCP2) -\sigma (\St_2) +\sigma(\T_2)\\
&=-7 .
\end{align*}
These imply that $b_2^+(X)=1$ and $b_2(X)=9$. The intersection form cannot be even, otherwise it would be written as a direct sum of hyperbolic pieces and $E_8$s, but the values of $b_2^+(X)$ and  $b_2(X)$ say that such a decomposition is impossible.  Hence $X$ is homeomorphic to $\CP2 \# 8 \barCP2$ by Freedman's theorem \cite{Fre}.
\end{proof}

\subsubsection{Kodaira dimension}\label{s:kod}

In this subsection we compute the symplectic Kodaira dimension \cite{Li} of $X$. Along the way, we distinguish $X$ from $\CP2\#8\barCP2$. Our argument is similar to Park's \cite{P}.

\begin{definition}\cite{Li}
For a minimal symplectic $4$-manifold $(M,\omega)$ with symplectic canonical class $K_\omega$, the Kodaira dimension of $(M,\omega)$ is defined in the following way
$$\kappa(M,\omega)=\left \{
\begin{array}{lr}
-\infty & \text{if } K_\omega\cdot[\omega] < 0 \text{ or } K_\omega\cdot K_\omega<0,\\
0 & \text{if } K_\omega\cdot[\omega] = 0 \text{ and } K_\omega\cdot K_\omega=0,\\
1 & \text{if } K_\omega\cdot[\omega] > 0 \text{ and } K_\omega\cdot K_\omega=0,\\
2 & \text{if } K_\omega\cdot[\omega] > 0 \text{ and } K_\omega\cdot K_\omega>0.
\end{array} \right .
$$
The Kodaira dimension of a non-minimal manifold is defined to be that of any of its symplectic minimal models.
\end{definition}

It is known that the Kodaira dimension is well-defined for every closed symplectic $4$-manifold and depends only on the diffeomorphism type of the manifold.  

 Let $K$ and $\omega$ denote the canonical class and the symplectic class of $\CP2 \# 11 \barCP2$ respectively. Let $X_0=\CP2 \# 11\barCP2 \setminus \St_2$. Then
\begin{align*}
K&=K|_{X_0} +K|_{\St_2},\\
\omega&=\omega|_{X_0} +\omega|_{\St_2}.
\end{align*}
Hence
\begin{align*}
K\cdot\omega = K|_{X_0} \cdot \omega|_{X_0} + K|_{\St_2}\cdot \omega|_{\St_2}. 
\end{align*}
Let $K_X$, and $\omega_X$ denote the canonical class and the symplectic class of $X$ respectively. Then
\begin{align*}
K_X&=K|_{X_0} +K_{\T_2},\\
\omega_X&=\omega|_{X_0} +\omega|_{\T_2}.
\end{align*}
Since $K_{\T_2}=0$ by Proposition \ref{p:can}, we have
\begin{align*}
K_X\cdot \omega_X= K|_{X_0} \cdot \omega|_{X_0}= K\cdot\omega - K|_{\St_2}\cdot \omega|_{\St_2}.
\end{align*}
\begin{lemma}\label{l:symclass}
For every $k>0$, the manifold $\CP2\# k\barCP2$ admits a symplectic structure whose cohomology class is given by $\omega=ah-b_1e_1-\dots b_ke_k$ for some rational numbers  $a,b_1,\dots,b_k$ with $a>b_1>\dots>b_k$ and $a>b_1+\dots+b_k$. 
\end{lemma}
\begin{proof}
Note that $a$ is the symplectic area of $\CPo \subset \CP2$ and $b_j$ is the symplectic area of the exceptional sphere $e_j$ for all $j=1,\dots,k$. By \cite[Chapter 7]{MS}, we have $b_j=\lambda^2$ where $\lambda$ is the weight of the corresponding blow-up. In other words $\lambda$ is the radius of the Darboux ball which will be removed during the blow-up process. Since the weights of the blow-ups can be chosen to be arbitrarily small, the result follows.
\end{proof}

\begin{lemma}\label{l:kodaira}
We have $K_X\cdot \omega_X>0$.
\end{lemma}

\begin{proof}
Write
\begin{align*}
K&= -3h +e_1+\dots+e_{11},\\
\omega&=ah-b_1e_1-\dots - b_{11} e_{11},
\end{align*}
where $a,b_1,\dots, b_{11}$ are rational numbers such that $a>b_1>\dots>b_{11}$ and 
\begin{align}\label{e:cond}
a>b_1+\dots+b_{11}.
\end{align} 
Then we have 
\[
K\cdot \omega=-3a+b_1+\dots+b_{11}.
\]
Let $\gamma_0,\dots,\gamma_4$ be the basis of $H^2(\St_2,\mathbb{Q})$ which is dual to $u_0,\dots,u_4$. Then the adjunction formula implies
\begin{align*}
K|_{\St_2}&=(K\cdot u_0)\gamma_0+ \dots +(K\cdot u_4)\gamma_4\\
&=3\gamma_0.
\end{align*}
We calculate the restriction of the symplectic class on $\St_2$ using Lemma \ref{l:classes} 
\begin{align*}
\omega|_{\St_2}&=(\omega\cdot u_0)\gamma_0+ \dots +(\omega\cdot u_4)\gamma_4\\
&=(6a-b_1-2b_2-\dots-2b_{11})\gamma_0 + (a-b_1-b_2-b_3)\gamma_1\\
&\quad + (a-b_1-b_4-b_5)\gamma_2 + (a-b_1-b_6-b_7) \gamma_3 + (a-b_1-b_8-b_9)\gamma_4.
\end{align*}
Let $P$ denote the intersection matrix for $\St_2$ then
\[ 
P^{-1}=-\frac{1}{12}
\left( \begin{array}{ccccc}
4 & 2 & 2 & 2 & 2 \\
2 & 7 & 1 & 1 & 1 \\
2 & 1 & 7 & 1 & 1 \\ 
2 & 1 & 1 & 7 & 1 \\
2 & 1 & 1 & 1 & 7 \\
\end{array} \right).\] 
Hence 
\begin{align*}
K|_{\St_2}\cdot \omega|_{\St_2} &= -\frac{1}{4} [ 4(6a-b_1-2b_2\dots-2b_{11}) + 2(a-b_1-b_2-b_3)\\
&\quad + 2(a-b_1-b_4-b_5) + 2(a-b_1-b_6-b_7) + 2(a-b_1-b_8-b_9)  ]\\
&=-\frac{1}{4}[32a-12b_1-10b_2-10b_3-\dots - 10 b_9-8b_{10}-8b_{11}].\\
\end{align*}
Therefore
\begin{align*}
K_X\cdot \omega_X&= K|_{X_0} \cdot \omega|_{X_0}= K\cdot\omega - K|_{\St_2}\cdot \omega|_{\St_2}\\
&=-3a+b_1+\dots+b_{11}+[8a-3b_1-\frac{5}{2}b_2-\dots -\frac{5}{2}b_9-2b_{10}-2b_{11}]\\
&=5a-2b_1-\frac{3}{2}b_2-\dots - \frac{3}{2}b_9-b_{10}-b_{11}\\
&>0 \quad \text{ (By Equation (\ref{e:cond}))}.
\end{align*}
\end{proof}

\begin{remark}
	The important input which ensures there is a symplectic form on $X$ with $K_X\cdot \omega_X>0$, is the high multiplicity of the homology class $h$ in the classes $[u_i]$ of the spheres in the embedding of $\St_2$. Note that if the embedding of $\St_2$ were completely disjoint from the $\CPo$ representing $h$, then by McDuff's theorem, the resulting star surgered manifold would be a standard blow-up of $\CP2$. This computation provides a quantitative way of ensuring when the plumbing spheres intersect $\CPo$ enough for the star surgery to produce an exotic smooth structure.
\end{remark}

\begin{proposition}\label{p:exotic}
The manifold $X$ is not diffeomorphic to $\CP2\# 8 \barCP2$.
\end{proposition}
\begin{proof}
The standard symplectic form on $\CP2\# k \barCP2$ satisfies $K\cdot \omega <0$. According to [\cite{LL}, Theorem D] there is a unique symplectic structure on $\CP2\# k \barCP2$ for $2\leq k \leq 9$ up to diffeomorphism and deformation. Hence $\CP2\# 8 \barCP2$ does not admit a symplectic structure with $K\cdot \omega>0$. We have just seen that $K_X\cdot \omega_X>0$. Hence $X$ is not diffeomorphic to $\CP2\# 8 \barCP2$.
\end{proof}

\begin{proposition}\label{p:minimal}
The manifold $X$ is minimal
\end{proposition}

The proof is postponed until the next section. With this in hand, we are ready to prove about the Kodaira dimension of $X$.

\begin{proposition}\label{p:kodaira}
The manifold $X$ has symplectic Kodaria dimension $2$.
\end{proposition} 

\begin{proof} Recall that the the symplectic Kodaira dimension is defined on the minimal model. Since $X$ is minimal, $K_X\cdot \omega_X>0$, and $K_X^2=3\sigma (X) +2 \chi (X)=1 >0$, its symplectic Kodaira dimension is $2$.
\end{proof}

This proposition shows that $X$ can also be distinguished from $\CP2\# 8 \barCP2$ using the symplectic Kodaira dimension, since $\kappa (\CP2\# 8 \barCP2)=\kappa(\CP2)=-\infty$.

\subsubsection{Minimality}\label{s:min}\label{s:min}
{
Throughout this section, we assume that the reader is familiar with the Seiberg-Witten invariants of manifolds with $b_2^+=1$. See \cite{M, FS3, SS} for excellent expositions. 

In order to prove the minimality of $X$ one needs to know the effect of star surgery on Seiberg-Witten invariants. It suffices to know all the Seiberg-Witten basic classes of $X$.  The effect of such cut and paste operations on Seiberg-Witten invariants was studied by Michalogiorgaki in \cite{Mic} in a more general framework (see also \cite{R}, for an analogous result in the perspective of Heegaard Floer theory).
 
\begin{theorem}\cite[Theorem 1]{Mic}\label{theo:mic}
Suppose $Y$ is a rational homology sphere which is a  monopole $L$-space. Let $P$ and $B$ be negative definite $4$-manifolds with $b_1(P)=b_1(B)=0$ and $\partial P=\partial B= Y$. Let  $X=Z \bigcup _Y P$ and $X'=Z\bigcup_Y B$, for some $4$-manifold $Z$. If $\frak{s} \in \mathrm{Spin}^c (X), \frak{s}' \in \mathrm{Spin}^c (X')$, $d_X(\frak{s}),d_{X'}(\frak{s'})\geq 0$ 
and $\frak{s}|_{Z}=\frak{s'}|_{Z}$ then $SW_X(\frak{s})=SW_{X'}(\frak{s'})$.\\
\noindent In the case $b_2^+(X)=1$, $SW_{X,a_1}(\frak{s})=SW_{X',a_2}(\frak{s'})$, where $a_1\in H_2(X,\mathbb{Z}), a_2\in H_2(X',\mathbb{Z})$ specify chambers such that $a_1|_P=a_2|_B=0$ and $a_1|_Z=a_2|_Z$.
\end{theorem}

This result and the wall-crossing formula reduces the problem of determining basic classes of $X$ to a cohomology computation. In this section,  we incorporate a search method invented by Ozsv\'ath and Szab\'o to find all the basic classes of $X$. First we find a homology class that determines a common chamber for the manifolds before and after star surgery.

\begin{lemma}\label{l:vector1}
The element $V\in H_2(\CP2\#11\barCP2)$ defined by $$V=86h-36e_1-25e_2-25e_3-25e_4-25e_5-25e_6-25e_7-19e_8-31e_9-20e_{10}-20e_{11}$$
satisfies the following conditions:
\begin{enumerate}
 \item $V\cdot [u_i]=0$ for $i=0,\dots,4$ (i.e. $V$ is orthogonal to each embedded sphere of $\St_2$) \label{i:vector11}
 \item $V\cdot V>0$ \label{i:vector21}
 \item $V\cdot h>0$ \label{i:vector31}
 \item $V\cdot K>0$, where $K=-3h+e_1+\dots +e_{11}$ is the canonical class of $\CP2\#11\barCP2$. \label{i:vector41}
\end{enumerate}
\end{lemma}

\begin{proof}
The proof is a direct computation. 

\end{proof}

\begin{lemma} \label{l:SWcan} The small perturbation Seiberg-Witten invariant of $X$ at the canonical class $\widetilde{K}\in H^2(X,\mathbb{Z})$ is non-zero.
\end{lemma}

\begin{proof}
We first compute the Seiberg-Witten invariant of $\CP2\# 11 \barCP2$ in the chamber determined by the element $V$ in Lemma \ref{l:vector1}. Orient  $H_2^+(\CP2\#11\barCP2)$ with $h$. Note that the homology class $h$ gives the chamber of the positive scalar curvature metric hence all Seiberg-Witten invariants of $\CP2\# 11\barCP2$ are zero in this chamber. In particular $SW_{\CP2\# 11\barCP2,h}(K)=0$ for the canonical class $K$. By part (\ref{i:vector21}) of Lemma \ref{l:vector1}, $V\in H_2^+(\CP2\#11\barCP2)$, and by part (\ref{i:vector31}) $V$ has the correct orientation. Hence $V$ determines a chamber.  Since  $K\cdot h <0$, part (\ref{i:vector41}) says that there is a wall between the chambers of $h$ and $V$ with respect to the canonical class. By the wall-crossing formula \cite{LL},  $SW_{\CP2\# 11\barCP2,V}(K)=\pm 1$. 

Next we relate the Seiberg-Witten invariants of $\CP2\# 11\barCP2$ and $X$. The canonical class $\widetilde{K}$ of $X$ satisfies $\widetilde{K}_{X\setminus \T_2}=K_{\CP2\# 11\barCP2\setminus \St_2}$. Part (\ref{i:vector1}) says $V$ determines a chamber in $X$. If we can show that  $\partial \T_2$ is a monopole $L$-space then by Theorem \ref{theo:mic}, noting that $d_X(\widetilde{K})=d_{\CP2\# 11\barCP2}(K) =0$, we have $SW_{\CP2\# 11\barCP2,V}(K)= SW_{X,V}(\widetilde{K})$. Since there is a unique small perturbation chamber for manifolds with $b_2^-\leq 9$, the latter is equal to  $SW_{X}(\widetilde{K})$.

It remains to see that $\partial \T_2=\partial \St_2$ is a monopole $L$-space. In fact this holds for $\partial \St_i$ for all $i$. The plumbing graph of $\St_i$  satisfies $|m(v)|>d(v)$ at each vertex $v$, where $m(v)$ is the weight of $v$ and $d(v)$ is the number of edges connected to $v$. In [\cite{OS4}, Theorem 7.1], Ozsv\'ath and Szab\'o show that  the boundary of such a plumbing is a Heegaard Floer $L$-space. Their proof uses only the surgery exact triangle and the formal properties of Heegaard Floer homology. Hence it can be repeated to show that boundaries of such plumbings are also Monopole $L$-spaces. Alternatively, one can refer to the recently established equivalence of the Heegaard Floer homology and the Monopole Floer homology to see that every Heegaard Floer $L$-space is also a Monopole $L$-space \cite{KLT, CGH}. 
\end{proof}

\begin{remark}
Lemma \ref{l:SWcan} gives an alternative proof of Proposition \ref{p:exotic}. Indeed, the small perturbation Seiberg-Witten invariant is well defined for manifolds with $b_2^-\leq 9$ since there is a unique chamber for such manifolds. For $\CP2\# k\barCP2$ with $k\leq 9$ the unique chamber is the one given by the positive scalar curvature metric. Hence  $SW_{\CP2\# k\barCP2}(L)=0$ for all characteristic cohomology class $L$, for all $k\leq 9$.
\end{remark}

\begin{proof}[Proof of Proposition \ref{p:minimal}] 
By the blow-up formula for the Seiberg-Witten invariant, it suffices to show that the only basic classes of $X$ are $\pm \widetilde{K}$. Therefore, we want to check which integral characteristic cohomology classes (representing \Spinc structures) on $X$ are Seiberg-Witten basic classes. While we will show how to use Theorem 5.9 and the wall-crossing formula to compute the Seiberg-Witten invariant on a given cohomology class, there are infinitely many classes to check. The strategy of Ozsv\'{a}th and Szab\'{o} (which they used to prove minimality of a rationally blown-down manifold) is to check only the finitely many \emph{adjunctive classes} and then use the information about the adjunctive basic classes to rule out the possibility of non-adjunctive basic classes.

The fact that the homology of $\T_2$ has nontrivial rank makes the search somewhat more complicated than the case of rational blow-down. Even so, the computations can be handled by a simple computer program. First we find a basis for the subspace $H_2(\CP2\#11\barCP2)$ which is orthogonal to the homology of the cofiguration $\St_2$. The following elements form such a basis
$$A_1=h-e_3-e_5-e_7-e_9+e_{11},\qquad A_2:=-3h+2e_1+e_3+e_5+e_7+e_9+2e_{10}+2e_{11}$$
$$A_3=e_2-e_3, \qquad A_4=e_4-e_5, \qquad A_5=e_6-e_7, \qquad A_6=e_{10}-e_{11}, \qquad A_7=e_{8}-e_{9}.$$
\noindent Note that all of these homology classes can be represented by embedded spheres, excepting $A_2$ which can be represented by an embedded torus.  Moreover, all of these surfaces can be chosen in the complement  of $\St_2$ in $\CP2\#11\barCP2$. Additionally we let $A_{8}$ and $A_{9}$ be the homology generators of $\T_2$, which can be represented by embedded tori in $\T_2$, as indicated in Proposition \ref{p:hom}. Then $A_1,\cdots, A_9$ represent a basis for $H_2(X;\Z)$ so we can represent cohomology classes in $H^2(X;\Z)$ by a tuple of integers representing $[\langle L,A_1\rangle, \dots , \langle L,A_9\rangle ]$. 

First we will determine which of these tuples represent \emph{integral}, \emph{characteristic}, and\emph{adjunctive} cohomology classes. Because the inverse of the intersection matrix has rational coefficients, some tuples of integers could represent rational, but not integral, homology classes. We check a mod 2 equivalence to see if the cohomology class is characteristic and an adjunctive inequality for each $A_i$. Computationally, we find that there are exactly $243000$ cohomology classes $L\in H^2(X,\mathbb{Q})$ which satisfy these conditions for all $ i=1,\dots,9$,
$$\langle L, A_i \rangle  \in \mathbb{Z}, \qquad \langle L, A_i \rangle \equiv A_i^2 \; \;\mathrm{mod} \; 2, \qquad  |\langle L, A_i \rangle |  \leq -A_i^2.$$
In order to test whether a characteristic class is basic for $X$, we first check whether the expected dimension is nonnegative and even (necessary conditions for the Seiberg-Witten invariant to be non-zero). Let $d_X(L):=(L^2-3\sigma (X) -2 \chi(X))/4=(L^2-1)/4$. In the second round of our search, we check how many of these cohomology classes satisfy $d_X(L) \in \mathbb{Z}$, $d_X(L) \geq 0$ and $d_X(L) \equiv 0 \mod 2$. It turns out that there are $25040$ such classes.

For this collection of $25040$ classes, we will use Theorem \ref{theo:mic} and the wall-crossing formula to calculate the Seiberg-Witten invariant on each of those classes. To do this, we must relate each class on $X$ to a class on $\CP2\#11\barCP2$ so that the two classes have equal restrictions to $X\setminus \T_2=(\CP2\#11\barCP2)\setminus \St_2$. The difficulty is, once we restrict the chosen class to $X\setminus \T_2$, we must find a class on $\St_2$ which extends it over $\CP2\#11\barCP2$. For this, we compare the classes on $\partial \St_2$ which occur as the restriction of a class on $\St_2$ to the classes on $\partial \T_2=\partial \St_2$ that occur as the restriction of a class on $\T_2$. We showed that the image of the map $H^2(\T_2,\mathbb{Z})\to H^2 (\partial \T_2,\mathbb{Z})$ is an index $2$ subgroup, hence it has order $24$. Our aim is to find a  numerical criterion for a characteristic cohomology class $L\in H^2(\St_2,\mathbb{Z})$ to extend to $\T_2$ when restricted to $\partial \St_2 =\partial \T_2$. First observe that the set of \Spinc structures on $\partial \St_2$ as a $H^2(\partial \St_2,\mathbb{Z})$ torsor is isomorphic to $2H^2(\St_2,\partial \St_2,\mathbb{Z})$-orbits of characteristic elements in $H^2(\St_2,\mathbb{Z})$. Hence two characteristic classes $L_1,L_2\in H^2(\St_2,\mathbb{Z})$  restrict to the same \Spinc structure on the boundary if and only if $L_1=L_2+2PD(Z)$ for some $Z\in H_2(\St_2,\mathbb{Z})$. We will use $d$-invariants to determine which orbits in $H^2(\St_2,\Z)$ restrict to a \Spinc structure on $\partial \St_2=\partial \T_2$ which appears in the boundary of a \Spinc structure on $\T_2$ and to choose distinguished representatives of these orbits which maximize the expected dimension of the Seiberg-Witten moduli space. Define $d_{\St_2} (L)= (L^2-3\sigma(\St_2)-2\chi (\St_2))/4=(L^2+3)/4$. Observe that the $d_{\St_2}(L)\; \mathrm{mod}\; 2$ is constant in the orbit $L + 2H^2(\St_2,\partial \St_2,\mathbb{Z})$. Each orbit has a representative $L'$ satisfying $[u_i]^2 +2\leq \langle L', [u_i]  \rangle \leq  -[u_i]^2$ for all $i=0,\dots,4$, where each $u_i$ is a sphere appearing as a vertex in the star shaped plumbing graph of $\St_2$. Computing $d_{\St_2}$ for all of these representatives, we see that, the possible  $\mathrm{mod}\; 2$ reductions of $d_{\St_2}$ for  characteristic cohomology classes on $\St_2$ are $\{-2/3,-1/3, -1/4, -1/12, 0 , 1/4, 2/3,-11/12,1\}$. 

In the proof of Proposition \ref{p:hom}, we computed the intersection form of $\T_2$. From this, we see that the possible $\mod 2$ reductions of $d_{\T_2}$ for  characteristic cohomology classes on $\T_2$ are $\{-1/3,0,2/3,1\}$. We observe that exactly $24$ of the $2H^2(\St_2,\partial \St_2,\mathbb{Z})$ orbits have the $\mathrm{mod}\; 2$ reductions of their $d_{\St_2}$ belong to this set.  Hence we conclude that a characteristic cohomology class $L\in H^2(\St_2,\mathbb{Z})$ extends to $\T_2$ if and only if $d_{\St_2}(L)\in \{-1/3,0,2/3,1\}$ modulo $2$. We can explicitly write distinguished representatives of these $24$ orbits. They are the elements of the following set.
\begin{align} \label{e:phi} \Phi =\{&
[ 1, 0, 0, 0, 0 ],
[ -3, 2, 2, 2, 2 ],
[ -1, 2, 0, 0, 2 ],
[ -3, 0, 2, 2, 0 ],
[ 1, 2, 0, 2, 0 ],
[ 3, 0, 0, 0, 0 ],
[ -1, 2, 2, 0, 0 ],\\
&\nonumber [ -3, 2, 0, 2, 0 ],
[ -1, 0, 2, 0, 2 ],
[ -3, 0, 0, 2, 2 ],
[ 1, 0, 2, 2, 0 ],
[ 1, 0, 0, 2, 2 ],
[ -1, 0, 2, 2, 0 ],
[ -3, 0, 2, 0, 2 ],\\
&\nonumber [ -3, 0, 0, 0, 0 ],
[ 1, 2, 0, 0, 2 ],
[ 5, 0, 0, 0, 0 ],
[ -3, 2, 0, 0, 2 ],
[ -1, 0, 0, 2, 2 ],
[ 1, 0, 2, 0, 2 ],
[ -1, 2, 0, 2, 0 ],\\
&\nonumber [ -3, 2, 2, 0, 0 ],
[ 1, 2, 2, 0, 0 ],
[ -1, 0, 0, 0, 0 ] \}.
\end{align}
Here each cohomology class $L\in H^2(\St_2,\mathbb{Z})$ is represented by the  tuple $[\langle L, [u_0] \rangle , \dots, \langle L, [u_4] \rangle  ]$. Elements of $\Phi$ maximize $d_{\St_2}$ in their respective $2H^2(\St_2,\partial \St_2,\mathbb{Z})$ orbits.

We continue our search for basic classes of $X$. Recall that in round two, we got $25,040$ potential adjunctive basic classes. We restrict each one to $X\setminus\T_2$ (i.e. we forget the intersections with $A_8$ and $A_9$). Each of these characteristic cohomology classes glues to exactly one element of the set $\Phi$ to define a characteristic cohomology class on $\CP2\#11\barCP2$. There are $600,960$ triples $(A,B,C)$ where $A$ is an adjunctive class on $\T_2$, $B$ is an adjunctive class on $X\setminus \T_2=\CP2\#11\barCP2 \setminus \St_2$,   $C$ is one of the 24 distinguished cohomology classes on $\St_2$, and the pair $(B,A)$ is one of the 25,040 potential adjunctive basic classes on $X$. We now restrict our attention to all triples where the combination $(B,C)$ represents a cohomology class on $\CP2\#11\barCP2$ with $d\in \Z$, $d\geq 0$ and $d\equiv 2 \mod 2$. This leaves us with $219,064$ possible triples. In only $122,212$ of these triples does $(B,C)$ represent an integral and characteristic cohomology class on $\CP2\#11\barCP2$. Finally, we check whether the chamber determined by $V$ in $\CP2\#11\barCP2$ is the same as the chamber of the positive scalar curvature metric with respect to each of the characteristic cohomology classes $(B,C)$ in the remaining triples by comparing the signs of $(B,C)\cdot V$ with $(B,C)\cdot H$. When the chambers agree, $B$ will not descend to a basic class on $X$, but when they disagree the wall-crossing formula and Theorem \ref{theo:mic} ensure that $(A,B)$ is a basic class on $X$. It turns out that only $2$ of the remaining triples $(A,B,C)$ have the property that $V$ is in a different chamber than $H$ with respect to $(B,C)$. These are necessarily the canonical class and its negative. Hence $\widetilde{K}$ and $-\widetilde{K}$ are the only adjunctive basic classes of $X$.

Lastly we argue that there can be no non-adjunctive basic class. If there was such an $L$, then the adjunction relations \cite{OS2} imply that adding or subtracting twice the Poincare dual of any surface with negative self intersection where the adjunction inequality fails, we would obtain another basic class $L'$ with $d_X(L')>d_X(L)$. 

Since $X$ has only finitely many basic classes this process eventually stops at an adjunctive basic class, $\overline{L}$ with $d_X(\overline{L})>d_X(L)\geq 0$, but this is a contradiction because $\widetilde{K}$ and $-\widetilde{K}$ are the only adjunctive basic classes and $d_X(\widetilde{K})=d_X(-\widetilde{K})=0$.
\end{proof}

\begin{proof}[Proof of Theorem \ref{theo:main}]
Follows from Lemma \ref{l:homeo}, Proposition \ref{p:exotic},  Proposition \ref{p:minimal}, and Proposition \ref{p:kodaira}.
\end{proof}

\begin{remark}
The Seiberg-Witten invariants of $X$ agree up to isomorphism with the all known examples of minimal exotic $\CP2\#8\barCP2$ \cite{P, SS, FS2, Mic}. Hence we cannot distinguish $X$ from these manifolds using Seiberg-Witten invariants.
\end{remark}

}
\subsection{An exotic $\CP2\#7\barCP2$} \label{s:morestar}
\begin{lemma}\label{l:embedS}
There is an embedding of the star-shaped plumbing $\mathcal{Q}$ into $\CP2\#12\barCP2$ such that the spheres represent the following homology classes:
\begin{align*}
[u_0]&=6h-2e_1-e_2-2e_3\cdots -2e_9-2e_{11}-2e_{12}\\
[u_{1,1}]&=2h-e_1-e_2-e_3-e_4-e_5-e_6\\
[u_{2,1}]&=h-e_1-e_2-e_9\\
[u_{2,2}]&=h-e_3-e_5-e_7\\
[u_{3,1}]&=h-e_1-e_6-e_7\\
[u_{3,2}]&=h-e_3-e_4-e_9-e_{10}\\
[u_{4,1}]&=h-e_1-e_5-e_8-e_{10}\\
\end{align*}
\end{lemma}

\begin{proof}
In Lemma~\ref{l:fibration2}, we showed that there is an elliptic fibration with two $I_2$ fibers, two $I_4$ fibers, and a section, and we specified the homology classes of the spheres making up the singular fibers and the section. According to that notation, the spheres in the $I_2$ fibers were $\widetilde{C}_1$, $\widetilde{L}_1$ and $\widetilde{C}_2$, $\widetilde{L}_2$. The spheres in the $I_4$ fibers were denoted $\widetilde{X}_3$, $\widetilde{Y}_3$, $\widetilde{Z}_3$, $\widetilde{E}_3$, and $\widetilde{X}_4$, $\widetilde{Y}_4$, $\widetilde{Z}_4$, $\widetilde{E}_4$. $\widetilde{C}_1$, $\widetilde{C}_2$, $\widetilde{Z}_3$, and $\widetilde{E}_4$ intersected the section $E_2$.

To construct the embedding of $\mathcal{Q}$, first blow-up at the intersection of $\widetilde{Y}_4$ and $\widetilde{Z}_4$ to produce two adjacent $-3$ spheres. Then we can perturb the fibration near one of the $I_2$ fibers to split it into two fishtail fibers. Blow-up each the self-intersection points of each of these fishtail fibers to create two $-4$ spheres. To produce the central $-5$ sphere, take the symplectic resolution of the section $E_2$, the two blown-up fishtail fibers, and $E_4$ (the sphere in the blown-up $I_4$ fiber which intersects the section). The remaining spheres can be taken to be $u_{1,1}=\widetilde{C}_2$, $u_{2,1}=\widetilde{Z}_3$, $u_{2,2}=\widetilde{Y}_3$, $u_{3,1}=\widetilde{X}_4$, $u_{3,2}=\overline{Z_4}$, $u_{4,1}=\overline{Y_4}$ (where $\overline{C}$ denotes the proper transform under blow-up). Using the homology computation indicated in  Lemma~\ref{l:fibration2}, the homology classes for the $u_{i,j}$ follows.
\end{proof}

\begin{lemma}\label{l:vector}
The element $R\in H_2(\CP2\#12\barCP2)$ defined by $$R=533h-188e_1-186e_2-192e_3-126e_4-185e_5-189e_6-156e_7-104e_8-159e_9-56e_{10}-0e_{11}-151e_{12}$$
satisfies the following conditions:
\begin{enumerate}
 \item $R\cdot [u_{i,j}]=0$ for all $i$ and $j$, \label{i:vector12}
 \item $R\cdot R>0$ \label{i:vector22},
 \item $R\cdot h>0$ \label{i:vector32},
 \item $R\cdot K>0$, where $K=-3h+e_1+\dots +e_{12}$ is the canonical class of $\CP2\#12\barCP2$. \label{i:vector42}
\end{enumerate}
\end{lemma}

\begin{proof}
All of the above claims can be verified by a direct computation.
\end{proof}

\begin{proposition}\label{p:ex7}
The result of star surgery on this embedding of $\mathcal{Q}$ into $\CP2 \#12\barCP2$ is an exotic copy of $\CP2\#7\barCP2$ which supports a symplectic structure.
\end{proposition}

\begin{proof}
Let $X'=((\CP2\#12\barCP2)\setminus \mathcal{Q})\cup \mathcal{R}$ be the result of star-surgery. Then $\chi(X')=\chi(\CP2\#12\barCP2)-\chi(\mathcal{Q})+\chi(\mathcal{R})=15-8+3=10$. Since $\mathcal{Q}$ and $\mathcal{R}$ are both negative definite, $\sigma(X')=-6$. Since $\mathcal{R}$ is simply connected, $X'$ is simply connected. Therefore $X'$ is homeomorphic to $\CP2\#7\barCP2$ by Freedman's theorem \cite{Fre}.
Recall that for 4-manifolds with $b_2^+=1$, the Seiberg-Witten invariants depend upon a choice of a chamber. Let $h$ denote the homology class of $\mathbb{C}P^1 \subset \CP2\# 12\barCP2$. This class gives the chamber of positive scalar curvature metric, so $SW_{\CP2\# 12\barCP2,h}(L)=0$ for every characteristic class $L$. In particular for the canonical class $K$, we have $SW_{\CP2\# 12\barCP2,h}(K)=0$. By Lemma \ref{l:vector}, the class $R$ determines a chamber and there is a  wall between this chamber and the chamber of $h$ with respect to $K$. Hence by the wall crossing formula \cite{LL}, we have $SW_{\CP2\# 12\barCP2,R}(K)=\pm 1$. We will show that $X'$ also has a non-zero Seiberg-Witten invariant. Indeed the $3$-manifold $\partial S$ is a monopole $L$-space.  Let $K_{X'}$ denote the canonical class of $X'$. Clearly we have $K_{X'}|_{X'\setminus \mathcal{R}}=K|_{\CP2\# 12\barCP2 \setminus \mathcal{Q}}$ and $d_{X'}(K_{X'})=d_{\CP2\# 12\barCP2}(K)=0$. Hence by \cite{Mic}, we have $SW_{X',R}(K_{X'})=SW_{\CP2\# 12\barCP2,R}(K)=\pm 1$. For manifolds with $b_2^-\leq 9$, the choice of a chamber is unique. Therefore $X'$ is not diffeomorphic to $ \CP2\# 7\barCP2$ whose Seiberg-Witten invariants are all zero in the unique chamber of positive scalar curvature.
\end{proof}

\subsection{An exotic $\CP2\#6\barCP2$}\label{s:morestar2}

For this construction we use the elliptic fibration with two $I_5$ fibers, two fishtails, and a section, with homology classes specified in lemma \ref{l:fibration3}

\begin{lemma}
The configuration of spheres $\mathcal{U}$ embeds into $\CP2\#13\barCP2$.
\end{lemma}

\begin{proof}

\begin{figure}[h]
	\includegraphics[width=0.40\textwidth]{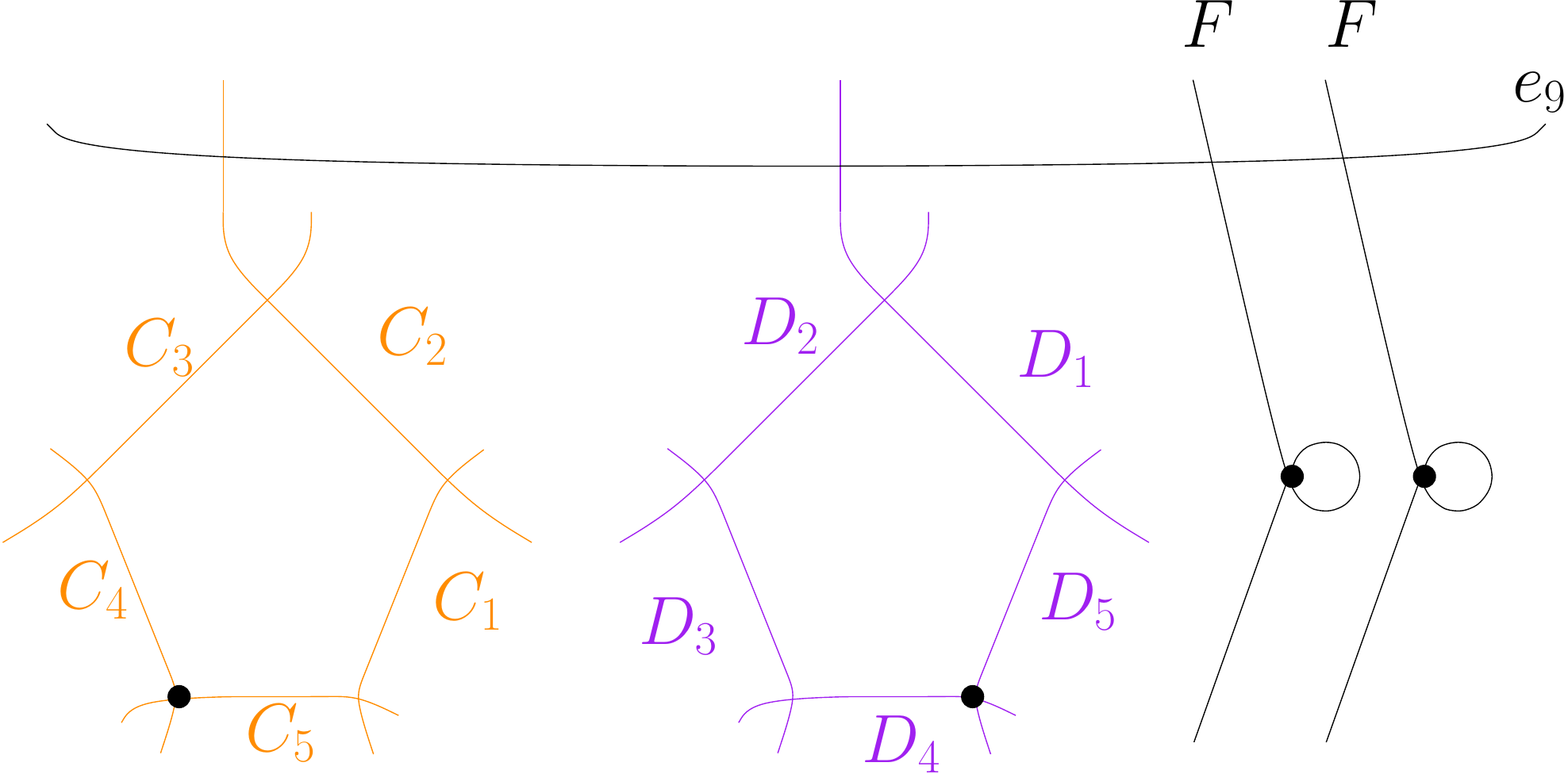}
	\caption{Elliptic fibration on E(1) with section}
	\label{fig:BlowupE_1}
\end{figure}

By blowing up at one singular point in each $I_5$ fiber, we obtain four $-3$ spheres which will make up the ends of each of the arms in the configuration $\mathcal{U}$. The central sphere of square $-5$ is obtained by taking the symplectic resolution of a section with two blown-up fishtail fibers as well as the $-2$ spheres in each $I_5$ fiber which intersect this section.

\begin{figure}[h]
	\includegraphics[width=0.40\textwidth]{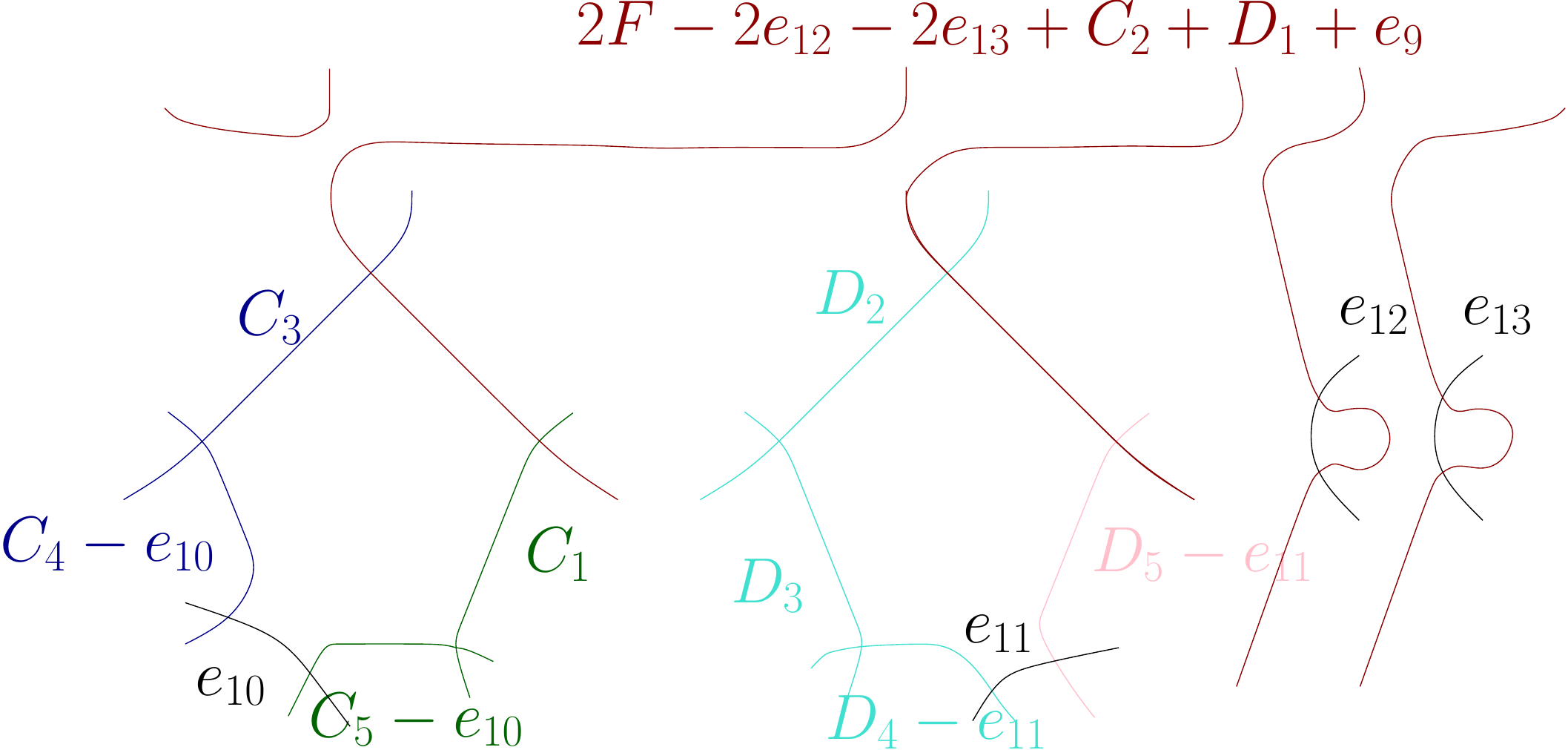}
	\caption{Embedding of $\mathcal{U}$ into $\CP2\#13\barCP2$ }
	\label{fig:BlowupE_1-2}
\end{figure}

The configuration is shown in figure \ref{fig:BlowupE_1-2}, and the homology classes are given below.

\begin{align*}
&[u_0]=2[F]-2e_{12}-2e_{13}+[C_2]+[D_1]+e_9 & & & & & &\\
&[u_{1,1}]=[D_2] &[u_{2,1}]&=[C_3] &[u_{3,1}]&= [C_1] &[u_{4,1}]=[D_5]-e_{11}&\\
&[u_{1,2}]=[D_3] &[u_{2,2}]&=[C_4]-e_{10} &[u_{3,2}]&=[C_5]-e_{10} & & \\
&[u_{1,3}]=[D_4]-e_{11} & & & & & &
\end{align*}
\end{proof}

\begin{lemma}\label{l:vector2}
The element $R\in H_2(\CP2\#13\barCP2)$ defined by \begin{align*}
R=&5656h-1728e_1-1846e_2-1836e_3-1915e_4-1905e_5-1728e_6-1600e_7\\
&-1905e_8-1890e_9-246e_{10}-295e_{11}-393e_{12}-1241e_{13}
\end{align*}
satisfies the following conditions:
\begin{enumerate}
 \item $R\cdot [u_{i,j}]=0$ for all $i$ and $j$ (i.e. $R$ is orthogonal to each embedded sphere of $\mathcal{U}$) \label{i:vector1}
 \item $R\cdot R>0$ \label{i:vector2}
 \item $R\cdot h>0$ \label{i:vector3}
 \item $R\cdot K>0$, where $K=-3h+e_1+\dots +e_{13}$ is the canonical class of $\CP2\#13\barCP2$. \label{i:vector4}
\end{enumerate}
\end{lemma}

\begin{proof}
The proof is a direct check.
\end{proof}

\begin{proposition}
The result of the $(\mathcal{U},\mathcal{V})$ star surgery on this embedding of $\mathcal{U}$ into $\CP2 \#12\barCP2$ is an exotic copy of $\CP2\#6\barCP2$ which supports a symplectic structure.
\end{proposition}

\begin{proof}
The proof is similar to the proof of Propsition~\ref{p:ex7}
\end{proof}

\subsection{Another symplectic exotic $\CP2\#8\barCP2$}

\begin{proposition}
The configuration $\mathcal{K}$ can be symplectically embedded in $\CP2\#12\barCP2$ so that the homology classes of the central sphere $C_0$ and the four spheres in the arms $C_1,C_2,C_3,C_4$ are given as follows: 
\begin{align*}
[C_0]&=2f+e_1-2e_{10}-2e_{11}-e_{12}\\
[C_1]&=h-e_1-e_2-e_3\\
[C_2]&=h-e_1-e_4-e_5\\
[C_3]&=h-e_1-e_6-e_7\\
[C_4]&=h-e_1-e_8-e_9\\
\end{align*}
where $f=3h-(e_1+\cdots +e_9)$.
\end{proposition}

\begin{proof}
This follows immediately from blowing up the embedding of $\St_2$ into $\CP2\# 11\barCP2$ given in lemma \ref{l:classes} along a single point on the section.
\end{proof}

Note the other elliptic fibrations constructed in this paper can be used to find other embeddings of $\St_2$ and $\mathcal{K}$ into blow-ups of $E(1)$. It is an interesting question to ask whether the same star surgery operation performed on different embeddings of the plumbing into the same manifold can result in non-diffeomorphic manifolds.

\begin{theorem}
The manifold resulting from star surgery on this embedding of $\mathcal{K}$ into $\CP2\# 12\barCP2$ is an exotic copy of $\CP2\#8\barCP2$.
\end{theorem}

\begin{proof}
The manifold is homeomorphic to $\CP2\#8\barCP2$ by Freedman's theorem once we show that it is simply connected and has Euler characteristic $11$ and signature $-7$. The Euler characteristic and signature computations follow from the fact that $\chi(\mathcal{L})=2$, $\sigma(\mathcal{L})=-1$, $\chi(\mathcal{K})=6$, and $\sigma(\mathcal{K})=-5$. The manifold is simply connected as in the proof of lemma \ref{l:homeo} because the generator of $\pi_1(\mathcal{L})$ is isotopic to the meridian of $C_4$ in the embedding which is homotopically trivial because $C_4$ is one sphere in an $I_3$ fiber where the other transversally intersecting spheres are not cut out in the star surgery.

The diffeomorphism type can be distinguished from $\CP2\#8\barCP2$ either by showing that the resulting symplectic manifold has Kodaira dimension two as in lemma \ref{l:kodaira}, or by showing that the canonical class and its negation are basic classes in the small perturbation chamber of the star surgered manifold as in lemma \ref{l:SWcan}. In fact the value of $K_X\cdot \omega_X$ to compute the Kodaira dimension comes out to be exactly the same value as for the star surgery using $\St_2$ and $\T_2$, and the element $V$ defining the chamber on $\CP2\# 12\barCP2$ which descends to the small perturbation chamber in the star surgered manifold can be defined identically as in lemma \ref{l:vector1}.
\end{proof}

Similar constructions could yield applications of the $(\mathcal{M},\mathcal{N})$ and $(\mathcal{O},\mathcal{P})$ star surgeries, using the embeddings yielding the exotic copies of $\CP2\#7\barCP2$ and $\CP2\#5\barCP2$ of sections \ref{s:morestar} and \ref{s:morestar2}.

\section{Infinitely Many Exotic Smooth Structures on $\CP2 \# 7 \barCP2$} \label{s:smoothexample}
{

Using a different embedding of $\St_2$, we can produce other examples of exotic 4-manifolds using star surgery. Here we find an embedding which produces an exotic  $\CP2 \# 7 \barCP2$ by  combining our star surgery operation with a knot surgery in the double node neighborhood which was introduced by Fintushel and Stern \cite{FS2}. Thanks to the knot surgery we will have a better control over the Seiberg-Witten invariants of our manifold. The price we pay is that our manifold is no longer symplectic.

\begin{proof}[Proof of Theorem \ref{theo:main2}]
We will start with the elliptic fibration with two $I_2$ fibers, and two $I_4$ fibers whose homology classes are specified in lemma \ref{l:fibration2}. The homology classes of these spheres in $\CP2\# N \barCP2$ will be crucial to our computation of the Seiberg-Witten invariant of the manifold resulting from star-surgery along this embedding.

We follow the same steps as in \cite{FS2}. Let $K_n$ denote the $n$-twist knot. Recall that $K_n$ admits a Seifert surface of genus one and its  symmetrized Alexander polynomial is given by 
$$\Delta_{K_n}(t)=nt-(2n+1)+nt^{-1}.$$
Consider the elliptic fibration described in Lemma   \ref{l:fibration2}. We pick one of the $I_2$ fibers, say the one consisting of $\widetilde{C}_1$ and $\widetilde{L}_1$, and perturb the fibration locally so the $I_2$ fiber turns into a double node neighborhood. Then we pick a regular fiber in the double node neighborhood and do knot surgery on it. The knot surgery operation does not preserve the fibration structure inside the double node neighborhood. In particular the section represented by the exceptional sphere $E_2$ does not survive after the knot surgery. On the other hand if  we choose the gluing map in the knot surgery carefully, then $E_2$ turns into a \emph{pseudo-section}, an immersed sphere with one transverse self intersection which is a section outside of the double node neighborhood.

\begin{figure}[h]
	\includegraphics[width=0.60\textwidth]{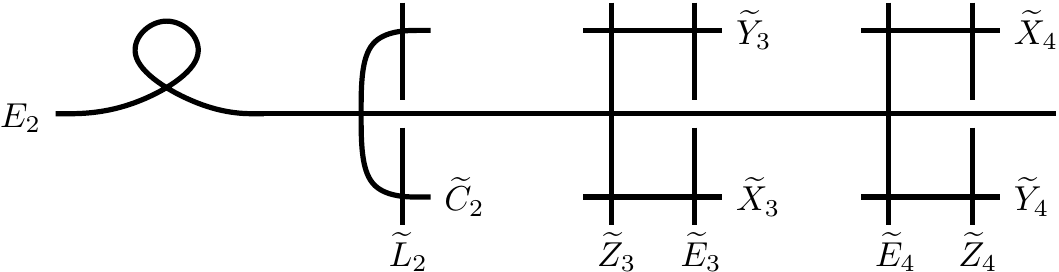}
	\caption{$E(1)_{K_n}$.}
	\label{fig:fibration2-1}
\end{figure}

Before describing our construction further, we would like to make a couple of observations about the manifold $E(1)_{K_n}$ which is the result of the knot surgery described above. First note that   $E(1)_{K_n}$ is simply connected: The fundamental group of the complement of a regular fiber is generated by a normal circle which bounds a disk in $E_2$. The second observation is that the only Seiberg-Witten basic classes of   $E(1)_{K_n}$ are $\pm PD([F])$ where $[F]$ is the fiber class, and the small perturbation Seiberg-Witten invariants at these classes are $n$, \cite{Sz, FS4}. Hence $E(1)_{K_n}$ is homeomorphic but not diffeomorphic to $E(1)=\CP2\# 9 \barCP2$.

We continue with the construction.  Blow-up $E(1)_{K_n}$ at the double point of the pseudo-section $E_2$. The proper transform of the pseudo-section is  an embedded sphere $S$ whose homology class is represented is $e_2-2e_{10}$. Now we see the configuration $\St_2$ embedded in $E(1)_{K_n}\# \barCP2$ using Lemma \ref{l:fibration2}:
\begin{align*}
[u_0]&=[S]+[\widetilde{E}_4]=e_1-2e_{10},\\
[u_1]&=[\widetilde{C}_2]=2h-e_1-e_2-e_3-e_4-e_5-e_6,\\
[u_2]&=[\widetilde{Z}_3]=h-e_1-e_2-e_9,\\
[u_3]&=[\widetilde{X}_4]=h-e_1-e_6-e_7,\\
[u_4]&=[\widetilde{Y}_4]=h-e_1-e_5-e_8.
\end{align*}

\begin{figure}[h]
	\includegraphics[width=.60\textwidth]{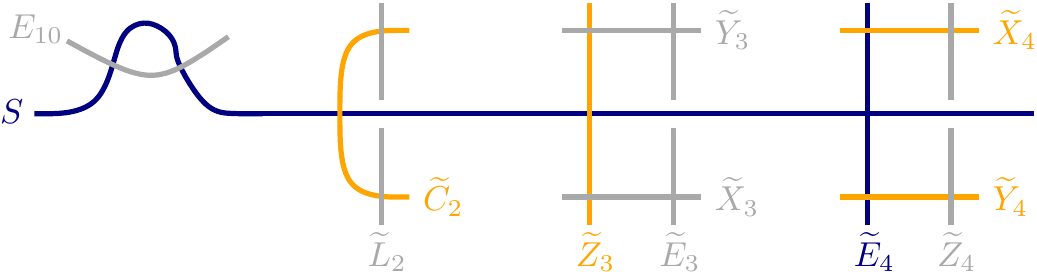}
	\caption{Configuration $\St_2$ inside $E(1)_{K_n}\# \barCP2$. }
	\label{fig:fibration2-2}
\end{figure}

Let $Y_n$ denote the result of star surgery of $E(1)_{K_n}\# \barCP2$ along the configuration $\St_2$ described above. We claim that $Y_n$ is homeomorphic to $\CP2\# 7 \barCP2$. First we must see that $Y_n$ is simply connected. By Lemma \ref{p:fund}, the generator of $\pi_1(\T_2)$ is represented by the curve in $\partial \T_2=\partial \St_2$ given by the boundary of the normal disk to any of the $-2$ spheres in $\St_2$. We will take the representative curve which bounds the normal disk to $u_2$. Since $u_2$ is one $-2$ sphere in the $I_4$ fiber, and there are other $-2$ spheres in that fiber which intersect $u_2$ transversally but are otherwise disjoint from the embedding of $\St_2$ ($\widetilde{X_3}$ or $\widetilde{E_3}$), this curve bounds a disk in $E(1)_{K_n}\# \barCP2 \setminus \St_2=Y_n\setminus \T_2$. Therefore $Y_n$ is simply connected. Next we check that $Y_n$  has the same Euler characteristic, signature, and parity of $\CP2\#7\barCP2$. A simple computation shows 
$$\chi(Y_n)=\chi(E(1)_{K_n}\# \barCP2)-\chi(\St_2)+\chi(\T_2)= 10,$$
and 
$$\sigma(Y_n)=\sigma(E(1)_{K_n}\# \barCP2)-\sigma(\St_2)+\sigma(\T_2)=-6.$$
Since $b_2(Y_n)=8$ and $b_2^+(Y_n)=1$, the intersection form cannot be a direct sum of hyperbolic pieces and $E_8$'s, so the parity of $Y_n$ is odd. Therefore $Y_n$ is homeomorphic to $\CP2\# 7 \barCP2$ by Freedman's theorem.

Finally we compute the Seiberg-Witten invariants of $Y_n$. By the blow-up formula $E(1)_{K_n}\# \barCP2$ has exactly four Seiberg-Witten basic classes $\pm PD([F]) \pm e_{10}$. The small perturbation Seiberg-Witten invariant in the chamber of $h$ evaluates as $\pm n$. We need to translate this information to a chamber whose representative homology class is orthogonal $\St_2$. Consider the following homology class
$$H:=50h-32e_1-14e_2-12e_3-21e_4-5e_5-15e_6-3e_7-12e_8-4e_9-16e_{10}.$$
It can be checked that $H\cdot H>0$, $H\cdot h >0$, and $H\cdot[u_i]=0$ for all $i=0,\dots,4$. Let $K=-PD([F]) - e_{10}=-3h+e_1+\dots+e_9-e_{10}$. We have $K\cdot H<0$ and $K\cdot h<0$. Hence there is no wall between the chambers determined by $H$ and $h$ with respect to $K$. Note that $K|_{\St_2}$ is the canonical class of $\St_2$, so  $K$ descends to $Y_n$  as a characteristic class $\widetilde{K}$ with $\widetilde{K}|_{\T_2}=0$. 

Let $X_n=E(1)_{K_n}\# \barCP2$, we will check if $d_{X_n}(K)\geq 0$ and $d_{Y_n}(\widetilde{K})\geq 0$. Clearly 
$$d_{X_n}(K)=\frac{K^2-3\sigma(X_n)-2\chi(X_n)}{4}=\frac{-1 -3(-9) - 2(13)}{4}=0.$$
On the other hand 
\begin{align*}
d_{Y_n}(\widetilde{K})&=\frac{\widetilde{K}^2-3\sigma (Y_n) - 2\chi (Y_n)}{4}\\
&=\frac{(K)^2-(K|_{\St_2})^2 + (\widetilde{K}|_{\T_2})^2-3\sigma (Y_n) - 2\chi (Y_n)}{4}\\
&=\frac{(-1)-(-3) + (0)-3(-6) - 2(10)}{4}\\
&=0.
\end{align*}
Hence by Theorem \ref{theo:mic}, we have 
$$|SW_{Y_n,H}(\pm \widetilde{K})| = |SW_{E(1)_{K_n}\# \barCP2,H}(\pm K)|=n.$$
Since the small perturbation Seiberg-Witten invariant is well-defined for those manifolds with $b_2^-\leq 9$, we conclude that   $Y_n$ has at least two basic classes. In particular $Y_n$ is not diffeomorphic to $\CP2\# 7\barCP2$ which does not have any basic classes.

It remains to prove the minimality of $Y_n$ for $n\geq 2$. By the blow-up formula, it suffices to show that there are exactly two basic classes whose Seiberg-Witten invariants are $\pm n$. We will show that $\pm \widetilde{K}$ are the only Seiberg-Witten basic classes of $Y_n$ satisfying $|SW_{Y_n}(\pm \widetilde{K})| =n$. In other words, we will prove that the cohomology class $P:=-PD([F])+e_{10}$ (the only other basic class up to sign of $X_n$), does not descend to a basic class of $Y_n$. Suppose to the contrary that there is a basic class $\widetilde{P}$ of $Y_n$ such that $\widetilde{P}|_{Y_n-\T_2}=P|_{X_n\setminus \St_2}$. Then
\begin{align*}
d_{Y_n}(\widetilde{P})&=\frac{\widetilde{P}^2-3\sigma (Y_n) - 2\chi (Y_n)}{4}\\
&=\frac{(P)^2-(P|_{\St_2})^2 + (\widetilde{P}|_{\T_2})^2-3\sigma (Y_n) - 2\chi (Y_n)}{4}\\
&=\frac{(-1)-(-1/3) + (\widetilde{P}|_{\T_2})^2-3(-6) - 2(10)}{4}\\
&=-13/6 + (\widetilde{P}|_{\T_2})^2/4 <0.
\end{align*}
The last inequality follows from the fact that the intersection form of $\T_2$ is negative definite. This contradicts with the assumption that $\widetilde{P}$ is a basic class.
\end{proof}

}

\bibliography{References}

\begin{thebibliography}{10}

\bibitem{A1}
Anar Akhmedov.
\newblock Small exotic 4-manifolds.
\newblock {\em Algebr. Geom. Topol.}, 8(3):1781--1794, 2008.

\bibitem{ABP}
Anar Akhmedov, R.~{\.I}nan{\c{c}} Baykur, and B.~Doug Park.
\newblock Constructing infinitely many smooth structures on small 4-manifolds.
\newblock {\em J. Topol.}, 1(2):409--428, 2008.

\bibitem{AP2}
Anar Akhmedov and B.~Doug Park.
\newblock Exotic smooth structures on small 4-manifolds.
\newblock {\em Invent. Math.}, 173(1):209--223, 2008.

\bibitem{AP1}
Anar Akhmedov and B.~Doug Park.
\newblock Exotic smooth structures on small 4-manifolds with odd signatures.
\newblock {\em Invent. Math.}, 181(3):577--603, 2010.

\bibitem{BK}
Scott Baldridge and Paul Kirk.
\newblock A symplectic manifold homeomorphic but not diffeomorphic to {$\Bbb
  C\Bbb P^2\#3\overline{\Bbb C\Bbb P}{}^2$}.
\newblock {\em Geom. Topol.}, 12(2):919--940, 2008.

\bibitem{BO}
Mohan Bhupal and Burak Ozbagci.
\newblock Symplectic fillings of lens spaces as lefschetz fibrations.
  arxiv:1307.6935.

\bibitem{BS}
Mohan Bhupal and Andr{\'a}s~I. Stipsicz.
\newblock Weighted homogeneous singularities and rational homology disk
  smoothings.
\newblock {\em Amer. J. Math.}, 133(5):1259--1297, 2011.

\bibitem{CGH}
Vincent Colin, Paolo Ghiggini, and Ko~Honda.
\newblock $hf = ech$ via open book decompositions: a summary. arxiv:1103.1290.

\bibitem{D}
S.~K. Donaldson.
\newblock Irrationality and the {$h$}-cobordism conjecture.
\newblock {\em J. Differential Geom.}, 26(1):141--168, 1987.

\bibitem{EG}
Hisaaki Endo and Yusuf~Z. Gurtas.
\newblock Lantern relations and rational blowdowns.
\newblock {\em Proc. Amer. Math. Soc.}, 138(3):1131--1142, 2010.

\bibitem{EMV}
Hisaaki Endo, Thomas~E. Mark, and Jeremy Van Horn-Morris.
\newblock Monodromy substitutions and rational blowdowns.
\newblock {\em J. Topol.}, 4(1):227--253, 2011.

\bibitem{EO}
John~B. Etnyre and Burak Ozbagci.
\newblock Invariants of contact structures from open books.
\newblock {\em Trans. Amer. Math. Soc.}, 360(6):3133--3151, 2008.

\bibitem{FS6}
Ronald Fintushel, B.~Doug Park, and Ronald~J. Stern.
\newblock Reverse engineering small 4-manifolds.
\newblock {\em Algebr. Geom. Topol.}, 7:2103--2116, 2007.

\bibitem{FS}
Ronald Fintushel and Ronald~J. Stern.
\newblock Rational blowdowns of smooth {$4$}-manifolds.
\newblock {\em J. Differential Geom.}, 46(2):181--235, 1997.

\bibitem{FS4}
Ronald Fintushel and Ronald~J. Stern.
\newblock Knots, links, and {$4$}-manifolds.
\newblock {\em Invent. Math.}, 134(2):363--400, 1998.

\bibitem{FS2}
Ronald Fintushel and Ronald~J. Stern.
\newblock Double node neighborhoods and families of simply connected
  4-manifolds with {$b^+=1$}.
\newblock {\em J. Amer. Math. Soc.}, 19(1):171--180 (electronic), 2006.

\bibitem{FS3}
Ronald Fintushel and Ronald~J. Stern.
\newblock Six lectures on four 4-manifolds.
\newblock In {\em Low dimensional topology}, volume~15 of {\em IAS/Park City
  Math. Ser.}, pages 265--315. Amer. Math. Soc., Providence, RI, 2009.

\bibitem{FS5}
Ronald Fintushel and Ronald~J. Stern.
\newblock Pinwheels and nullhomologous surgery on 4-manifolds with {$b^+=1$}.
\newblock {\em Algebr. Geom. Topol.}, 11(3):1649--1699, 2011.

\bibitem{Fre}
Michael~Hartley Freedman.
\newblock The topology of four-dimensional manifolds.
\newblock {\em J. Differential Geom.}, 17(3):357--453, 1982.

\bibitem{FM}
Robert Friedman and John~W. Morgan.
\newblock On the diffeomorphism types of certain algebraic surfaces. {I}.
\newblock {\em J. Differential Geom.}, 27(2):297--369, 1988.

\bibitem{GM}
David Gay and Thomas~E. Mark.
\newblock Convex plumbings and {L}efschetz fibrations.
\newblock {\em J. Symplectic Geom.}, 11(3):363--375, 2013.

\bibitem{GS1}
David~T. Gay and Andr{\'a}s~I. Stipsicz.
\newblock Symplectic rational blow-down along {S}eifert fibered 3-manifolds.
\newblock {\em Int. Math. Res. Not. IMRN}, (22):Art. ID rnm084, 20, 2007.

\bibitem{GS2}
David~T. Gay and Andr{\'a}s~I. Stipsicz.
\newblock Symplectic surgeries and normal surface singularities.
\newblock {\em Algebr. Geom. Topol.}, 9(4):2203--2223, 2009.

\bibitem{Gom}
Robert~E. Gompf.
\newblock Handlebody construction of {S}tein surfaces.
\newblock {\em Ann. of Math. (2)}, 148(2):619--693, 1998.

\bibitem{Ko}
Dieter Kotschick.
\newblock On manifolds homeomorphic to {${\bf C}{\rm P}^2\# 8\overline{{\bf
  C}{\rm P}}{}^2$}.
\newblock {\em Invent. Math.}, 95(3):591--600, 1989.

\bibitem{KLT}
{\c{C}}a{\u{g}}atay Kutluhan, Yi-Jen Lee, and Clifford~Henry Taubes.
\newblock {$HF=HM$ I} : Heegaard floer homology and seiberg--witten floer
  homology. arxiv:1007.1979.

\bibitem{LL}
T.~J. Li and A.~Liu.
\newblock Symplectic structure on ruled surfaces and a generalized adjunction
  formula.
\newblock {\em Math. Res. Lett.}, 2(4):453--471, 1995.

\bibitem{Li}
Tian-Jun Li.
\newblock The {K}odaira dimension of symplectic 4-manifolds.
\newblock In {\em Floer homology, gauge theory, and low-dimensional topology},
  volume~5 of {\em Clay Math. Proc.}, pages 249--261. Amer. Math. Soc.,
  Providence, RI, 2006.

\bibitem{MS}
Dusa McDuff and Dietmar Salamon.
\newblock {\em Introduction to symplectic topology}.
\newblock Oxford Mathematical Monographs. The Clarendon Press Oxford University
  Press, New York, 1995.
\newblock Oxford Science Publications.

\bibitem{Mic}
Maria Michalogiorgaki.
\newblock Rational blow-down along {W}ahl type plumbing trees of spheres.
\newblock {\em Algebr. Geom. Topol.}, 7:1327--1343, 2007.

\bibitem{M}
John~W. Morgan.
\newblock {\em The {S}eiberg-{W}itten equations and applications to the
  topology of smooth four-manifolds}, volume~44 of {\em Mathematical Notes}.
\newblock Princeton University Press, Princeton, NJ, 1996.

\bibitem{OS2}
Peter Ozsv{\'a}th and Zolt{\'a}n Szab{\'o}.
\newblock The symplectic {T}hom conjecture.
\newblock {\em Ann. of Math. (2)}, 151(1):93--124, 2000.

\bibitem{OS4}
Peter Ozsv{\'a}th and Zolt{\'a}n Szab{\'o}.
\newblock Holomorphic triangle invariants and the topology of symplectic
  four-manifolds.
\newblock {\em Duke Math. J.}, 121(1):1--34, 2004.

\bibitem{PS}
Heesang Park and Andr{\'a}s~I. Stipsicz.
\newblock {Smoothings of singularities and symplectic surgery}, 2012.
\newblock arXiv:1211.6830v1 [math.GT].

\bibitem{P1}
Jongil Park.
\newblock Seiberg-{W}itten invariants of generalised rational blow-downs.
\newblock {\em Bull. Austral. Math. Soc.}, 56(3):363--384, 1997.

\bibitem{P}
Jongil Park.
\newblock Simply connected symplectic 4-manifolds with {$b^+_2=1$} and
  {$c^2_1=2$}.
\newblock {\em Invent. Math.}, 159(3):657--667, 2005.

\bibitem{PSS}
Jongil Park, Andr{\'a}s~I. Stipsicz, and Zolt{\'a}n Szab{\'o}.
\newblock Exotic smooth structures on {$\Bbb{CP}^2\#5\overline{\Bbb{CP}^2}$}.
\newblock {\em Math. Res. Lett.}, 12(5-6):701--712, 2005.

\bibitem{Per}
Ulf Persson.
\newblock Configurations of {K}odaira fibers on rational elliptic surfaces.
\newblock {\em Math. Z.}, 205(1):1--47, 1990.

\bibitem{R}
Lawrence~P. Roberts.
\newblock Rational blow-downs in {H}eegaard-{F}loer homology.
\newblock {\em Commun. Contemp. Math.}, 10(4):491--522, 2008.

\bibitem{Sta2}
Laura Starkston.
\newblock Examples and properties of star surgeries. arxiv:1407.3293.

\bibitem{Sta}
Laura Starkston.
\newblock Symplectic fillings of {S}eifert fibered spaces. arxiv:1304.2420. to
  appear in trans. amer. math. soc.

\bibitem{SS}
Andr{\'a}s~I. Stipsicz and Zolt{\'a}n Szab{\'o}.
\newblock An exotic smooth structure on {$\Bbb C\Bbb P^2\#6\overline{\Bbb C\Bbb
  P^2}$}.
\newblock {\em Geom. Topol.}, 9:813--832 (electronic), 2005.

\bibitem{SSW}
Andr{\'a}s~I. Stipsicz, Zolt{\'a}n Szab{\'o}, and Jonathan Wahl.
\newblock Rational blowdowns and smoothings of surface singularities.
\newblock {\em J. Topol.}, 1(2):477--517, 2008.

\bibitem{Sym1}
Margaret Symington.
\newblock Symplectic rational blowdowns.
\newblock {\em J. Differential Geom.}, 50(3):505--518, 1998.

\bibitem{Sym2}
Margaret Symington.
\newblock Generalized symplectic rational blowdowns.
\newblock {\em Algebr. Geom. Topol.}, 1:503--518 (electronic), 2001.

\bibitem{Sz}
Zolt{\'a}n Szab{\'o}.
\newblock Exotic {$4$}-manifolds with {$b^+_2=1$}.
\newblock {\em Math. Res. Lett.}, 3(6):731--741, 1996.

\bibitem{W}
Chris Wendl.
\newblock Strongly fillable contact manifolds and {$J$}-holomorphic foliations.
\newblock {\em Duke Math. J.}, 151(3):337--384, 2010.

\end{thebibliography}
\bibliographystyle{plain}

\end{document}